\documentclass[11pt]{article}
\usepackage[latin1]{inputenc}
\usepackage{latexsym,amsmath,amsfonts}
\begin{document}

\def\N{{\Bbb N}}
\def\Z{{\Bbb Z}}
\def\R{{\Bbb R}}
\def\T{{\Bbb T}}
\def\C{{\Bbb C}}
\newcommand{\finishproof}{\hfill $\Box$ \vspace{5mm}}

\def\qedbox{$\rlap{$\sqcap$}\sqcup$}
\def\skipaline{\removelastskip\vskip12pt plus 1pt minus 1pt}

\def\Proof{\removelastskip\skipaline
\noindent \it Proof.  \rm}


\newtheorem{Theorem}{Theorem}[section]
\newtheorem{Lemma}{Lemma}[section]
\newtheorem{Proposition}{Proposition}[section]
\newtheorem{Corollary}{Corollary}[section]
\newtheorem{Remark}{Remark}[section]
\newtheorem{Example}{Example}[section]
\newtheorem{Definition}{Definition}[section]

\renewcommand{\theequation}{\thesection.\arabic{equation}}

\newcommand{\sss}{\smallskip}
\newcommand{\ms}{\medskip}
\newcommand{\bs}{\bigskip}
\newcommand{\ab}[1]{|{\mbox det}( #1 )|}
\newcommand{\sgn}{{\rm Sgn\,    }}
\newcommand{\qed}{\nolinebreak\hfill\rule{2mm}{2mm}
\par\medbreak}
\newcommand{\proof}{\par\medbreak\it Proof: \rm}
\newcommand{\rem}{\par\medbreak\it Remark: \rm}
\newcommand{\defi}{\par\medbreak\it Definition : \rm}
\newcommand{\la}{\langle }
\newcommand{\ra}{\rangle }
\newcommand{\ko}{\langle k,     \omega\rangle }
\newcommand{\gep}{\;>{\rm c}\; }
\newcommand{\kth}{e^{{\rm i}\langle k,     \theta\rangle}\; }
\newcommand{\kx}{e^{{\rm i}\langle k,     x\rangle}\; }
\newcommand{\nx}{e^{{\rm i}\langle n,     x\rangle}\; }
\newcommand{\beq}{\begin{equation} }
\newcommand{\eeq}{\end{equation} }
\newcommand{\I}{\Cal I}
\newcommand{\J}{\Cal J}


\textheight220mm \textwidth150mm \hoffset-1.2cm \voffset-1cm


\setlength{\columnsep}{5pt}

\title{Rigidity of the Reducibility of Gevrey \\
Quasi-periodic Cocycles on $U(n)$ }
\author{\\ Xuanji Hou\footnote{\footnotesize X. H. was supported by National Natural Science  Foundation of
China (Grant 11001121) and  Program for New Century Excellent Talents in University  (Grant NCET-12-0869).}
\quad  and\ \  
\  Georgi Popov\footnote{\footnotesize G.P. was partially supported by Agence Nationale de la Recherche, Projet DYNPDE: ANR-2010-BLAN-102-03 }\\
 }


\date{\today}
\maketitle

\begin{abstract}
We consider the reducibility problem of 
cocycles $(\alpha,A)$ on $\T^d\times U(n)$ in Gevrey classes, where $ \alpha$
is a Diophantine vector. We prove that, if a Gevrey  cocycle is conjugated
to a constant cocycle $(\alpha,C)$ by a suitable measurable conjugacy
$(0,B)$, then  for almost all $C$ it can be conjugated to
$(\alpha,C)$ in the same Gevrey class, provided that $A$ is
sufficiently close to a constant. If $B$  is continuous  we obtain it is  Gevrey smooth. We consider as well the global problem of  reducibility in Gevrey classes when  $d=1$.
\end{abstract}

\tableofcontents

\section{Introduction}
This article is concerned with the reducibility of cocycles in Gevrey classes on the unitary group $U(n)$. A cocycle on $U(n)$ is a diffeommorphisms of $\T^d \times U(n)$, $\T^d$ being the  torus $\T^d=\R^d/ \Z^d$,   given by the skew-product
\begin{eqnarray*}
(\alpha,A):&& \T^d \times \C^n\rightarrow \T^d \times \C^n \\
&& (\theta,v)\mapsto (\theta+\alpha, A(\theta)v),
\end{eqnarray*}
where $\alpha\in \T^d$ and $A:\T^d \to U(n)$ is a map. The corresponding dynamics is defined by the iterates of the cocycle by composition $(\alpha,A)^n$, $n\in\Z$. 
We denote by
 $C^r(\T^d,U(n))$
($r=0,1,\cdots,\infty,\omega$)  the set of all $C^r$ functions $A$. For any $\rho\geq 1$ and $L>0$ we denote  by
${\mathcal G}_L^\rho(\T^d,U(n))$ the class of Gevrey-${\mathcal G}^\rho$ functions with an exponent $\rho$ and Gevrey constant $L$. A map $A\in
C^\infty(\T^d,U(n))$ belongs to that class if it satisfies \eqref{eq:gevrey} (see Section \ref{sec:gevrey-functions}).
Denote by  $SW_\rho^{\mathcal{G}}(\T^d, U(n))$ ($SW^{r}(\T^d, U(n))$),
the set of all Gevrey-$\mathcal{G}^\rho$ ($C^r$) quasi-periodic cocycles on  $U(n)$.

The dynamics is particularly simple if $(\alpha,A)$ is a constant cocycle. 
The cocycle $(\alpha,A)$ is said to be constant
if $A$ is a constant matrix.
Two cocycles $(\alpha,A),(\alpha,\widetilde{A})\in
SW^r(\T^d, U(n))$ are  said to be  conjugated if there exists
 $B:\T^d\rightarrow U(n)$ such that
\begin{eqnarray*}
Ad(B).(\alpha,A) := (\alpha,B(\cdot+\alpha)^{-1} A
B)=(\alpha,\widetilde{A}),
\end{eqnarray*}
which means that
$B(\theta+\alpha)^{-1}A(\theta)B(\theta)=\widetilde{A}(\theta)$ for any $\theta\in \T^d$.
The cocycle  $(\alpha, A)$ is said to be
 reducible if it is conjugated to a constant one. We say also  that the conjugation or the reducibility is Gevrey-${\mathcal{G}^\rho}$,
 $C^r$, or measurable, if $B$ belongs to the corresponding class of  functions.

Reducibility problem of cocycles  has been investigated for a long time.
 The local reducibility problem (the cocycle is close to a constant
one) is usually studied using  KAM-type iterations. In particular, Eliasson's KAM method developed in \cite{E}
has been fruitfully used to obtain full-measure reducibility for
generic one-parameter families of cocycles \cite{C1, El01, Kr99a,
Kr99b, HeY, HY2}. The global reducibility problem (cocycles are no longer
 close to a constant one) has been studied by Avila,
Krikorian and others.  By means of a  renormalization scheme Krikorian obtained a global  density
result for $C^\infty$ cocycles on $SU(2)$ \cite{Kr1} and also results
for  cocycles on $SL(2,\R)$ \cite{AK,Kr2}. Almost reducibility for Gevrey cocycles has been studied by Chavaudret in \cite{C1}.\\

The \textit{rigidity problem} we are interested in, can be
formulated as follows. Suppose that
 a  Gevrey-${\mathcal{G}^\rho}$  cocycle is measurably  reducible. Is it also Gevrey-${\mathcal{G}^\rho}$ reducible?
In the case of $C^\infty$ or $C^\omega$ cocycles the rigidity problem has been investigated in
 \cite{AK, Kr2, HY1, HY2}.\\

In this paper, we will focus our attention on the Gevrey case. We will prove  a local
 rigidity result of reducibility in Gevrey classes  which can be viewed as a Gevrey analogue of  the main result in \cite{HY1}.
 To this end we use  techniques developed in \cite{P}.
  When $d=1$, the local result together with
 Krikorian's renormalization scheme   imply as in \cite{Kr1, AK} a global rigidity result    for
 Gevrey quasi-periodic cocycles on
$\T^1\times U(n)$. 

Why are we interested in Gevrey classes? 
Gevrey classes appear  naturally in the KAM theory when dealing with Diophantine frequencies \cite{P1, P}. They provide a natural framework for  studying KAM systems,   Birkhoff normal forms with an exponentially small reminder terms and the Nekhoroshev theory, and give an inside relation between these theories  \cite{M-S,M-P,P1,P}. One can consider as well the more general  Roumieu classes of non-quasi-analytic functions. 
In the  case of Bruno-R\"ussmann arithmetic conditions we suggest that similar results hold in appropriate Roumieu spaces. 

To formulate the main results we recall certain arithmetic conditions.
Given $\gamma>0$ and $\tau>d-1$, we say that $\alpha\in \R^d$ is $(\gamma,\tau)$-Diophantine if
\begin{eqnarray}\label{eq:sdc}
|e^{2\pi i\langle k, \alpha\rangle}-1 | >\frac{\gamma^{-1}}{|k|^\tau},
\quad 0\neq k\in \Z^d,
\end{eqnarray}
and we
 denote by  ${\rm DC\, }(\gamma,\tau)$ the set of all such Diophantine
vectors. Hereafter, $i:=\sqrt{-1}$ stands for the imaginary unit. It
is well known that ${\rm DC\, }(\tau):=\bigcup_{\gamma>0}{\rm DC\, }(\gamma,\tau)$ is
a set of full Lebesgue  measure. For any given $\alpha\in \R^d$, we
denote by  $\Upsilon(\alpha;\chi,\nu)$ the set of all vectors
$(\phi_1,\cdots,\phi_n)\in \R^n$, satisfying
\begin{eqnarray}\label{upsilon-def}
|\langle k,\alpha\rangle +\phi_p-\phi_q-j|\geq
\frac{\chi}{(1+|k|)^\nu}
\end{eqnarray}
for any $p\neq q\in \{1,2,\cdots,n\}$, 
 $k\in \Z^d$ and $j\in \Z$.
The set
$$
\Upsilon(\alpha):=\bigcup_{\chi,\nu>0}\Upsilon(\alpha;\chi,\nu)
$$
 has full Lebesgue measure in $\R^n$.
 Recall that the Lie group $U(n)$ consists  of all $A\in
GL(n,\C)$ satisfying $A^*A=I$. Hereafter, $I$ stands for the identity
matrix and $A^*$ is the adjoint matrix to $A$ in $M_n=M_n(\C)$.
The corresponding Lie algebra $u(n)$
 is the set of  $X\in gl(n,\C)$ satisfying $X^*+X=0$. Any
$A\in U(n)$ is diagonalizable, and the set of eigenvalues of $A$, denoted
by $ {\rm Spec\, }(A)$, is a subset of $\{z\in\C:|z|=1\}$. Denote by
$\Sigma(\alpha;\chi,\nu)$ the set of $A\in U(n)$ with spectrum
$ {\rm Spec\, }(A):=\{\lambda_1, \lambda_2,\cdots, \lambda_n\}$ satisfying
\begin{eqnarray}
|\lambda_p-\lambda_q e^{2\pi i \langle k,\alpha\rangle}|\geq
\frac{\chi}{(1+|k|)^\nu}
\end{eqnarray}
for any $p\neq q\in \{1,2,\cdots,n\}$ and $k\in \Z^d$.
 Let
$\Sigma(\alpha)=\bigcup_{\chi,\nu>0}\Sigma(\alpha;\chi,\nu)$. It is
obvious that $A\in \Sigma(\alpha)$ if and only if
$$ {\rm Spec\, }(A)=\{e^{2\pi i\varrho_1},e^{2\pi i\varrho_2},\cdots,e^{2\pi i\varrho_n}\}$$
with $(\varrho_1,\varrho_2,\cdots,\varrho_n)\in \Upsilon(\alpha)$. 

In Section \ref{sec:measurable-functions} we assign to any measurable map
$B: \T^d\to M_n$ a number $\lceil B \rfloor$  which evaluates  the distance from $B$ to the set of ``totally degenerate maps''.
A measurable maps $C:\T^d \to M_n$ will be called {\em totally degenerate} if there exist constant matrices $S,T\in U(n)$ such that the first row of the matrix $SC(\theta)T$ is zero for a.e. $\theta\in \T^d$. We say that $B:\T^d\to M_n$ is $\epsilon$-non-degenerate if $\lceil B \rfloor\ge \epsilon $.

We are going to state  the  main results of the article. 

\begin{Theorem}\label{Local} Let $\rho>1$ and $(\alpha, Ae^{G})\in SW_\rho^{\mathcal{G}}(\T^d,U(n))$, where $\alpha\in DC(\gamma,\tau)$, $A\in U(n)$ is a constant matrix and
$G\in \mathcal{G}_L^\rho(\T^d,u(n))$. Then for any $\ell>1$ there is a positive constant $\delta=\delta(d,n,\rho,L,\gamma,\tau,
\ell)$ such that for any $\epsilon \in (0,1]$ the following holds.
If the cocycle
$(\alpha, Ae^{G})$ is conjugated to a constant cocycle $(\alpha,C)$ with  $C\in \Sigma(\alpha)$
 by a measurable map $B:\T^d\rightarrow U(n)$
where  
\begin{equation}
\lceil B^*\rfloor\geq \epsilon>0 \quad \mbox{and}\quad \|G\|_L<\delta \epsilon^\ell, 
\label{eq:small-condition-epsilon}
\end{equation}
then  $(\alpha,
Ae^{G})$ can be conjugated to $(\alpha,C)$ by a Gevrey map
$\widetilde B\in \mathcal{G}^{\rho}(\T^d,
U(n))$ in the same Gevrey class. Moreover,
$\widetilde{B}(\theta)=B(\theta)$ for a.e. $\theta\in\T^d$, which
implies that $B$ is a $\mathcal{G}^{\rho}$ map if it is
continuous.
\end{Theorem}
Making use of the above local result and of  the 
renormalization we obtain a global rigidity result. The renormalization scheme we apply in this paper has been  developed by Krikorian   and  it is  often used when  studying  the global properties  of 1-dimensional quasi-periodic cocycles \cite{AK,HY2,Kr1,Kr2}. 
To formulate the the global result in the case $d=1$  we need the following arithmetic condition on $\alpha$ involving  the Gauss map $G: (0,1)\to (0,1)$, where $G(x)=\{x\}^{-1}$ and $\{x\}$ stands for the fractional part of $x$. We denote by ${\rm RDC\, }(\gamma,\tau)$ the set of all irrational $\alpha\in (0,1)$ such that  $G^m(\alpha)$ belongs to ${\rm DC\, }(\gamma,\tau)$ for infinitely many $m\in\N$.  It can be shown that ${\rm RDC\, }(\gamma,\tau)$ is of full Lebesgue  measure in $(0,1)$ as long as ${\rm DC\, }(\gamma,\tau)$ is of positive measure \cite{AK}.  We set as well ${\rm RDC\, }=\bigcup_{\gamma,\tau>0}{\rm RDC\, }(\gamma,\tau)$. 

The  global result is stated as follows.

\begin{Theorem}\label{Gloabl} For any $\alpha\in {\rm RDC\, }$,  if $(\alpha, A)\in SW_\rho^{\mathcal{G}}(\T^1,U(n))$
 is conjugated to a constant cocycle $(\alpha,C)$ with
$C\in \Sigma(\alpha)$ by a measurable $B:\T^1\rightarrow U(n)$, then
it can be conjugated to $(\alpha,C)$ by a Gevrey map $\widetilde
B\in \mathcal{G}^\rho(\T^d, U(n))$  of the same class.  Moreover,
$\widetilde{B}(\theta)=B(\theta)$ for a.e. $\theta\in\T^1$, which
implies that $B$ is Gevrey-$\mathcal{G}^\rho$ if it is continuous.
\end{Theorem}

\begin{Remark}
We remark that the proofs in this paper can also be generalized to obtain similar local and global results  for Gevrey cocycles 
on compact semisimle Lie groups.
\end{Remark}

The article is organized as follows. In Sect. 2 we give certain facts about analytic, Gevrey and measurable functions which are needed in the sequel. In particular we prove the  Approximation Lemma and the Inverse Approximation Lemma for Gevrey functions $P:\T^d\to u(n)$ of Gevrey index $\rho>1$ which gives the optimal  approximation of $P$ with analytic functions $P_j$ in the complex strips $\T^d_{h_j}$, where $h_j=h_0 \delta^j$, $j\in\N$, and $0<\delta<1$. By optimal we mean that $P_j-P_{j-1}$ is $O(\exp (-C h_m^{- 1/(\rho -1) })$ in $\T^d_{h_j}$, where $C>0$ is a constant. 
We point out that the approximation with the truncated Fourier series is not optimal. 
In Sect. 2.3 we introduce the important quantity  $\lceil B \rfloor$  giving a sort of a ``distance'' between a measurable map $B:\T^d\to M_n$   and the set of ``totally degenerate maps''.  The definition of $\lceil B \rfloor$ is  invariant with respect to the choice of the unitary bases in $\C^n$. 
We introduce  the sets  $\Gamma(N,\epsilon)$ and $\Pi(\widetilde N, \xi, \varepsilon)$ in order to keep track on the evolution of the quantity  $\lceil \cdot \rfloor$ when performing  certain  operations on $B$ such as  truncation of the Fourier series of $B$ up to order $N$ and multiplication. The set $\Pi$ obeys a simple rule under multiplication which allows one to use it successfully   in the Iterative Lemma.  

In Sect. 3 we prove the local rigidity result. First we establish the KAM Step - Proposition \ref{KAM}. It provides a conjugation of an  analytic cocycle $(\alpha, Ae^F)$ in $\T^d_h$ with sup-norm $|F|_h\le \epsilon\ll 1$  to another one  $(\alpha, A_+e^{F_+})$ with sup-norm $|F_+|_{(1-\kappa)h}\le \epsilon^{1+\sigma}$, where $0<\kappa<1$ and $0<\sigma \ll 1$ are constants and $A,A_+\in U(n)$ are constant matrices. The sup-norm of the conjugating operator $R$, however, can be estimated only by $\epsilon^{-K_*}$, where the constant $K_*\ge 1$ may be  large  due to the presence of resonances. On the other hand, it belongs to a certain $\Pi(N,1/n, \epsilon^{1-4\sigma})$, $N$ being the order of the truncated Fourier series of $F$,  which gives control on $\lceil R \rfloor$. To solve the corresponding homological equations for the non-resonant terms we use a  variant of the inverse function theorem - Lemma \ref{Basic-Lemma}. Iterating the KAM step we obtain almost reducibility with optimal estimates in Lemma \ref{anal-appro-gev}. By optimal we mean again that the small constants $\varepsilon_m$ in Lemma \ref{anal-appro-gev} are of the size of $ \exp (-C h_m^{- \frac{1}{{\rho -1}}} )$, $C>0$. In Sect. \ref{Gevrey reducibility} we prove reducibility in the Gevrey class $SW_\rho^{\mathcal{G}}(\T^d,U(n))$ provided that the cocycle is reducible by a measurable conjugation $B$ satisfying \eqref{eq:small-condition-epsilon} (see   Lemma \ref{Gev-conj} ). The  idea    (see Lemma \ref{conj-as-large}) is first  to consider  the conjugation with $B^\ast R_m$ for $m\gg 1$, where $R_m$  gives the  the conjugation   to the  cocycle $(\alpha,A_m e^{F_m})$ in  the Iterative Lemma. Using the $\epsilon$-non-degeneracy of $B^\ast$ given by \eqref{eq:small-condition-epsilon} and the relation
$R^{(m)}\in \Pi(L_m, n^{-m}, \epsilon/4n^{m} )$ in 
\eqref{eq:R} with some $L_m\in \N$,  we obtain that the eigenvalues of $A_m$ satisfy a suitable non-resonant condition. This  allows us  to estimate $R_m$ by $\varepsilon_m^{1/2}$ using  the KAM Step (Proposition \ref{KAM}, (ii)). Then the Inverse Approximation Lemma gives a conjugation in the class $\mathcal{G}^\rho(\T^d,U(n))$. We point out that there is no loss of Gevrey regularity. 

In Sect. \ref{Global} we prove Theorem \ref{Gloabl} adapting the renormalization scheme to the case of Gevrey classes.

\section{Preliminaries}
In this section we introduce the necessary tools to prove the KAM Step and the Iterative Lemma. In Sect. \ref{Sect:analytic} we recall well-known facts on the Fourier series of analytic functions $P:\T^d_{h_j}\to u(n)$. In Sect. \ref{sec:gevrey-functions} we prove the  Approximation Lemma and the Inverse Approximation Lemma for Gevrey functions $P:\T^d\to u(n)$ of Gevrey index $\rho>1$ which gives the best approximation of $P$ with analytic functions $P_j$ in the complex strips $\T^d_{h_j}$, where $h_j=h_0 \delta^j$, $j\in\N$, and $0<\delta<1$. By ``best approximation'' we mean that the sup-norm of $P_j - P_{j+1}$ is of the size of  
$ \exp (-C h_j^{- \frac{1}{{\rho -1}}} )$ in $\T^d_{h_{j+1}}$, where $C>0$. We point out the the usual approximation with the truncated Fourier series is not optimal, it gives an estimate with $ \exp (-C h_j^{- \frac{1}{{\rho}}} )$. The usual approximation with entire functions due to Moser is not optimal either. 
In Sect. \ref{sec:measurable-functions}  we introduce the important invariant  $\lceil B \rfloor$ for measurable functions $B:\T^d \to M_n$ and the sets $\Gamma$ and $\Pi$, which we need in the KAM Step and in the Iterative Lemma. 

\subsection{Analytic functions}\label{Sect:analytic}

Denote by the $M_n= M_n(\C)$ the linear space of all $n\times n$ matrices with norm $|A|\:= \sup \{\|Au\|:\ \|u\|=1\}$, where $\|\cdot\|$  is the  norm on $\C^n$ associated with the Hermitian inner product on it. 
Given $h>0$ we set
\[
\R^d_h:= \{\theta\in \C^d:\, |{\rm Im\,} \theta_j|<h, 1\le j\le
d\},\]
\[\T^d_h:= \R^d_h/\Z^d=\{\theta\in \C^d/\Z^d:\, |{\rm Im\,} \theta_j|<h,
1\le j\le d\},
\]
and  for any  analytic
function $F: \R^d_h\rightarrow M_n$ ($F: \T^d_h\rightarrow M_n$)
we define
\[
|F|_h=\sup_{|{\rm Im\, } \theta|<h}|F(\theta)|.
\]
Denote by $C_h^\omega(\T^d, M_n)$ the Banach  space of all
analytic functions $F: \T^d_h\rightarrow M_n$, equipped with the sup-norm $|\cdot|_h$.
The Fourier expansion of  $F$ is given by
\[
F(\theta)=\sum_{k\in \Z^d}\widehat{F}(k)e^{2\pi  i\langle k,\theta\rangle},
\]
and the Fourier coefficients satisfy the
estimate
\begin{equation}
|\widehat{F}(k)|\leq |F|_{h}e^{-2\pi |k|h}.
\label{eq:fourier-coef}
\end{equation}
We introduce as well the Wiener  norm
\begin{equation}
|F|_{1,h}:= \sum_{k\in\Z^n} |\widehat{F}(k)| e^{2\pi |k|h}
\label{eq:1-norm}
\end{equation}
and we denote by $\mathfrak{B}_h$  space of all $F\in C_h^\omega(\T^d, M_n)$ with bonded norm $|F|_{1,h}<\infty$. One can easily see that  $\mathfrak{B}_h$ is a Banach space and even a Banach algebra - for any $F, G\in \mathfrak{B}_h$ one has 
\begin{equation}
|F G|_{1,h} \le |F|_{1,h}|G|_{1,h}. 
\label{eq:product}
\end{equation}
Taking into account \eqref{eq:fourier-coef} we get  the following relation between the two norms
\begin{equation}\label{eq:norm-relation}
|F|_h \le |F|_{1,h}, \ |F|_{1,h_+} \le |F|_h \sum_{k\in\Z^d} e^{-2\pi |k|(h-h_+)} \le \frac{c_\ast}{(h-h_+)^d} |F|_h
\end{equation}
for any $0<h_+<h$, where $c_\ast = c_\ast(d)$ is a positive constant.

We denote by  $T_N F$ and $R_N F$ ($N\in \N$)  the truncated trigonometric polynomial of $F$ of order $N$  and the corresponding
remainder term respectively, i.e.
\[
T_N F=\sum_{|k|\leq N} \widehat{F}(k) e^{2 \pi  i\langle k,\theta\rangle}
\quad \mbox{and} \quad R_N F=\sum_{|k|> N} \widehat{F}(k) e^{2\pi  i \langle k,\theta\rangle} .
\]
One obtains as in \eqref{eq:norm-relation} the well-known estimate
\begin{eqnarray}\label{rem-fou-est}
|R_N F|_{h_+} \le |R_N F|_{1,h_+} \le \frac{c_*N^d}{(h-h_+)^d}e^{-N(h-h_+)} |F|_h
\end{eqnarray}
where  $0<h_+<h$.

For any subset $\Omega\subseteq M_n$, we denote by $C_h^\omega(\T^d, \Omega)$
the set of all $F\in C_h^\omega(\T^d, M_n)$ satisfying
$F(\T^d)\subseteq \Omega$. In particular the space  $C_h^\omega(\T^d, u(n))$ consists of all analytic functions $F:\T^d_h\to M_n$ such that $F(\theta)^\ast = - F(\theta)$ for each $\theta\in \T^d$. This
is  a
Banach subspace of $C_h^\omega(\T^d, M_n)$ and $F\in C_h^\omega(\T^d, M_n)$ is in
$C_h^\omega(\T^d,u(n))$ if and only if
\begin{eqnarray*}
\widehat{F}(k)^*=-\widehat{F}(-k).
\end{eqnarray*}

\subsection{Approximation and inverse approximation lemma for Gevrey functions}\label{sec:gevrey-functions}
Given $\rho\ge 1$, $L>0$, and  a subset  $\Omega\subseteq M_n$, we denote by $\mathcal{G}_L^\rho(\T^d, \Omega)$ the
set of all $C^\infty$ functions $P: \T^d \rightarrow \Omega$  such that
\begin{equation}
\|P\|_{L}:=\sup_{k\in \N^d}\sup_{\theta\in \T^d}(|\partial^{k}P(\theta)|L^{-|k|}k!^{-\rho})<\infty
\label{eq:gevrey}
\end{equation}
where $|k|=k_1+\cdots+k_d$ and
$k!=k_1!\cdots k_d!$ for
$k=(k_1,\cdots,k_d)\in \N^d$. Hereafter we suppose that $\Omega$ is closed in $M_n$. Then $\mathcal{G}_L^\rho(\T^d, \Omega)$ is complete. For  $\rho=1$ this space consists of analytic functions.  When $\rho>1$ the  space $\mathcal{G}_L^\rho(\T^d, \Omega)$ is not quasi-analytic, i.e. the   unique continuation rule does not hold any more and there exist functions with compact support. 
On the other hand, functions of that class can be nicely approximated by analytic functions as follows as we shall see below.
\begin{Proposition} (Approximation Lemma)\label{app-lem}
Fix $\rho >1$, $L\ge 1$, $0<\delta<1$  and set $h_j=h_0\delta^j$, $j\in\N$, where $0<h_0\le 1/(2L)$.
Then for any  $P\in \mathcal{G}_L^\rho(\T^d, u(n))$
there is a sequence $P_j\in C^{\omega}_{h_j} (\T^d, u(n))$, $j\geq 0$, such
that
\[
\displaystyle  \sup_{\theta\in\T^d}\, |P_j(\theta)-P(\theta)| \ \le
\  C_0\, L^d\, \exp \left (-(cL h_j)^{- \frac{1}{{\rho -1}}} \right)
\|P\|_L
 \]
 and
\[|P_{j+1}-P_j|_{h_{j+1}}\ \leq \ C_0 L^d e^{-(cL h_j)^{-1/(\rho-1)})}\|P\|_{L},\]
\[|P_0|_{h_0}\ \leq \ C_0 (1+L^d e^{-(cL h_0)^{-1/(\rho-1)})})\|P\|_{L},\]
where  $c=c(\rho)$ and $C_0=C_0(d, \rho)$  are positive constants depending only on $\rho$ and  on $d$ and $\rho$ respectively.
\end{Proposition}
\begin{proof}
Proposition \ref{app-lem} is a variant of Proposition 3.1 \cite{P}. The proof given bellow  is adapted to the case when  $P$ takes its values in $u(n)$  simplifying as well  some arguments of  \cite{P}.

\vspace{0.3cm}
\noindent
{\em 1.  Almost analytic extension of P}. We recall the following estimates from \cite{P}.
\begin{Lemma}\label{Lemma:Stirling}
There is  a constant $C(\rho)\ge 1$, depending only on
$\rho$,
such that for any $t\in (0,1]$ and $m\in {\N}$ satisfying
\begin{equation}
1 \le m \le t^{\, -\frac{1}{{\rho - 1}}} + 1\,  ,
                                    \label{a.1}
\end{equation}
the following inequality holds
\begin{equation}
t^{m}\,  m !^{\, \rho -1}\ \le\ C(\rho)
m^{(\rho-1)/2}e^{-(\rho-1)m}.
                                    \label{a.2}
\end{equation}
\end{Lemma}
\begin{proof}
 Stirling's formula implies
$$
t^{m}\,  m !^{\, \rho -1}\
\le\
C_1(\rho) m^{(\rho-1)/2}
e^{-(\rho-1)m}
\exp\left((\rho-1)  m\ln m + m \ln t \right)
$$
$$
=\
C_1(\rho) m^{ (\rho-1)/2}
e^{-(\rho-1)m}
\exp\left\{(\rho-1) m \ln \left(m\, t^{\frac{1}{{\rho -1}}}\right)\right\} .
$$
Moreover, \eqref{a.1} yields
\[
\displaystyle
m \ln \left(m\, t^{\frac{1}{\rho -1}}\right) \le m
\ln \left(1+  t^{\frac{1}{{\rho -1}}}\right) \le
  m\,  t^{\frac{1}{\rho -1}} \le 1 +   t^{\frac{1}{\rho
-1}} \le 2\, ,
\]
which proves (\ref{a.2}).
\end{proof}

We define an almost analytic extensions $F_j$ of $P$ in
$\T^d_{2h_j}$, $j\ge 0$,  as follows
\begin{equation}
F_j(\theta + i\widetilde \theta)\
=\ \sum_{\alpha\in
{\mathcal M}_j}\,
\partial_\theta^k
P(\theta)\frac{(i\widetilde \theta)^k }
{k !}, \quad \theta\in\T^d,\ \widetilde\theta\in\R^d.
                                    \label{eq:aae}
\end{equation}
The  index set ${\mathcal M}_j$ consists of all multi-indices
$k=(k_1,\ldots,k_d)\in \N^d$
such that $k_1\le N_j,\cdots,k_d\le N_j$,
 where
\begin{equation}
N_j =\left[(2L h_j)^{-\frac{1}{{\rho -1}}}\right]
                                    \label{a.4}
\end{equation}
and   $[x] = \inf\{k\in\Z:\, x\ge k\}$ is the integer part of $x\in\R$. Estimating \eqref{eq:aae} term by term and using \eqref{eq:gevrey} one obtains
\[
|F_j|_{2h_j}\ \le \ \|P\|_L
\sum_{k\in
{\mathcal M}_j}\,
(2L h_j)^{|k|}
k !\, ^{\rho -1}\ .
\]
Let $k=(k_1,\ldots,k_s,\ldots k_d)\in {\mathcal M}_j$ and $k_s>0$.
Then   $0<\alpha_k\le  N_j = \left[(2L h_j)^{-\frac{1}{{\rho -1}}}\right]$ and setting  $t= 2Lh_j \le 2Lh_0<1$   and $m=k_s$ in Lemma \ref{Lemma:Stirling} one obtains
\[
(2Lh_j)^{k_s}k_s !^{\, \rho-1}  \le C(\rho)  m^{(\rho-1)/2}e^{-(\rho-1)m}
\]
which implies
$$
|F_j|_{2h_j}\ \le \ \left(
1 +  C(\rho)\sum_{m=1}^\infty
m^{(1-\rho)/2}e^{-(\rho-1)m}\right)^{d}\|P\|_L \ =\
 C_1(\rho,d)\|P\|_L .
$$
On the other hand,   applying $\displaystyle \bar\partial_{s} :=
\frac{1}{2}\left(\frac{\partial}{\partial \theta_s} + i\frac{\partial}{\partial\widetilde \theta _s}\right) $, $1\le s\le d$,
to $F_j$  one  gets
\begin{equation}
2 \bar\partial_{k}F_j(\theta + i\widetilde \theta)\ =\
 \sum_{
k\in
{\mathcal M}_j^s} \,
\partial_\theta^k
\partial_{\theta_s}
P(\theta)\frac{(i\widetilde \theta)^k}
{k ! }
                                    \label{eq:bar-d-1}
\end{equation}
where ${\mathcal M}_j^s$ consists of all multi-indices $k=(k_1,\ldots,k_s,\ldots k_d)\in {\mathcal M}_j$ such that $k_s=N_j$.
Each term in the sum can be estimated in $\T^d_{2h_j}$ by
$$
L(2Lh_j)^{|k|} k ! \, ^{\rho -1}(k_s+1)^{\rho} \|P\|_L.
$$
Since
$$
 (2Lh_j)^{-\frac{1}{{\rho -1}}}  \le  k_s=N_j <
(2Lh_j)^{-\frac{1}{{\rho -1}}}  + 1,
$$
one obtains  from (\ref{a.2}) (with $t = 2L_1h_j$ and  $m = k_s =N_j$) the estimate
$$
(2Lh_j)^{k_s}\,  k_s !^{\, \rho -1} (k_s+1)^{\rho}
\le  C
(2Lh_j)^{-\frac{\rho}{{\rho -1}}-\frac{1}{2}}
\exp\left(-(\rho-1)(2Lh_j)^{-\frac{1}{{\rho -1}}}\right) .
$$
This implies as above
\begin{equation}
\begin{array}{lcrr}
|\bar\partial_{k}F_j|_{2h_j} \\  [0.3cm] 
\le \
 C_1\,  L  (Lh_j)^{-\frac{\rho}{{\rho -1}}-\frac{1}{2}}
\exp \left ( - (\rho-1)(2L h_j)^{- \frac{1}{{\rho -1}}} \right)\,
\|P\|_L \\  [0.3cm]
 \le \
C\, L\,
\exp \left ( - (cL h_j)^{-
\frac{1}{{\rho -1}}} \right)\,  \|P\|_L \, ,
\end{array}
\label{eq:estimate1}
\end{equation}
where $C_1$ and $C = C(\rho,d)>0$ are positive  constants depending only on $d$ and $\rho$ and $c=2(\frac{1}{2}(\rho-1))^{1-\rho }$. In the same way, differentiating \eqref{eq:bar-d-1} one obtains the estimate
\begin{equation}
|\bar\partial ^l F_j|_{2h_j}
 \le \
C(\rho,d)\, L^{|l|}\,
\exp \left ( -  (cL h_j)^{-
\frac{1}{{\rho -1}}} \right)\,  \|P\|_L
\label{eq:estimate-bar-d}
\end{equation}
for any $l=(l_1,\ldots,l_d)\in \N^d$ of length $|l|\ge 1$ and with components $0\le l_s\le 1$, $s\in \{1,\ldots,d\}$.
Moreover, $F_j(\theta)=P(\theta)$ for any $\theta\in\T^d$ and \eqref{eq:aae} yields
\begin{equation}
\forall \ z\in \T^d_{2h_j},\ F_j(z)^\ast = - F_j(\bar{z}).
\label{a:star}
\end{equation}
From now on we consider $F_j$ as $\Z^d$-periodic functions on $\R^d_{2h_j}$ with values in $M_n$ which means that $F_j(z+p) = F_j(z)$ for any $p\in\Z^d$.

\vspace{0.3cm}
\noindent
{\em 2.  Construction  of $P_j$}.
We are going to approximate $F_j$ by  $1$-periodic analytic  in $\R^d_{h_j}$ functions
 using   Green's formula
\begin{equation}
\frac{1}{2\pi i} \int_{\partial D}\, \frac{f(\eta)}{\eta -\zeta}\, d\eta  +
\frac{1}{2\pi i} \int\!\!\!\int_{D}\,
\frac{\bar\partial f(\eta)}{\eta -\zeta}\,
d\eta \wedge d\bar\eta
= \left\{ \begin{array}{ll} f(\zeta) & \textrm{if $\zeta\in D$}\\
0 & \textrm{if $\zeta\notin \bar D$}
\end{array} \right.
                                     \label{eq:Green}
\end{equation}
where $D \subset {\C}$ is a bounded domain symmetric with respect
to the real axis and with a piecewise smooth
boundary $\partial D$ which is positively oriented with respect to
$D$, $\bar D =  D\cup\partial D$,
and $f\in C^1(\bar D,M_n)$. Notice that
$$
F(\zeta)= \frac{1}{2\pi i} \int_{\partial D}\, \frac{f(\eta)}{\eta
-\zeta}\, d\eta
$$
is analytic in $D$ with values in $M_n$.
\begin{Lemma} \label{Lemma:star}
Suppose that $f(z)^\ast = - f(\bar{z})$ for any $z\in \partial D$. Then $F(z)^\ast = - F(\bar{z})$ for any $z\in D$.
\end{Lemma}
The proof is immediate using the symmetry of $\partial D$ with respect to the involution $z\to \bar{z}$.

Denote by $D_j\subset {\C}$ the open rectangle
$\{z\in \C:\,  |{\rm Re\, } z |<1/2,\, | {\rm Im\, } z|<2h_j \}$,
by  $\partial D_j$ its boundary which is  positively oriented
with respect to $D_j$,
and  by $\Gamma_j$ the union of the oriented segments
\[
\Gamma_j:=[-1/2 - 2i h_j, 1/2 - 2i h_j] \cup [1/2  + 2i h_j, -1/2 + 2i h_j].
\]
Given $\eta\in {\C}$,  we consider the $1$-periodic
meromorphic function
$$
\zeta\ \mapsto\   K(\eta,\zeta) := \frac{1}{\eta - \zeta} +
\sum\limits_{k=1}^{\infty} \left(\frac{1}{\eta - \zeta + k} +
\frac{1}{\eta - \zeta - k}\right).
$$
Obviously, $K(\eta,\zeta)= - K(\zeta,\eta)$ and the meromorphic function $\eta\to K(\eta,\zeta)$ is $1$-periodic for any $\zeta$ fixed.
Set $D:=\{z\in \C,\,  |{\rm Re\, } z |<1/2,\, | {\rm Im\, } z|<1/2 \}$. Writing $K=K_0 + K_1$, where
\[
K_0(\eta,\zeta):= \sum\limits_{k=-2}^{2} \frac{1}{\eta - \zeta + k}\quad \mbox{and}\quad
 K_1(\eta,\zeta):= 2\sum\limits_{k=3}^{\infty} \frac{\eta-\zeta}{(\eta - \zeta)^2 - k^2}
\]
one can find $C>0$ such that
\begin{equation}
\forall\, z\in D,\quad i\int_{D}\,
|K(\eta,z)|\,
 d\eta \wedge d\bar\eta \le C.
\label{eq:estimate2}
\end{equation}
Consider the function
\[
F_{j,1}(z) : =
\frac{1}{2\pi i} \int_{\Gamma_j}\,
F_j(\eta_1,z_2,\ldots,z_{d})K(\eta_1,z_1)\, d\eta_1 \ , \quad z \in
\R^d_{2h_j} .
\]
It is smooth  and  $\Z^d$-periodic  in the strip
$\R^d_{2h_j}$ and analytic with respect to $z_1$.
 Moreover, for any $z\in \R^d_{2h_j}$ such that  $z_1\in D_j$ we have
$$
F_{j,1}(z) =
\frac{1}{2\pi i} \int_{\partial D_j}\,
F_j(\eta_1,z_2,\ldots,z_{d})K(\eta_1,z_1)\, d\eta_1
$$
since the function under the integral is $1$-periodic with respect to
$\eta_1$. Lemma \ref{Lemma:star} implies that
$F_{j,1}(z)^\ast=-{F_{j,1}(\overline{z})}$ for any $z\in \R^d_{2h_j}$ such that  $z_1\in D_j$ and by continuity and  periodicity we get it for any $z\in \R^d_{2h_j}$.
Moreover,    (\ref{eq:Green}) yields
$$
F_{j,1}(z) = F_{j}(z) -
\frac{1}{2\pi i} \int_{D_j}\,
\bar\partial_{\eta_1} F_j(\eta_1,z_2,\ldots,z_{d})K(\eta_1,z_1)\,
d\eta_1 \wedge d\bar\eta_1 .
$$
Set $F_{j,0}(z): = F_{j}(z)$  and denote by
${\mathcal U}_{j,1}$ the set of all $(z_1,\ldots,z_d)\in \R^d_{2h_j}$ such that $|{\rm Im\, }z_1| \le
h_j$.   Using \eqref{eq:estimate1},  \eqref{eq:estimate-bar-d} and \eqref{eq:estimate2} we obtain   for any multi-index
$l = (0,l_2,\ldots, l_{d})\in   {\N}^{d}$ with
 $0 \le l_s\le 1$ for $ 2 \le s \le d$
the following estimate 
\begin{equation}
\left |\bar\partial^l
(F_{j,1} - F_{j,0})\right|_{{\mathcal U}_{j,1}}
\ \le \
 C\, L^d\,
\exp \left ( - (cL h_j)^{- \frac{1}{{\rho -1}}}
\right)\,   \|P\|_L ,
                                    \label{eq:estimate-bar-d-1}
\end{equation}
where $C = C(\rho, d) > 0$.
For $2\le s\le d$ we define by recurrence ${\mathcal U}_{j,s}$ as the set of all $(z_1,\ldots z_s, \ldots z_d)$ in   ${\mathcal U}_{j,s-1}$ such that $|{\rm Im\, }z_s| \le
h_j$ and set
\[
F_{j,s}(z):=
\frac{1}{2\pi i} \int_{\Gamma_j}\,
F_{j,s-1}(z_1,\ldots,z_{s-1}, \eta_s, z_{s+1}, \ldots, z_{d})K(\eta_s,z_s)\, d\eta_s
\]
for $z \in U_{j,s-1}$.
By construction $F_{j,s}$ is a smooth  $\Z^d$-periodic function  with values in $M_n$ and also analytic with respect to the variables $(z_1,\ldots,z_s)$.
It follows by induction that $F_{j,s}-F_{j,s-1}$ satisfies (\ref{eq:estimate-bar-d-1})
in ${\mathcal U}_{j,s}$ for any $l = (0,\ldots,0,l_{s+1},\ldots,l_{d})$ with  $0\le l_s\le 1$ and that $F_{j,s}(z)^\ast=-{F_{j,s}(\overline{z})}$ in
${\mathcal U}_{j,s}$.

Finally, the function  $P_j := F_{j,d}$  is analytic and $\Z^d$-periodic in  $\R^d_{h_j}$. Moreover, $P_{j}(z)^\ast=-{P_{j}(\overline{z})}$ in
$\R^d_{h_j}$, hence,  $P_j\in C^\omega_{h_j}(\T^d,u(n))$. Moreover,
$$
|P_j -F_{j}|_{h_j}
\ \le \
 C\, L^d\,
\exp \left ( - (cL h_j)^{- \frac{1}{{\rho -1}}}
\right)  \|P\|_L  .
$$
In particular,
$$
\begin{array}{rcl}
|P_{j+1} -P_{j}|_{h_{j+1}} \le |P_{j+1} -F_{j+1}|_{h_{j+1}} + |P_{j} -F_{j}|_{h_{j+1}} +|F_{j+1}
-F_{j}|_{h_{j+1}} \\ [0.3cm]
\le C \, L^d\,
\exp \left ( -  (cL h_j)^{- \frac{1}{{\rho -1}}}
\right)  \|P\|_L .
\end{array}
$$
Moreover,
$$
|P_j(\theta) -P(\theta)| \ \le \  C\, L^d\, \exp \left (
- (cL h_j)^{- \frac{1}{{\rho -1}}} \right)
\|P\|_L
$$
in $\R^d$, since $F_j(\theta) = P(\theta)$ for $\theta$ real.
Finally,
$$
|P_{0}|_{h_{0}} \le |F_0|_{h_{0}}  + |P_{0} -F_{0}|_{h_{0}}  \le \  C \left( 1 + L^d
\exp \left ( -(cL h_0)^{- \frac{1}{{\rho -1}}}
\right)\right) \|P\|_L .
$$
This completes the proof of the proposition. \end{proof}

Conversely, there is also the following.
\begin{Proposition} (Inverse Approximation Lemma)
\label{inv-app-lem} Let $0<\delta<1$, $0<h_0\le 1$ and  $h_j=h_0\delta^j$, $j\in\N$. Let $\Omega$ be a closed subset of $M_n$ and
$P_j\in C^{\omega}_{h_j} (\T^d, \Omega)$, $j\geq 0$, satisfy
\[
|P_{j+1}-P_{j}|_{h_{j}}\leq C_0  e^{-(L h_j)^{-1/(\rho-1)}}
\]
for any $j\ge 0$, where $C_0, L>0$. Then there is $C=C(\rho,d)\ge 1$ and $c_0=c_0(\rho,d)\ge 1$ and
$P\in {\mathcal G}_{c_0L}^\rho(\T^d,\Omega)$ such that $\lim P_j = P$ in ${\mathcal G}_{c_0L}^\rho(\T^d,\Omega)$ and
\[
\|P-P_{j}\|_{c_0L} \le  \frac{C C_0}{1-\delta} L^2e^{-\frac{1}{2}(L h_j)^{-1/(\rho-1)}}
\]
for any $j\in\N$.
\end{Proposition}

\begin{proof} For any $k\in\N^d$ and $j\ge 0$ one obtains by Cauchy
\[
\sup_{\theta\in \T^d} \left|\partial^k(P_{j+1}(\theta)-P_{j}(\theta))\right| \le  C_0  k !\,  h_j^{-|k|-1}e^{-(L h_j)^{-1/(\rho-1)}}.
\]
Using the inequality $x^m e^{-x} \le m !$ for 
\[
x= \frac{1}{2}(L h_j)^{-1/(\rho-1)}\ \mbox{and}\ m= [(\rho-1)(|\alpha|+2)] +1) ,
\] 
where $[x] = \inf\{m\in\Z:\, x\ge m\}$ stands for the integer part of $x\in\R$,  one gets the following estimate
\[
\begin{array}{lcrr}
\displaystyle \sup_{\theta\in \T^d} \left|\partial^k(P_{j+1}(\theta)-P_{j}(\theta))\right| \\[0.3cm]
\displaystyle  \le  C_0 (2^{\rho-1}L)^{|k|+2} k !\,   ([(\rho-1)(|k|+2)] +1)!\,  h_j\,  e^{-\frac{1}{2}(L h_j)^{-1/(\rho-1)}}
\end{array}
\]
for $j\gg 1$. 
On the other hand, using  the properties of the Gamma function, one obtains
\[
\begin{array}{lcrr}
([(\rho-1)(|k|+2)] +1)!\, = \Gamma([(\rho-1)(|k|+2)] +2) \le \Gamma((\rho-1)(|k|+2) +2)\\[0.3cm]
 \le c_1^{|k|+1}\Gamma(|k|+1)^{\rho-1}= c_1^{|k|+1} |k|!\, ^{\rho-1} \le c_2^{|k|+1} k!\, ^{\rho-1}
\end{array}
\]
where $c_1= c_1(\rho)\ge 1$ and  $c_2= c_2(\rho,d)\ge 1$. Setting $c_0:= 2^{\rho-1} c_2$ this  implies
\[
\|P_{j+1}-P_{j}\|_{c_0L} \le  C_0 (c_0 L)^{2}   h_j e^{-\frac{1}{2}(L h_j)^{-1/(\rho-1)}},
\]
and we get
\[
\|P_{m}-P_{j}\|_{c_0L} \le  C_0 (c_0 L)^{2}  e^{-\frac{1}{2}(L h_j)^{-1/(\rho-1)}} \frac{\delta^j}{1-\delta}
\]
for $m>j\gg 1$, hence, the sequence $P_m= P_0 + \sum_{j=1}^m (P_j
-P_{j-1})$ is Cauchy in  ${\mathcal G}_{c_0L}^\rho(\T^d,\Omega)$, which is
a complete space since $S$ is closed.  Taking the limit as
$m\to\infty$ we get $P\in {\mathcal
G}_{c_0L}^\rho(\T^d,\Omega)$ and the estimate of $P-P_j$. 
\end{proof}

\begin{Corollary}
The sequence $P_j$ in Proposition \ref{app-lem} satisfies the estimate
\[
\|P-P_{j}\|_{c_0L} \le  \frac{c c_0}{1-\delta} L^{d+2}
e^{-\frac{1}{2}(cL h_j)^{-1/(\rho-1)}}\|P\|_{L}.
\]
for some $c=c(\rho,d)>0$, $c=c(\rho)>0 $ and $c_0=c_0(\rho,d)>0$.
\end{Corollary}

\subsection{Measurable functions with values in $U(n)$. }\label{sec:measurable-functions}
In this section we introduce the important invariant  $\lceil B \rfloor$ for measurable functions $B:\T^d \to M_n$ and the sets $\Gamma$ and $\Pi$, which we need in the KAM step and in the Iterative Lemma. This sets give information on the quantity  $\lceil \cdot \rfloor$ under truncation of the Fourier series of $B$ and under multiplication. 

Consider the  Fourier expansion 
\[B(\theta)\sim\sum_{k\in\Z^d}\widehat{B}(k)e^{2\pi  i\langle k,\theta\rangle}\]
of a  measurable function $B:\T^d \to M_n$ and denote by 
$\widehat{b}_{p,q}(k)$  the $(p,q)$ entry of $\widehat{B}(k)$. Note that any measurable function  $B:\T^d\rightarrow
U(n)$ is always in $L^2(\T^d, M_n)$ since  $|B(\theta)|=1$ for any $\theta\in\T^d$ (the norm $|\cdot|$ on $M_n=M_n(\C)$ is fixed in Sect. \ref{Sect:analytic}).  
To measure the minimal size of the rows of the $n\times n$ matrix $\widetilde B$ with entries 
\begin{eqnarray}\label{b-tilde}
\widetilde b_{p,q}:= \sup_{k\in\Z^d}\, |\widehat{b}_{p,q}(k)|
\end{eqnarray}
we define
\begin{eqnarray*}
\displaystyle \lceil B \rfloor_0 : = \min_{1\leq p\leq n}\, \max_{1\leq q\leq n}\, |\widetilde{b}_{p,q}|  = \min_{1\leq p\leq n}\, \sup\, \{|\widehat{b}_{p,q}(k)|:\, k\in\Z^d,\,  1\le q\le n \} .
\end{eqnarray*}
The equality $\lceil B\rfloor_0=0$ means that there is a row of the matrix  $\widetilde {B}$ equal to $0$, or equivalently that there is a row of $B$ which is zero for a.e.  $\theta\in\T^d$, which 
implies that $\det B(\theta)=0$ for a.e.
$\theta\in\T^d$. In particular,  
\begin{equation}
B:\T^d\rightarrow U(n)\ \mbox{measurable}\ \Longrightarrow \ \lceil B \rfloor_0>0
\label{eq:unitary-non-zero}
\end{equation}
since  $|\det
B(\theta)|=1$.  
On the other hand,  $\lceil B \rfloor_0\ge \epsilon $ if and only if  for any $p\in \{1,\cdots,n\}$ there is $q\in \{1,\cdots,n\}$   and $k\in \Z^d$ such
that 
\begin{equation}
\label{eq:epsilon-non-degenerate}
|\widehat{b}_{p,q}(k)|\geq \epsilon . 
\end{equation} 
The quantity $\lceil \quad \rfloor_0$ has the following properties. 
\begin{Lemma}\label{Lemma:0-norm}
\begin{enumerate}
\item For any measurable $B:\T^d \to M_n$ and $T\in U(n)$, 
\[
\lceil B T \rfloor_0 \,  \ge \, \frac{1}{n} \lceil B  \rfloor_0 . 
\]
\item Let   $W(\theta)=\exp\left(2\pi i\, {\rm diag\, }(\langle k^{(1)},\theta \rangle ,\dots, \langle k^{(n)},\theta \rangle)\right)$, where $k^{(1)}, \ldots, k^{(n)}\in\Z^d$. Then
\[
 \forall\, S,T\in U(n), \quad \lceil SWT \rfloor_0 \geq n^{-3/2}, 
\]
\item For any constant function $B\in U(n)$, $\lceil B \rfloor_0 \geq 1/\sqrt{n}$.
\end{enumerate}
\end{Lemma}
\begin{proof}
Set $Q=BT$ and denote by $\hat Q_{p,q}(k)$, $1\le p,q\le n$,  and $\hat Q_{p}(k)$, $1\le p\le n$, the corresponding entries and rows of $\hat Q(k)$, $k\in\Z^d$. The rows  $T_q^\ast$ of $T^\ast$, $1\le q\le n$, form an orthonormal basis of $\C^n$ and we get
\[
\max_{1\le q\le n} \left| \hat Q_{p,q}(k) \right| = \max_{1\le q\le n} \left|\langle \hat B_{p}(k), T_q^\ast \rangle \right| \ge  
\frac{1}{n} \| \hat B_{p}(k) \| \ge \frac{1}{n} \max_{1\le q\le n}  \left| \hat B_{p,q}(k) \right|
\]
which proves the first part of the Lemma. 
To prove the second one, it will be enough to show that for any $S=(S_{p,q})_{1\le p,q\le n}\in U(n)$
\[\lceil SW\rfloor_0 \geq n^{-1/2}.\]
  The $(p, q)$ entry of $SW(\theta)$ is $S_{p,q}e^{2\pi i \langle k^{(q)},\theta\rangle}$.  
For any given $p$, $\sum_{q=1}^{n}|S_{p,q}|^2=1$, so there exists $q$ such that $|S_{p,q}|\geq n^{-1/2}$.  Hence, for any given $p$
 there exists $q$  such that $|S_{p,q}e^{2\pi i \langle k^{(q)}, \theta \rangle}|\geq n^{-1/2}$, which implies that $\lceil SW\rfloor_0 \geq n^{-1/2}$. 
 We have also shown in particular that $\lceil B \rfloor_0 \ge  n^{-1/2}$ for any constant $B\equiv S\in U(n)$, which is the third conclusion. 
\end{proof}

In general, the quantity $\lceil B \rfloor_0$ can not be controlled  when multiplying $B$  by a matrix $S\in U(n)$ from the left. To make it invariant with respect to the choice of the unitary bases in $\C^n$ or under  multiplication with $S,T\in U(n)$ from both left and  right, we  define
\begin{eqnarray*}
\displaystyle \lceil B \rfloor :\ =\ \inf_{S,T\in U(n)}\, \lceil SB(\cdot)T \rfloor_0  .
\end{eqnarray*} 
Thus $\lceil B \rfloor=0$ if and only if there are constant matrices $S,T\in U(n)$ such that the first row of the matrix $S\widetilde BT$ (the definition of $\widetilde B $ is given in (\ref{b-tilde})) is zero, or equivalently, the first row of $SB(\theta)T$ is zero for a.e. $\theta\in \T^d$. Such maps $B$ will be called totally degenerate. 
We say that $B:\T^d\to M_n$ is $\epsilon$-non-degenerate if $\lceil B \rfloor\ge \epsilon $. 
\begin{Lemma}\label{Lemma:main-norm} 
\begin{enumerate}
\item For any constant function $B\in U(n)$ we have  $\lceil B \rfloor=1/\sqrt{n}$.  
\item $\lceil B \rfloor >0$ for  any measurable   $B:\T^d\rightarrow
U(n)$. 
\item  Set  $W(\theta)=\exp\left(2\pi i\, {\rm diag\, }(\langle k^{(1)},\theta \rangle ,\dots, \langle k^{(n)},\theta \rangle)\right)$ where  $k^{(1)}, \ldots, k^{(n)}$ belong to $\Z^d$. Then $\lceil W \rfloor \geq n^{-3/2}$.
\end{enumerate}
\end{Lemma}  
\begin{proof}
We are going to prove {\em 2}.  Suppose that there are sequences $\{S_j\}_{j\in\N}, \{T_j\}_{j\in\N}\subset U(n)$ such that 
\[
\lim \lceil S_jBT_j \rfloor_0=0.
\]   
Let $S,T\in U(n)$ be  accumulation points of the sequences $\{S_j\}_{j\in\N}$ and $\{T_j\}_{j\in\N}$. Then $SBT:\T^d\to U(n)$ is again measurable and one can easily show that  
$\lceil SBT \rfloor_0 =0$ which leads to  a contradiction to \eqref{eq:unitary-non-zero}. The first and the third parts of the lemma  follow from Lemma \ref{Lemma:0-norm}.  \end{proof}
 
Given $N\in\N$ and $\epsilon>0$,  we  say that $B\in L^2(\T^d, M_n)$ is $(N,\epsilon)$-non-degenerate if the truncated Fourier series of $B$ up order $N$ is  $\epsilon$-non-degenerate, i.e.  
\[  
\lceil T_N B\rfloor\geq\epsilon, 
\]
where   $T_N B$ is defined in  Sect. \ref{Sect:analytic}. We denote by   $\Gamma(N,\epsilon)$ the set of $(N,\epsilon)$-non-degenerate maps $B\in L^2(\T^d, M_n)$. 
 We point out that the definition of $\Gamma(N,\epsilon)$ here is different from that in \cite{HY1} - in contrast to \cite{HY1}, the set $\Gamma(N,\epsilon)$ is invariant under the action of $U(n)$  from both left and right on the target space $M_n$. This set  has the following properties which can be easily checked as in  \cite{HY1},  Lemma 3.1.
\begin{Lemma}\label{Lemma:Gamma} 
\begin{enumerate}
\item $S\Gamma(N,\epsilon)T =\Gamma(N,\epsilon)$ for any $S,T\in U(n)$
\item $B\in
\Gamma(N,\lceil B \rfloor /2)$ for  $N$ large enough since $\lim_{N\rightarrow \infty} \lceil T_N B \rfloor = \lceil B \rfloor $
\item  $\Gamma(N,\epsilon) W \subseteq \Gamma(N+\widetilde N ,\epsilon/n)$, 
where 
\[
W(\theta):= \exp\left(2\pi i\, {\rm diag\, }(\langle k^{(1)},\theta \rangle ,\dots, \langle k^{(n)},\theta \rangle)\right), 
\] 
\[
k^{(1)},\cdots,k^{(n)}\in \Z^d \quad \mbox{and}\quad \max\{|k^{(1)}| , \cdots , |k^{(n)}|\} \le
\widetilde N
\]
\item $\Gamma(N,\epsilon)P\subseteq \Gamma(N,\epsilon -\varepsilon)$  for any measurable $P:\T^d\rightarrow U(n)$ with
\[\varepsilon=\sup_{\theta\in \T^d}|P(\theta)-I|.\]
\end{enumerate}
\end{Lemma}
\begin{proof} 
We shall sketch the proof of {\em 3}, the other items follow immediately from the proof of    Lemma 3.1 \cite{HY1}. Take $B\in \Gamma(N,\epsilon)$, $S\in U(n)$ and  set $E:= S B$. The Fourier coefficients of $EW$ and $E$ are related by the identity 
\[
\widehat{EW}_{p,q}(k) = \widehat{E}_{p,q}(k+ k^{(q)}),\quad 1\le p,q\le n,\ k\in\Z^n.
\]
In particular, for any $p\in\{1,\ldots, n\}$ fixed and $k\in \Z^n$ with $|k|\le N$ there is $q\in\{1,\ldots, n\}$ and $l\in \Z^n$ with $|l|\le N+\widetilde N$ such that $\widehat{E}_{p,q}(k) = \widehat{EW}_{p,q}(l- k^{(q)})$, hence, $\lceil T_{N+\widetilde N}EW \rfloor_0 \ge \lceil T_{N}E \rfloor_0 $. Now Lemma \ref{Lemma:0-norm} implies  that $\forall S, T\, \in U(n)$
\[
\begin{array}{lcrr}
\displaystyle \lceil S T_{N+\widetilde N} B W T \rfloor_0  \ge \frac{1}{n}\lceil S T_{N+\widetilde N} B W\rfloor_0  \\[0.3cm]
\displaystyle \ge \frac{1}{n}\lceil T_{N} S B \rfloor_0 =\frac{1}{n}\lceil S T_{N} B \rfloor_0 \ge \lceil T_{N}B \rfloor \ge \frac{\epsilon}{n}
\end{array}
\]
and we get $\lceil T_{N+\widetilde N}BW \rfloor \ge  \epsilon/n $. \end{proof}


The set $\Gamma(N,\epsilon)$ provides information of the quantity $ \lceil \cdot\rfloor$ after truncating the Fourier series of a function up to order $N$, which is needed in KAM step.  
In order to evaluate  $ \lceil \cdot\rfloor$ for the product of two functions $PB$ where $P$ is $L^2$ and $B$  in $\Gamma(N,\delta)$ (this occurs in the Iterative Lemma below),  it is convenient to introduce the following notation. For any $\widetilde N\in\N$ and $\xi,\, \varepsilon\in \R$ we denote by
$\Pi(\widetilde N, \xi, \varepsilon)$ the set of all  $P\in
L^2(\T^d, M_n)$ such that the operator of multiplication from the left by $P$ maps $\Gamma(N,\delta)$ into $\Gamma(N+\widetilde N, \xi\delta- \varepsilon)$, i.e. 
\begin{equation}
B\in \Gamma(N,\delta)\quad \Rightarrow \quad BP  \in \Gamma(N+\widetilde N, \xi\delta- \varepsilon).
\label{eq:Pi}
\end{equation}
The above relation means that $PB$ is $(N+\widetilde N, \xi\delta- \varepsilon)$-non-degenerate  if  $B$ is $(N,\delta)$-non-degenerate. 
The definition of the sets $\Gamma(N,\epsilon)$ and $\Pi(\widetilde N, \xi, \varepsilon)$ seems technical  but it turns out to be quite helpful in Sections \ref{iter} and \ref{Gevrey reducibility}. 
Using  the definition of
$\Pi$ and  Lemma \ref{Lemma:Gamma} we obtain 
\begin{Lemma}\label{Lemma:Pi} 
\begin{enumerate} 
\item $S\in \Pi(0,1, 0)$ for any $S\in U(n)$, 
\item  $S\Pi(\widetilde N, \xi, \varepsilon)T= \Pi(\widetilde N, \xi, \varepsilon)$ for any $S,T\in U(n)$,
\item  The map $\theta\to \exp\left(2\pi i\, {\rm diag\, }(\langle k^{(1)},\theta \rangle ,\dots, \langle k^{(n)},\theta \rangle)\right)$ belongs to $\Pi(\widetilde N, 1/n, 0)$  provided that
\[k^{(1)},\cdots,k^{(n)}\in \Z^d \quad and\quad \max\{|k^{(1)}| ,\cdots, |k^{(n)}|\} \le \widetilde
N,\]
\item $P\in \Pi(0,
1, \varepsilon)$  for any measurable $P:\T^d\rightarrow U(n)$ with
$\varepsilon=\sup_{\theta\in \T^d}|P(\theta)-I| ,$
\item $\Pi(\widetilde N, \xi, \varepsilon_1) \subseteq \Pi(\widetilde N, \xi, \varepsilon_2)$ if $\varepsilon_1 \ge \varepsilon_2$.
\end{enumerate}
\end{Lemma}
The set $\Pi$ behaves nicely under multiplication. It obeys the following simple rule which allows us to keep control on the quantity $ \lceil \cdot\rfloor$ in the Iterative Lemma. 
\begin{Lemma}\label{Lemma:product}
If $P_1\in\Pi(\widetilde N_1, \xi_1, \varepsilon_1)$ and $P_2\in
\Pi(\widetilde N_2, \xi_2, \varepsilon_2)$, then \[P_1P_2\in
\Pi(\widetilde N_1+\widetilde N_2, \xi_1\xi_2,
\xi_2\varepsilon_1+\varepsilon_2).\]
\end{Lemma}
\begin{proof}
For any $B\in \Gamma(N,\delta)$, we have
\[  B P_1\in \Gamma(N+\widetilde N_1, \xi_1\delta- \varepsilon_1).\]
Now for  $P_1P_2$, we have
\begin{eqnarray*}
 B(P_1 P_2)=(BP_1) P_2 &\in & \Gamma(N+\widetilde N_1+\widetilde N_2,
\xi_2(\xi_1\delta- \varepsilon_1)- \varepsilon_2) \\
&=& \Gamma(N+(\widetilde N_1+\widetilde N_2), \xi_1\xi_2\delta-
(\xi_2\varepsilon_1+\varepsilon_2)),
\end{eqnarray*}
 which implies that
\[
P_1P_2\in \Pi(\widetilde N_1+\widetilde N_2, \xi_1\xi_2,
\xi_2\varepsilon_1+\varepsilon_2). 
\] 
\end{proof}

The following  assertion gives information on the quantity $\lceil \cdot \rfloor$ for sequences of measurable functions  with values in $U(n)$ when passing to a limit.
\begin{Lemma}\label{Uniform Bound}
Let  $B_m:\T^d\rightarrow U(n)$ 
and $D_m:\T^d\rightarrow U(n)$, $m\in\N$, be two sequences of measurable functions   such that $\lceil D_m \rfloor\geq \delta>0$
and 
\begin{eqnarray*}
\lim_{m\rightarrow\infty}
\int_{\T^d}|B_m(\theta)-D_m(\theta)|d\theta=0,
\end{eqnarray*}
Then \[\underline{\lim}_{m\rightarrow \infty}\lceil B_m \rfloor\geq \delta.\]
\end{Lemma}
\begin{proof}
Fix $S,T\in U(n)$ and $\epsilon>0$. There exists $m_0>0$ such that
\begin{eqnarray*}
\int_{\T^d} |S(B_m(\theta)-D_m(\theta))T| d\theta <\epsilon
\end{eqnarray*}
as $m\ge m_0$.  
Now  for any $k\in \Z^d$
\begin{eqnarray*}
|\widehat{S(B_m-D_m)T}(k)|\leq \int_{\T}|\{S(B_m(\theta)- D_m(\theta))T\}e^{-2\pi
i \langle k, \theta\rangle}|d\theta <\epsilon,
\end{eqnarray*} 
hence,  by the definition of $\lceil \quad \rfloor_0$,
\[   \lceil S B_m T\rfloor_0 =\lceil SD_m+ S(B_m-D_m)T\rfloor _0> \lceil  S D_m T\rfloor_0-\epsilon.\]
By the definition of $\lceil \quad \rfloor$ we obtain
\[   \lceil  B_m \rfloor > \lceil  D_m \rfloor-\epsilon.\]
We then get the desired conclusion. \end{proof}

\section{Local Setting}

In this section we prove  Theorem \ref{Local} which provides  a local rigidity result of the
reducibility problem in Gevrey classes.   Firstly, we describe the 
 KAM step in the case of  analytic cocycles. Next we 
approximate a Gevrey cocycle by a sequence of analytic cocycles and
apply the KAM step.  In this way we get a sequence of analytic
cocycles  tending  to  a constant. Then we use a convergence
argument to obtain Gevrey reducibility under a suitable smallness
assumption.

\subsection{The KAM step}
The KAM scheme we are using here is close to that in \cite{HY3}.  We want to conjugate a cocycle $(\alpha, Ae^{F})$ with small $F$   to a constant one. In other words, we are looking for  a constant matrix $\widetilde{A}\in U(n)$  and a 
$u(n)-$valued function $Y$ with a small norm, such that
\[
Ad(e^Y).(\alpha, Ae^F)=(\alpha, \widetilde{A})
\]
which means that
\begin{eqnarray}\label{he}
 e^{-Y(\cdot+\alpha)}Ae^{F}e^{Y}=\widetilde{A}.
\end{eqnarray}
The corresponding (affine) linearized equation reads 
\begin{eqnarray}\label{linear-eq}
  Y-A^{-1}Y(\cdot+\alpha)A= - A^{-1}F+A^{-1} \widetilde A-I. 
\end{eqnarray}
If   the inverse of  the operator
\begin{eqnarray}\label{lheo}
{\mathcal A}: \,  &C_h^\omega(\T^d, u(n))& \longrightarrow\ C_h^\omega(\T^d, u(n))\nonumber\\
&Y& \longmapsto \ Y-A^{-1}Y(\cdot+\alpha)A
\end{eqnarray}
was  bounded then the equation \eqref{he} could have been solved  by means of the implicit function theorem. The presence  of  small divisors, however,  does not allow  doing this. Indeed, expanding $Y$ in Fourier series one immediately observes that there is a lot of resonant terms which makes it impossible to find bounded solutions of \eqref{linear-eq} in general. To overcome this obstruction,  we follow the standard approach to normal forms - keep  resonant terms and  remove  non-resonant ones at each step of the iteration. To this end we divide the initial space into two spaces, one of resonant modes and another one containing only non-resonant terms where a suitable lower bound of  the operator \eqref{lheo} can be obtained. On the other hand, the space $C_h^\omega(\T^d, u(n))$ equipped with the sup-norm is not adapted for estimating the operator \eqref{lheo} below.  For this reason we fix $0<\tilde h <h$ and consider the operator \eqref{lheo} in the Banach space 
\[
\mathfrak{B}_{\tilde h}:= \left\{ X\in  C_{\tilde h}^\omega(\T^d, u(n)):\  |X|_{1,\tilde h}<\infty \right\}  
\]
equipped with the norm $|\cdot|_{1,\tilde h}$ (see Sect. \ref{Sect:analytic}). The advantage of this norm is that it gives a lower bound of \eqref{lheo} if there is a lower bound of each of the Fourier coefficients. More precisely, given 
   $\eta\in (0,1)$ and $A\in U(n)$ we suppose that there is a decomposition
\[
\mathfrak{B}_{\tilde h}=\mathfrak{B}^{(nre)}_{\tilde h}\oplus \mathfrak{B}^{(re)}_{\tilde h}
\]
on a direct sum of two closed sub-spaces $\mathfrak{B}^{(nre)}_{\tilde h}$ and $\mathfrak{B}^{(re)}_{\tilde h}$ 
(the decomposition depends on $A$ and $\eta$) in
  such a way that for any $Y\in
\mathfrak{B}_{\tilde h}^{(nre)}$ the following relations hold
\begin{equation}\label{cohomo-eqn-bound}
\mathcal{A}(Y)(\cdot) = Y-A^{-1}Y(\cdot+\alpha)A\in
\mathfrak{B}_{\tilde h}^{(nre)}\quad \mbox{and}\quad |\mathcal{A}(Y)|_{1,\tilde h}\geq \eta
|Y|_{1,\tilde h} \, . 
\end{equation} 
Let $\Pi_{nre}$ ($\Pi_{re}$) be
the standard projection from $\mathfrak{B}_{\tilde h}$ onto
$\mathfrak{B}^{(nre)}_{\tilde h}$ ($\mathfrak{B}^{(re)}_{\tilde h}$). We call
$\mathfrak{B}_{\tilde h}^{(nre)}$ ($\mathfrak{B}_{\tilde h}^{(nre)}$) the
$\eta$-nonresonant ($\eta$-resonant) subspace. With all these 
assumptions, one can solve $(\ref{he})$ partially, which is summarized in the following Lemma.

\begin{Lemma}\label{Basic-Lemma}
There is  a universal constant $\delta_*\in (0,1)$, such that 
  for any $F\in \mathfrak{B}_{\tilde h}$ satisfying
   $|F|_{1,\tilde h}\le \delta_* \eta^2$, there exist  $Y \in \mathfrak{B}_{\tilde h}^{(nre)}$ and
 $ F^{(re)} \in
\mathfrak{B}_{\tilde h}^{(re)}$  such that
\[e^{-Y(\cdot+\alpha)}Ae^{F}e^{Y}=Ae^{F^{(re)}},\]
i.e., 
\begin{eqnarray}\label{pre-post}
Ad(e^{Y}).(\alpha, Ae^{F})=(\alpha, Ae^{F^{(re)}}),
\end{eqnarray}
 with the estimates
 \[|Y|_{1,\tilde h}\leq \frac{2 }{\eta} |F|_{1,\tilde h}, \quad |F^{(re)}|_{1,\tilde h}\leq  \mbox{cst.}  |F|_{1,\tilde h} .\]
\end{Lemma}
\begin{proof} Lemma \ref{Basic-Lemma} is a counterpart of Lemma 3.1 \cite{HY3} in the discrete case.  The Lemma follows from the implicit function theorem. Given  $F$ with $|F|_{1,\tilde h}\ll 1$ we are looking for a solution  $Y\in \mathfrak{B}_{\tilde h}^{(nre)}$ of the equation 
\begin{eqnarray}\label{eq:equation-for-Y}
\mathcal{H}(F,Y):=\Pi_{nre}\log\{A^{-1}e^{-Y(\cdot+\alpha)}Ae^{F}e^{Y}\}=0
\end{eqnarray}
where 
\[
\log(X) = \sum_{r\ge 1}\frac{1}{r}(I-X)^r, \quad \mbox{for}\quad X\in M_n,\  |I-X|<1.  
\]
We are going to solve \eqref{eq:equation-for-Y} by means of a fixed point argument for  contraction maps. 
To this end, we firstly compute  the  partial  derivative of $\mathcal{H}$ with respect to $Y$.  Taking the power series expansions of $e^{Y}$ and $e^{F}$ at $Y=F=0$ we get
\begin{eqnarray}\label{basic-form}
A^{-1}e^{-Y(\cdot+\alpha)}Ae^{F}e^{Y}
&=&  I-A^{-1}Y(\cdot+\alpha)A+Y  \nonumber\\
 &+& \displaystyle A^{-1}\left\{\sum_{r\geq 2} \frac{(-1)^r}{r!}Y(\cdot+\alpha)^r \right\}Ae^{F}e^{Y}\nonumber  \\
 &+& \displaystyle  A^{-1}e^{-Y(\cdot+\alpha)}Ae^{F}\left\{\sum_{r\geq 2} \frac{1}{r!}Y^r\right\}\nonumber \\
&+& \displaystyle A^{-1}Y(\cdot+\alpha)A\left\{\sum_{r\geq 1} \frac{1}{r!} F^r \right\}
+ \left\{\sum_{r\geq 1} \frac{1}{r!} F^r \right\} Y    .
\end{eqnarray}
Using the above formula we obtain 
\begin{eqnarray}\label{dev-F}
\mathcal{L}_{(F,Y)}Z&:=&\lim_{t \to 0}\frac{1}{t}\{\mathcal{H}(F,Y+tZ)-\mathcal{H}(F,Y)\}\nonumber\\
&=&
\Pi_{nre}\{Z-A^{-1}Z(\cdot+\alpha)A+\mathcal{E}(F,Y)Z \},
\end{eqnarray}
for any     $Z\in \mathfrak{B}_{\tilde h}^{(nre)}$,  where  $\mathcal{E}=\mathcal{E}(F,Y)$ is a linear operator in $\mathfrak{B}_{\tilde h}$ depending  on $F,Y$. Moreover, applying  \eqref{eq:product} one gets  a positive constant $\mbox{cst.}$ such that
\begin{eqnarray}
|\mathcal{E}(F,Y) Z|_{1,\tilde h}\leq \mbox{cst.}(|F|_{1,\tilde h}+|Y|_{1,\tilde h})\,  |Z|_{1,\tilde h}
\end{eqnarray}
for any $F,Y\in \mathfrak{B}_{\tilde h}$ with $|Y|_{1,\tilde h}\le 1$ and $|F|_{1,\tilde h}\le 1$. 
In particular, using the definition of the space  $\mathfrak{B}_{\tilde h}^{(nre)}$ we obtain
\begin{eqnarray}\label{dev-0}
\mathcal{L}_{(0,0)}Z&=&\Pi_{nre}\{Z-A^{-1}Z(\cdot+\alpha)A\} 
=Z-A^{-1}Z(\cdot+\alpha)A
\end{eqnarray}
as well as  the estimate
\begin{eqnarray}\label{dev-0-est}
|\mathcal{L}_{(0,0)}Z|_{1,\tilde h}= |Z-A^{-1}Z(\cdot+\alpha)A|_{1,\tilde h} \geq \eta | Z |_{1,\tilde h}.
\end{eqnarray}
Using $(\ref{dev-F})$-$(\ref{dev-0-est})$ we prove that there
 exists a  constant $\delta_*\in (0,1)$ such that 
\begin{eqnarray}\label{dev-F-est}
|(\mathcal{L}_{(F,Y)}-\mathcal{L}_{(0,0)})Z|_{1,\tilde h} 
\leq  \frac{\eta}{2} |Y|_{1,\tilde h} 
\end{eqnarray}
provided that $|F|_{1,\tilde h}<\delta_* \eta^2$ and  $|Y|_{1,\tilde h}<2\delta_* \eta$. Hence, 
\begin{eqnarray}\label{dev-est-op}
\|\mathcal{L}_{(0,0)}^{-1}\|_{1,\tilde h} \leq \frac{1}{\eta } \quad \mbox{and }\quad \|\mathcal{L}_{(F,Y)}-\mathcal{L}_{(0,0)}\|_{1,\tilde h}\leq \frac{\eta}{2}
\end{eqnarray}
for $|F|_{1,\tilde h}<\delta_* \eta^2$ and  $|Y|_{1,\tilde h}<2\delta_* \eta$, 
where  $\|\cdot\|_{1,\tilde h}$ is the operator norm corresponding to the norm  $|\cdot|_{1,\tilde h}$. 

Denote by $\mathcal{W}_{\tilde h}$ the  ball
\[
\mathcal{W}_{\tilde h}\,  :=\, \{ Y\in \mathfrak{B}_{\tilde h}^{(nre)}: |Y|_{1,\tilde h}\leq 2\delta_* \eta\}
\]
which is complete with respect to the norm $|\cdot |_{1,\tilde h}$. 
For any fixed $F\in \mathfrak{B}_{\tilde h}$ satisfying $|F|_{1,\tilde h}\leq \delta_* \eta^2$, we consider  the  map
\[
\mathcal{F} (Y):=Y-\mathcal{L}_{(0,0)}^{-1}\Pi_{nre}\log\{A^{-1}e^{-Y(\cdot+\alpha)}Ae^{F}e^{Y}\} 
\]
from $ \mathfrak{B}_{\tilde h}^{(nre)}$ to itself. 
It follows from  $(\ref{dev-est-op})$ that
\[
\|Id-\mathcal{L}_{(0,0)}^{-1}\mathcal{L}_{(F,Y)}\|_{1,\tilde h}\leq \|\mathcal{L}_{(0,0)}^{-1}\|_{1,\tilde h}
\|\mathcal{L}_{(0,0)}-\mathcal{L}_{(F,Y)}\|_{1,\tilde h}\leq \frac{1}{\eta} \cdot \frac{\eta}{2}=\frac{1}{2}
\] 
as long as $Y\in \mathcal{W}_{\tilde h} $. 
On the other hand, for any  $Y_1, Y_2\in \mathcal{W}_{\tilde h} $ there is  $\xi \in [0,1]$ such that
\[
\mathcal{F} (Y_2)-\mathcal{F} (Y_1)=\left(Id-\mathcal{L}_{(0,0)}^{-1}\mathcal{L}_{(F,Y_1+\xi(Y_2-Y_1)}\right)(Y_2-Y_1).
\]
Then  for any  $Y_1, Y_2\in \mathcal{W}_{\tilde h} $ we obtain 
\begin{eqnarray}\label{eq:estimate-F}
|\mathcal{F} (Y_2)-\mathcal{F} (Y_1)|_{1,\tilde h}\leq \frac{1}{2} |Y_2- Y_1|_{1,\tilde h}. 
\end{eqnarray} 
Set    $Y_0=0$ and define $Y_j$ inductively by  $Y_j=\mathcal{F} (Y_{j-1})$ for  $j\geq 1$. Note that
\[
Y_1=\mathcal{F} (0)= - \mathcal{L}_{(0,0)}^{-1}\Pi_{nre} F
\]
which implies 
\begin{eqnarray}\label{est-Y-01}
|Y_1-Y_0|_{1,\tilde h}=|\mathcal{F} (0)|_{1,\tilde h}\leq  \frac{1}{\eta }  |F|_{1,\tilde h}\leq \delta_* \eta
\end{eqnarray} 
in view of \eqref{dev-est-op}. 
This  means that $Y_1\in \mathcal{W}_{\tilde h}$. One can prove inductively that  $Y_j$ belongs to  $\mathcal{W}_{\tilde h}$ for  each $j\geq 1$, and  that
\begin{eqnarray}
|Y_{j}-Y_{j-1}|_{1,\tilde h}\leq \frac{1}{2^{j-1}} |Y_1-Y_0|_{1,\tilde h}.
\end{eqnarray} 
In fact, if $Y_1,\cdots, Y_{j-1}$ are all in $\mathcal{W}_{\tilde h}$ and for all $s\in\{1,\cdots, j-1\}$
\begin{eqnarray}
|Y_{s}-Y_{s-1}|_{1,\tilde h}\leq \frac{1}{2^{s-1}} |Y_1-Y_0|_{1,\tilde h},\nonumber
\end{eqnarray} 
then  (\ref{eq:estimate-F}) implies 
\begin{eqnarray}
|Y_{j}-Y_{j-1}|_{1,\tilde h}= |\mathcal{F}(Y_{j-1})-\mathcal{F}(Y_{j-1})|\leq \frac{1}{2}|Y_{j-1}-Y_{j-2}|_{1,\tilde h}\leq \frac{1}{2^{j-1}} |Y_1-Y_0|_{1,\tilde h},\nonumber
\end{eqnarray} 
and using \eqref{est-Y-01} one obtains 
\[
\begin{array}{rcll}
|Y_{j}|_{1,\tilde h}&=&|(Y_{j}-Y_{j-1})+(Y_{j-1}-Y_{j-2})+\cdots+(Y_{1}-Y_{0})|_{1,\tilde h}   \\[0.3cm]
&\leq& \displaystyle 2 |Y_1-Y_0|_{1,\tilde h}    
 \leq \frac{2}{\eta}|F|_{1,\tilde h}\leq 2\delta_* \eta.
\end{array} 
\]
Since $\mathfrak{B}_{\tilde h}^{(nre)}$ is  closed, the limit  $Y= \lim_{j\to \infty}Y_j$ exists in $\mathfrak{B}_{\tilde h}^{(nre)}$. Moreover, 
\begin{eqnarray}\label{est-Y}
|Y|_{1,\tilde h}\leq  \frac{2}{\eta}|F|_{1,\tilde h}\leq 2\delta_* \eta, 
\end{eqnarray}
which means that  $Y\in \mathcal{W}_{\tilde h}$. In particular, $Y$ satisfies the equation 
  $\mathcal{F}(Y)=Y$ and in fact $Y$  is unique in $\mathcal{W}_{\tilde h}$ in view of \eqref{eq:estimate-F}. Hence, 
\begin{eqnarray}
\Pi_{nre}\log\{A^{-1}e^{-Y(\cdot+\alpha)}Ae^{F}e^{Y}\}=0 \nonumber
\end{eqnarray} 
and setting 
\begin{eqnarray}\label{def-F-re}
F^{(re)}:=\Pi_{re} \log \{A^{-1}e^{-Y(\cdot+\alpha)}Ae^{F}e^{Y}\}
\end{eqnarray}
we  obtain $(\ref{pre-post})$. It remains to estimate $F^{(re)}$.  It follows from  (\ref{basic-form}), (\ref{est-Y}) , and from the implication 
\[Y\in \mathfrak{B}_{\tilde h}^{(nre)}\Rightarrow \mathcal{A}(Y)=Y-A^{-1}Y(\cdot+\alpha)A\in  \mathfrak{B}_{\tilde h}^{(nre)},\]
that
\begin{eqnarray}
|F^{(re)}|_{1,\tilde h}\leq \mbox{cst.}  (|Y|_{1,\tilde h}^2+ |F|_{1,\tilde h})
\end{eqnarray}
Then using $ (\ref{est-Y})$ and the assumption $|F|_{1,\tilde h}\leq \delta_* \eta^2$, we obtain 
\begin{eqnarray}
|F^{(re)}|_{1,\tilde h}\leq \mbox{cst.} \left(\frac{4}{\eta^2}|F|_{1,\tilde h}^2+|F|_{1,\tilde h}\right)\\
\leq \mbox{cst.} \left(\frac{4}{\eta^2}(\delta_* \eta^2)|F|_{1,\tilde h}+|F|_{1,\tilde h}\right) \leq \mbox{cst.}  |F|_{1,\tilde h}.\nonumber
\end{eqnarray}
This completes the proof of the Lemma. \end{proof}

The previous  Lemma will be used in the K.A.M. step. Before formulating it we recall the notion of resonant (non-resonant) pairs. A pair  $(\lambda, \widetilde{\lambda})$ of complex numbers  is said to be  $(N,\delta)$-non-resonant (with respect to $\alpha$) if 
\begin{eqnarray}\label{eq:N-delta-nonresonant}
\forall\,  k\in
\Z^d,\quad 0\neq |k|\leq N\, : \quad |e^{ i\langle k,\alpha\rangle }\widetilde{\lambda}-\lambda|\geq \delta,
\end{eqnarray}
otherwise it is said to be $(N,\delta)$-resonant. Given  $A\in U(n)$,
we write $A\in {\rm NR \, }(N,\delta)$  if  any  pair of eigenvalues of  $A$ is $(N,\delta)$ non-resonant, otherwise we write $A\in {\rm RS \, }(N,\delta)$. 
We  fix  a small constant $0<\sigma= \sigma(\ell) < 1$ by 
\begin{equation}
\sigma=\min(1/100, (\ell-1)(5\ell)^{-1})
\label{eq:sigma}
\end{equation}
where $\ell>1$ appears in \eqref{eq:small-condition-epsilon}. 

\begin{Proposition}\label{KAM}
Let $\alpha\in {\rm DC\, }(\gamma,\tau)$. For any given $\kappa\in (0,1)$, there exist   
\[
\delta_0 = \delta_0(\kappa, \gamma,\tau,d)\in(0, 1), \quad \chi=\chi(\kappa, \gamma,\tau,d)\geq 1 \quad \mbox{and} \quad K_*=K_*(\kappa)\geq 1
\] 
such that the following holds.
\begin{enumerate}
	\item[{\em (i)}]
For any $h,\varepsilon\in ( 0,1)$, $A\in U(n)$, $N\geq 1$ and $F\in
C_h^\omega(\T^d,U(n))$ satisfying
\[
0<\varepsilon< \delta_0  h^\chi , \quad N=
\left(\frac{2n+1}{\kappa}\right)^n \left(\frac{2}{ h} \log \frac{1}{\varepsilon}+1\right),\quad
{\rm and} \quad |F|_h\leq \varepsilon
\]
there is  $R\in C_{(1-\kappa)h}^\omega(\T^d,U(n))$, $A_+\in U(n)$
and $F_+\in C_{(1-\kappa)h}^\omega(\T^d,U(n))$ such that
\[
{\rm Ad\, }(R).(\alpha, Ae^{F})=(\alpha, A_+e^{F_+})
\]
where
\[
|F_{+}|_{(1-\kappa)h}\leq \varepsilon_+ :=\varepsilon^{1+\sigma}, \quad
|R|_{(1-\kappa)h}\leq \varepsilon^{-K_*},\quad \mbox{and} \quad R\in
\Pi( N, 1/n, \varepsilon^{1-4\sigma}).
\]
\item[{\em (ii)}] If $A\in {\rm NR \, }(N,
 \varepsilon^\sigma)$, then   $R=e^{Y}$  with $Y\in
 C_{h}^\omega(\T^d,U(n))$ and 
 \[
 |Y|_{h}<\varepsilon^{1-2\sigma}, \quad |A_+-A|\leq \varepsilon^{1/2}.
 \]
 \end{enumerate}
\end{Proposition}

\begin{proof} The proof of the proposition is long and we divide it in several steps.\\

\noindent
{\em Step 1. Choosing the constants.}\\

    There exists  $S\in U(n)$ such that
$SAS^*={\rm diag\, }(\lambda_1,\cdots,\lambda_n)$. Hence, without loss of generality,  we may assume that $A$   is diagonal, $A={\rm diag\, }(\lambda_1,\cdots,\lambda_n)$. 
Set 
\begin{equation}
N_j = \left(\frac{\kappa}{2n+1}\right)^{n-j} N,\quad j=0,1,\ldots,n. 
\label{eq:N}
\end{equation} 
Choosing properly the constants $\delta_0=\delta_0(\kappa)>0$  and $\chi>1$ and taking  $\varepsilon\in (0, \delta_0  h^\chi)$ we can assume that  $\varepsilon^{\sigma/2}$ is smaller than any of the finitely many universal constants arising below, and that the following estimates hold
\begin{eqnarray}\label{rem-est1}
\displaystyle e^{-N_n h}\leq \cdots \leq e^{-N_1 h} \leq e^{-N_0 h} \\[0.3cm]
\displaystyle
=\ \exp\left\{-h\left(\frac{2}{ h} \log \frac{1}{\varepsilon}+1\right)\right\}\leq  \varepsilon^2, \nonumber
\end{eqnarray}
\begin{eqnarray}\label{rem-est2}
\displaystyle \frac{N^d}{(\kappa h)^d}\ =\ \frac{1}{(\kappa h)^d}\left(\frac{2n+1}{\kappa}\right)^{nd}\left(\frac{2}{ h} 
\log \frac{1}{\varepsilon}+1\right)^d\\[0.3cm]
\displaystyle\leq \frac{3^d (2n+1)^{nd}}{\kappa^{(n+1)d}h^{2d}} \left(\log \frac{1}{\varepsilon}\right)^d
\leq \left(\frac{1}{\varepsilon}\right)^{\sigma}, \nonumber
\end{eqnarray}
\begin{eqnarray}\label{rem-est3}
&&\displaystyle 2\gamma (4n)^{\tau+1} N^{\tau} \\[0.3cm]
\displaystyle &&=\ 2\gamma (4n)^{\tau+1} \left(\frac{2n+1}{\kappa}\right)^{\tau n} \left(\frac{2}{ h} \log \frac{1}{\varepsilon}+1\right)^{\tau}
< \ \left(\frac{1}{\varepsilon}\right)^{\sigma}. \nonumber
\end{eqnarray}
For technical reasons we  assume as well that   
the inequality
\[
|X|, |Y|\le\varepsilon^{1/2}, \ X, Y\in M_n
\]
implies 
\begin{eqnarray}\label{tech-est}
\left|e^{-X}e^{X+Y} -I\right| <1, \quad  |\log\{e^{-X}e^{X+Y}\} |\le 2 (|X|^2+|Y|).
\end{eqnarray}

\noindent
{\em Step 2. The operator $\mathcal{A}$.} \\

Expanding $Y$ in  Fourier series we write the operator $\mathcal A$ in  $(\ref{lheo})$ as follows
\begin{eqnarray}\label{k-lheo}
\sum_{k\in\Z^d}\widehat{Y}(k) e^{2\pi i \langle k, \theta \rangle}\ \stackrel{\mathcal{A}}\longmapsto \ \sum_{k\in\Z^d}\{\widehat{Y}(k)-A^{-1}e^{2\pi i \langle k, \alpha \rangle}\widehat{Y}(k)A\}e^{2\pi i \langle k, \theta \rangle}.
\end{eqnarray}
Denote by $E(p,q)$  the elementary matrix with entries  $E(p,q)_{s,t}=1$ if $(s,t)=(p,q)$  and $0$ otherwise. Then $(\ref{k-lheo})$  becomes
\begin{eqnarray}\label{pq-k-lheo}
&&\sum_{k\in\Z^d}e^{2\pi i \langle k, \theta \rangle} \left\{ \sum_{1\leq p,q \leq n} \widehat{y}_{p,q}(k) E(p,q)\right\} \nonumber\\
&\stackrel{\mathcal{A}}\longmapsto &\sum_{k\in\Z^d}e^{2\pi i \langle k, \theta \rangle}\left\{  \sum_{1\leq p,q \leq n} \overline{\lambda_p}(\lambda_p-\lambda_q e^{2\pi i \langle k, \alpha \rangle})\widehat{y}_{p,q}(k)E(p,q)\right\}
\end{eqnarray}
where $\widehat{y}_{p,q}(k)$ denotes the $(p,q)$ entry of $\widehat{Y}(k)$. In this way the corresponding homological equation \eqref{linear-eq} splits into a system of equations 
\[
\overline{\lambda_p}(\lambda_p-\lambda_q e^{2\pi i \langle k, \alpha \rangle})\widehat{y}_{p,q}(k) = f_{p,q},\quad |k|\le N,\ 1\le p,q \le n. 
\]
We would like to solve it and to get ``good'' estimates for the solutions,  or equivalently to invert the operator $\mathcal A$ in a suitable space. To do this 
we have  to deal with the divisor $|\lambda_p-e^{2\pi i\langle k, \alpha \rangle}\lambda_q|$  which could be arbitrary small if  the pair $(\lambda_p, \lambda_q)$  is  $(N, \varepsilon^{\sigma})$-resonant and only $\varepsilon^{\sigma}$-small if it is $(N, \varepsilon^{\sigma})$-non-resonant (note that all $\lambda_1,\cdots,\lambda_n$ are on the unit circle of $\C$). When $p=q$, the divisor takes the form $|1-e^{2\pi i\langle k, \alpha \rangle}|$ since  $|\lambda_p|=1$, which is only $\varepsilon^{\sigma}$-small  for $|k|\le N$. Indeed \eqref{eq:sdc} and (\ref{rem-est3} ) imply that it is larger  than $2(4n)^{\tau+1}\varepsilon^{\sigma}$.  To deal with the case $p\neq q$, we recall a simple fact known as " the uniqueness of the $(2n N, 2n \varepsilon^{\sigma})-$resonance" which says that for any $\lambda_p$ and $\lambda_q$ the following relation holds
\begin{eqnarray}\label{uniq-res}
\left\{\begin{array}{ll}|\lambda_p-e^{2\pi i\langle k, \alpha \rangle}\lambda_q|<2n \varepsilon^{\sigma}\\
  |\lambda_p-e^{2\pi i\langle l, \alpha \rangle}\lambda_q|<2n \varepsilon^{\sigma}\\
 |k|,|l|\leq 2n N 
\end{array} \right\}\quad \Rightarrow  \quad  k=l.
\end{eqnarray} 
In fact, taking into account \eqref{eq:sdc}, the violation of \eqref{uniq-res} would imply  
\begin{eqnarray*}
4n \varepsilon^{\sigma}\geq |e^{2\pi i\langle k, \alpha \rangle}-e^{2\pi i\langle l, \alpha \rangle}|\geq \frac{ \gamma^{-1}}{|k-l|^{\tau}}\geq\frac{1}{\gamma(4nN)^{\tau}},
\end{eqnarray*}
 which contradicts   $(\ref{rem-est3})$.  \\
 
\noindent
{\em Step 3. Structure of the resonances.}  \\

 We are going to describe  the structure of the spectrum of $A$ dividing it into blocks   of resonant pairs of eigenvalues. 

\begin{Lemma}\label{Removing Resonances}
There exist $0\leq j\leq n-1$ and $1\leq m \leq n$, such that  ${\rm Spec\, }(A)=\{\lambda_1,\cdots,\lambda_n\}$ can be divided into $m$ subsets $\Lambda_1,\cdots,\Lambda_m$, with the properties\\
a)  If $\lambda_p$ and $\lambda_q$ belong to one and the same $\Lambda_r$  then they are $(nN_j, n\varepsilon^{\sigma})-$resonant;\\
b)  If $\lambda_p,\lambda_q$ belong to  different subsets  then they are 
$(N_{j+1}, \varepsilon^{\sigma})-$nonresonant.
\end{Lemma}

\begin{proof}
We will say that  $\lambda_p, \lambda_q$  are  $(L,a)-$connected if there exists a $(L,a)$-{\it resonant path of length $r$}
\[\lambda_{p_0}, \lambda_{p_1},\cdots,\lambda_{p_{r-1}}, \lambda_{p_r}\in \{\lambda_{1},\lambda_{2},\cdots, \lambda_{n}\},\quad \mbox{with}\quad p_0=p,\ p_r=q,  \]
such that $(\lambda_{p_0},\lambda_{p_1})$ , $(\lambda_{p_1},\lambda_{p_2})$ ,$\cdots$,  $(\lambda_{p_{r-1}},\lambda_{p_r})$  are  all $(L,a)-$resonant. One can easily check  that such a   $(L,a)-$connected pair $(\lambda_p,\lambda_q)$   is $(rL, ra)-$resonant.

Note that any $\lambda_p,\lambda_q$ in a  $(N_j, \varepsilon^{\sigma})-$connected component are $(nN_j, n\varepsilon^{\sigma})-$re-sonant. Indeed, suppose that the pair $\lambda_p, \lambda_q$ can be connected by  a $(N_j,\varepsilon^{\sigma})-$resonant path of length $r$. Without loss of generality, eliminating the "closed loops", we can assume that $r \leq n$,  hence, the pair is   $(nN_j, n\varepsilon^{\sigma})-$resonant as well.

To prove the assertion, let us firstly divide $\{\lambda_1,\cdots,\lambda_n\}$  into  $(N_0,\varepsilon^{\sigma})-$connected components. If any 
$\lambda_p$ and $\lambda_q$ belonging  to two different $(N_0,\varepsilon^{\sigma})-$connected components are  
$(N_{1}, \varepsilon^{\sigma})-$nonresonant, we finish the proof  choosing  $\Lambda_1,\cdots,\Lambda_m$ to be the $(N_0,\varepsilon^{\sigma})-$connected components.  Otherwise, we consider the $(N_1,\varepsilon^{\sigma})-$connected components and repeat the procedure. More precisely, if there is  $j$ such that any  
$\lambda_p$ and $\lambda_q$  belonging to  different   $(N_j,\varepsilon^{\sigma})$- connected components  are   $(N_{j+1}, \varepsilon^{\sigma})$-nonresonant, we  denote by $\Lambda_1,\cdots,\Lambda_m$ the corresponding $(N_j,\varepsilon^{\sigma})-$connected components and finish the proof. Otherwise, we consider the $(N_{j+1}, \varepsilon^{\sigma})$-connected components. 
Thus there are two possibilities: a) either we stop at the $j^{th}$ step and set $m=j$; b) or the number of the  $(N_{j+1},\varepsilon^{\sigma})$-connected components is strictly less than the number of the  $(N_{j},\varepsilon^{\sigma})$-connected components. In the latter case,  there is  $j\leq n-1$ such that at $j^{th}-$step there is only one $(N_{j}, \varepsilon^{\sigma})$-connected component $\{\lambda_1,\cdots,\lambda_n\}$ and we take $m=1$ and $\Lambda_1=\{\lambda_1,\cdots,\lambda_n\}$. \end{proof}

From now on we fix  $j\le n-1$ as in Lemma \ref{Removing Resonances}.  Taking account of the structure of the resonances we are going to define the spaces $\mathfrak{B}^{(nre)}_h$ and  $\mathfrak{B}^{(re)}_h$ and verify the hypothesis of Lemma \ref{Basic-Lemma}.  To this end we assign to any $p\in \{1,\cdots, n\}$ an integer vector $k^{(p)}\in \Z^d$ as follows. First  for any $1\le t\le m$ we 
 choose a representative $\lambda_{p_t}\in \Lambda_t$ and set $k^{(p_t)}=0$. Let $p\in \{1,\cdots, n\}$ and $p\notin \{p_1,\cdots, p_m\}$.  
There exits $t\in\{1,\cdots,m\}$ such that $\lambda_p\in \Lambda_t$ and by  Lemma \ref{Removing Resonances}, {\em a)},  and (\ref{uniq-res}) there is a {\em unique}  $k^{(p)}\in \Z^d$ such  that 
\[
|k^{(p)}|\leq n N_j  \quad \mbox{and} \quad |\lambda_{p}- \lambda_{p_t}e^{2\pi i\langle k^{(p)}, \alpha \rangle }|<n\varepsilon^{\sigma}.
\]  
Let $\lambda_p$ and $\lambda_q$ belong to one and the same component $\Lambda_t$. Then
\[ 
\max\left\{|\lambda_{p}- \lambda_{p_t} e^{2\pi i\langle k^{(p)}, \alpha \rangle }|,\,|\lambda_{q}- 
\lambda_{p_t} e^{2\pi i\langle k^{(q)}, \alpha \rangle }|\right\} < n\varepsilon^{\sigma}
\]
which implies 
\[
|\lambda_{p}- \lambda_{q}e^{2 \pi i\langle k^{(p)}-k^{(q)}, \omega \rangle }|<2n\varepsilon^{\sigma}  .
\]  
By the uniqueness (\ref{uniq-res}) of the $(2n N, 2n \varepsilon^{\sigma})-$resonance, the integer vector $k^{(p)}-k^{(q)}$ can be characterized as the unique  $k\in \Z^d$ satisfying the inequalities
\begin{equation}\label{eq:k}
|k|\leq 2n N \quad \mbox{and} \quad |\lambda_p- \lambda_q e^{2\pi i\langle k, \alpha \rangle }|<2 n\varepsilon^{\sigma}.
\end{equation}
This implies  the relation 
\begin{equation}\label{eq:non-resonant-k}
\lambda_p, \lambda_q\in \Lambda_t,\
k\neq k^{(p)}-k^{(q)},\  |k|\leq 2n N \quad \Longrightarrow \quad 
|\lambda_p- \lambda_q e^{2\pi i\langle k, \alpha \rangle }|\geq 2 n\varepsilon^{\sigma}.
\end{equation}
The existence of $k\in \Z^d$ satisfying \eqref{eq:k} follows from   Lemma $\ref{Removing Resonances}$, {\em a)}. Indeed,    the eigenvalues  $\lambda_p, \lambda_q\in \Lambda_t$ are $(n N_j, n \varepsilon^{\sigma})$-resonant, which means there is $k^{(p,q)}\in \Z^d$ satisfying 
\[
|k^{(p,q)}|\leq n N_j\leq n N\quad \mbox{and} \quad |\lambda_p- \lambda_q e^{2\pi i\langle k, \alpha \rangle }|<n\varepsilon^{\sigma}
\]
and by $(\ref{eq:non-resonant-k})$ we get $k^{(p,q)}=k^{(p)}-k^{(q)}$.  In particular,  $|k^{(p)}-k^{(q)}|=|k^{(p,q)}|\leq n N_j$, and we obtain the relation 
\begin{equation}\label{eq:non-resonant-k-plus-1}
\lambda_p, \lambda_q\in \Lambda_t,\
N_j<|k|\leq 2n N \quad \Longrightarrow \quad 
|\lambda_p- \lambda_q e^{2\pi i\langle k, \alpha \rangle }|\geq 2 n\varepsilon^{\sigma}.
\end{equation}
Moreover,  $(\ref{eq:non-resonant-k-plus-1})$ and  Lemma $\ref{Removing Resonances}$, {\em b)} imply
\begin{equation}\label{eq:non-resonant-k-plus-2}
 \lambda_p, \lambda_q\in  \mbox{Spec}\, (A),  \quad 
nN_j< |k|\leq N_{j+1} \quad \Longrightarrow \quad 
|\lambda_p- \lambda_q e^{2\pi i\langle k, \alpha \rangle }|\geq \varepsilon^{\sigma}.
\end{equation}

\noindent
{\em Step 4. The spaces $\mathfrak{B}^{(nre)}_{\tilde h}$ and $\mathfrak{B}^{(re)}_{\tilde h}$.} \\

Let us go back to  the expansion $(\ref{pq-k-lheo})$.  Denote by $\mathcal{Z} _k$,  $k\in \Z^d$,  the following subset of  $\{1,\cdots,n\}^2 $
\begin{equation}\label{def-Z-k}
\mathcal{Z} _k:=\{(p,q): \ \lambda_p,\lambda_q\in \Lambda_t \hbox{\quad for some $1\leq t \leq m$ and $k=k^{(p)}-k^{(q)}$} \}
\end{equation}
and by $\mathcal{Z} _k^c:= \{1,\cdots,n\}^2 - \mathcal{Z} _k$ the complement of $\mathcal{Z} _k$. 
 Note  that  
\[
(p,q)\in \mathcal{Z}_k\Leftrightarrow (q,p)\in \mathcal{Z}_{-k} \quad \mbox{and}  \quad   (p,q)\in \mathcal{Z}_{0}\Rightarrow p=q. 
\]
Recall that  for any $(p,q)$ there is at most one $k$  such that $(p,q)\in \mathcal{Z}_k$, which implies that there is no more than $n^2$ non-empty $\mathcal{Z}_k$.

Set $\tilde h = (1-\kappa/2)h$ and consider the space $\mathfrak{B}_{\tilde h}$ consisting of all $X\in C^{\omega}_{\tilde h}(\T^d,u(n))$ of finite norm $|X|_{1,\tilde h}<\infty$.  We define  $\mathfrak{B}^{(re)}_{\tilde h}$ as the space of all $X\in \mathfrak{B}_{\tilde h}$ such that 
\begin{equation}\label{re-form}
\left(T_{N_{j+1}}X\right)(\theta)\ = \ \sum_{|k|\leq n N_j}e^{2 \pi  i\langle k,\theta\rangle}\left\{\sum_{(p,q)\in \mathcal{Z}_k} \widehat{x}_{p,q}(k)  E(p,q)\right\}
\end{equation}
 and denote by $\mathfrak{B}^{(nre)}_{\tilde h}$  the space of all $X\in \mathfrak{B}_{\tilde h}$ of the form
\begin{eqnarray}\label{nre-form}
X(\theta)=\sum_{ |k|\leq N_{j+1}}e^{2 \pi  i\langle k,\theta\rangle}\{\sum_{(p,q)\in \mathcal{Z}_k^c } \widehat{x}_{p,q}(k)   E(p,q)\}\\[0.3cm]
=\sum_{ |k|\leq n N_{j}}e^{2 \pi  i\langle k,\theta\rangle} \sum_{(p,q)\in \mathcal{Z}_k^c } \widehat{x}_{p,q}(k)   E(p,q)
+ \sum_{  n N_{j}<|k|\leq N_{j+1}} \widehat{X}(k) e^{2 \pi  i\langle k,\theta\rangle}\nonumber
\end{eqnarray}
(recall that the truncation $T_{N}F$ of the Fourier series of $F$ is defined in Sect. 2.1 and that $\widehat{f}_{p,q}(k)$ is the $(p,q)$ entry of $\widehat{F}(k)$).
Both spaces are closed in $\mathfrak{B}_{\tilde h}$  and obviously
\[
\mathfrak{B}^{(re)}_{\tilde h}=\mathfrak{B}^{(nre)}_{\tilde h}\oplus \mathfrak{B}^{(re)}_{\tilde h} .
\]
Moreover, the following assertion holds true. 
\begin{Lemma}\label{Hypothesis-Basic-Lemma} Let $X\in \mathfrak{B}^{(nre)}_{\tilde h} $. Then 
 \[
X-A^{-1}X(\cdot+\alpha)A\in \mathfrak{B}^{(nre)}_{1,\tilde h} \quad  \mbox{and} 
\quad |X-A^{-1}X(\cdot+\alpha)A|_{1,\tilde h} \ge \varepsilon^{\sigma}|X|_{1,\tilde h}.  
\] 
\end{Lemma}
\begin{proof} The first relation is evident. To estimate below the first sum in \eqref{nre-form} we use \eqref{eq:non-resonant-k} and for the second sum in  \eqref{nre-form}  we make use of \eqref{eq:non-resonant-k-plus-2}. 
\end{proof}

\noindent
{\em Step 5. Applying Lemma \ref{Basic-Lemma}}. \\

The previous lemma says that \eqref{cohomo-eqn-bound} holds with   
$\eta=\varepsilon^{\sigma}$. By assumption and by \eqref{eq:norm-relation} and  \eqref{rem-est2} we get as well  
\[
|F|_{1,\tilde h}\le c_\ast(\kappa h)^{-d}|F|_{h}\le c_\ast(\kappa h)^{-d}\varepsilon < \varepsilon^{1-3\sigma/2}
\]
 and applying  Lemma \ref{Basic-Lemma} we find  $Y\in \mathfrak{B}^{(nre)}_{\tilde h}$ and $\widetilde{F}\in \mathfrak{B}^{(re)}_{\tilde h}$  such that
\begin{equation}\label{yf-est}
\left\{
\begin{array}{lcrr}
|Y|_{\tilde h}\leq |Y|_{1,\tilde h}\leq 2\varepsilon^{-\sigma}|F|_{1,\tilde h}\le 2c_\ast(\kappa h)^{-d}\varepsilon^{1-\sigma} \le \varepsilon^{1-2\sigma}, \\ [0.3cm]
|\widetilde{F}|_{\tilde h} \leq |\widetilde{F}|_{1,\tilde h} \leq cst. |F|_{1,\tilde h} \leq \varepsilon^{1-2\sigma}
\end{array}
\right.
\end{equation}
and 
\[
{\rm Ad\, }(e^{Y}).(\alpha, Ae^{F})=(\alpha, Ae^{\widetilde{F}}).
\]
Setting  $Q(\theta)={\rm diag\, }(e^{2\pi i\langle k^{(1)},\theta\rangle},\cdots,e^{2\pi i\langle k^{(n)},\theta\rangle})$ we consider 
\begin{align*}
Ad(Q).(\alpha, Ae^{\widetilde{F}})=(\alpha, Q(\cdot+\alpha)^{-1}AQ\exp\{Q^{-1} \widetilde{F} Q \}). 
\end{align*}
Since $\widetilde{F}\in \mathfrak{B}^{(re)}_{\tilde h}$  is of  the form $(\ref{re-form})$ with $f$ replaced by $\widetilde f$ we get
\begin{eqnarray*}
Q(\theta)^{-1} \widetilde{F} (\theta)Q(\theta)&=&\sum_{|k|\leq n N_j}\sum_{(p,q)\in \mathcal{Z}_k} e^{2\pi i \langle k+k^{(q)}-k^{(p)}, \theta \rangle}\widehat{\widetilde{f}}_{p,q}(k)  E(p,q)\\[0.3cm]
&+& Q(\theta) ^{-1}  (R_{N_{j+1}}\widetilde{F} ) (\theta)Q(\theta) = ( \uppercase\expandafter{\romannumeral 1 })+ (\uppercase\expandafter{\romannumeral 2} ) .
\end{eqnarray*}
It follows from the definition of  $\mathcal{Z}_k$ that $(\uppercase\expandafter{\romannumeral 1 })$ is a constant and by \eqref{eq:norm-relation} and \eqref{rem-est2} one gets
\begin{eqnarray*}
|(\uppercase\expandafter{\romannumeral 1 })|
&\leq &\sum_{|k|\leq n N_j}|\sum_{(p,q)\in \mathcal{Z}_k} \widehat{\widetilde{f}}_{p,q}(k)  E(p,q)|\\
&\leq  & n^2 |\widetilde{F}|_{1,\tilde h} \leq    \varepsilon^{1-3\sigma}.
\end{eqnarray*}
On the other hand, 
\begin{eqnarray}
|Q|_{h}
\leq e^{n N_j h}\leq  e^{\kappa N h}\leq \frac{1}{2}\varepsilon^{-K_*},
\end{eqnarray}
where $K_*=K_*(\kappa)>0$ is a constant depending only on $\kappa$. In view of (\ref{rem-fou-est}) and \eqref{eq:N}-\eqref{rem-est2} this  implies 
\begin{eqnarray*}
|(\uppercase\expandafter{\romannumeral 2 })|_{(1-\kappa)h}&\le & e^{2n N_j h} |R_{N_{j+1}}\widetilde{F}|_{(1-\kappa)h} \\
&\le&   e^{2n N_j h} \left\{c_*
     \frac{N_{j+1}^d }{(\kappa h)^d} e^{-\kappa N_{j+1} h} \varepsilon^{1-2\sigma} \right\} \\
& = &
e^{- N_j h} \left\{c_*
     \frac{N_{j+1}^d }{(\kappa h)^d}  \varepsilon^{1-2\sigma} \right\}\\
&\leq& 
\varepsilon^{2}  \left\{c_*
     \frac{N_{j+1}^d }{(\kappa h)^d}  \varepsilon^{1-2\sigma} \right\}
\leq  \varepsilon^{1+2\sigma}.
\end{eqnarray*}
Finally, we obtain  the conjugation
\begin{eqnarray}\label{pre-conj}
Ad(Q e^{Y}).(\alpha, A e^{F})=(\alpha,\widetilde{A}e^{(\uppercase\expandafter{\romannumeral 1 })+(\uppercase\expandafter{\romannumeral 2})} ),
\end{eqnarray}
where  $\widetilde{A}:=Q(\cdot+\alpha)^{-1} AQ$ is a constant. Moreover, there exists $F_+\in C^{\omega}_{(1-\kappa)h}(\T^d, U(n))$ (one can take 
$F_+=\log\{e^{-(\uppercase\expandafter{\romannumeral 1 })}e^{(\uppercase\expandafter{\romannumeral 1 })+(\uppercase\expandafter{\romannumeral 2})}\}$) 
such that 
\begin{eqnarray*}
\widetilde{A}e^{(\uppercase\expandafter{\romannumeral 1 })+(\uppercase\expandafter{\romannumeral 2})}
=\widetilde{A}e^{(\uppercase\expandafter{\romannumeral 1 })}e^{F_+}
\end{eqnarray*}
and using (\ref{tech-est}) we get
\[|F_+|_{(1-\kappa)h}\leq 2 \left(|(\uppercase\expandafter{\romannumeral 1 })|^2+|(\uppercase\expandafter{\romannumeral 2})|_{(1-\kappa)h}\right) \le 
4\varepsilon^{1+2\sigma} \leq \varepsilon^{1+\sigma}.\]
Recall that $(\uppercase\expandafter{\romannumeral 1 })$ is a constant. Then $A_+=\widetilde{A}e^{(\uppercase\expandafter{\romannumeral 1 })}\in U(n)$ is a constant matrix which satisfies 
\begin{eqnarray}\label{a-est}
|A_+-\widetilde{A}|\leq 2 |(\uppercase\expandafter{\romannumeral 1 })| \leq 2 \varepsilon^{1-3\sigma} \leq \varepsilon^{1/2} 
\end{eqnarray}
in view of (\ref{tech-est}). 
Setting $R=e^{Y} Q$, one obtains from  $(\ref{pre-conj})$ the equality
\begin{eqnarray}\label{last-conj}
Ad(R).(\alpha, A e^{F})=(\alpha, A_+e^{F_+} ).
\end{eqnarray}
One  can check easily the desired estimates.  

The relation  $R=e^{Y} Q\in
\Pi( N, 1/n, \varepsilon^{1-4\sigma})$ follows from Lemma \ref{Lemma:product}, since $e^{Y} \in
\Pi( 0, 1, \varepsilon^{1-4\sigma})$ by Lemma \ref{Lemma:Pi}, {\em 4)},  and $Q\in
\Pi( N, 1/n, 0)$  by Lemma \ref{Lemma:Pi}, {\em 3)}.

If $A\in {\rm NR \, }(N,\varepsilon^\sigma)$ using $(\ref{yf-est})$ and $(\ref{a-est})$ ) we get  $Q=I$, $R=e^{Y}$ and  $A=\widetilde{A}$ with  $|Y|_{h}<\varepsilon^{1-2\sigma}$ and $|A_+-A|\leq \varepsilon^{1/2}$.
This completes the proof of the proposition. \end{proof}

\begin{Remark}
This proof can extended  for cocycles on other groups, for example, $GL(n,\C)$, $GL(n,\R)$, $SO(n,\R)$, as well as  for general compact semi-simple Lie groups.
\end{Remark}
For any $\kappa>0$, repeating the  KAM step   infinitely many  times and choosing $h_m=(1-\frac{\kappa}{2^m})h$ and  $\varepsilon_m=\varepsilon^{(1+\sigma)^m}$ and the corresponding $N_m$ as above, we are going to  obtain  almost reducibility  with analytic radius of conjugations decreasing to $(1-2\kappa)h$ in the next section. The sequence $\varepsilon_m$ will be well-adapted to the corresponding Gevrey class choosing $\varepsilon =\varepsilon _0$ as in  \eqref{eps-cond1}.  Almost reducibility for Gevrey cocycles  has been  proved  by   Chavaudret in \cite{C1}. We point out that the KAM Step and the  KAM iteration scheme below are somewhat  different from that  in \cite{C1}.  

\subsection{The iterative Lemma}\label{iter}

We assume that $A\in U(n)$ and $G\in \mathcal{G}_L^\rho(\T^d, u(n))$, where $\rho>1$. Let $\sigma$ , $\chi$, $\delta_0$ and $K_*$ be the same  as in the assumptions of Proposition \ref{KAM}. 
Let  $0<h_0\leq 1/2L$ ($h_0$ will be specified later) and set
\begin{eqnarray}
&&\kappa =1-\left(1+\frac{\sigma}{2}\right)^{1-\rho},\ \mbox{and} \nonumber\\
&& h_m:=(1-\kappa)h_{m-1}=(1-\kappa)^mh_0=h_0 \left(1+\frac {\sigma}{2}\right)^{(1-\rho)m} ,\, m\in\N.
\end{eqnarray}
By Proposition \ref{app-lem} one can find a
sequence of $G_m\in C^\omega_{h_m}(\T^d, u(n))$ satisfying
\begin{equation}\label{G}
\left\{ \begin{array}{l}
|G_0|_{h_0}\leq c_0 L^d (1+e^{-(cL
h_0)^{-1/(\rho-1)})})\|G\|_{L}\\[0.3cm]
|G_m-G_{m-1}|_{h_{m}}\leq c_0 L^d e^{-(cL
h_{m-1})^{-1/(\rho-1)})}\|G\|_{L}, \ m\geq 1 \\[0.3cm]
\displaystyle\sup_{\theta\in \T^d}|G(\theta)-G_{m-1}(\theta)|\leq c_0 L^d
e^{-(cL h_m)^{-1/(\rho-1)})}\|G\|_{L}, \  m\geq 1, 
\end{array} \right.
\end{equation}
where
 $c=c(\rho)>0$ and $c_0=c_0(\rho,d)\ge 1$. We set
\begin{equation}\label{eps-cond1}
\left\{ \begin{array}{l}
\varepsilon_0:=\exp\left(-\frac{1}{16(1+K_*/\sigma)}(cLh_0)^{-\frac{1}{\rho-1}}\right),\ \mbox{and} \\[0.3cm]
\varepsilon_m:=\varepsilon_0^{(1+\frac{\sigma}{2})^m}=\exp\left(-\frac{1}{16(1+K_*/\sigma)}(cLh_m)^{-\frac{1}{\rho-1}}\right),
\quad m\in\N.
\end{array} \right.
\end{equation} 
Choose $0<h_0\ll 1$ small enough depending only on $\sigma$, $\chi$, $L$, $\rho$, and $\delta_0$ such that
\begin{eqnarray}\label{small-cond}
cLh_0 \le 1, \quad \varepsilon_0^{\sigma/2} \le \frac{1}{8},
\end{eqnarray}
and
\begin{equation}\label{KAM-cond}
\varepsilon_m=\varepsilon_0^{(1+\frac{\sigma}{2})^m}=\exp\left(-\frac{1}{16(1+K_*/\sigma)}(cLh_m)^{-\frac{1}{\rho-1}}\right)\leq \delta_0 h_m^\chi
\end{equation}
for any $ ,\quad m\in\N$. 
We suppose as well that for any matrix $P$
 satisfying $|P-I|\leq \varepsilon_m \le \varepsilon_0$ the following inequality holds true
\begin{eqnarray}\label{log-cond}
 |\log P|\leq 2|P-I| \leq 2 \varepsilon_m. 
\end{eqnarray}
It is enough to get the above inequality for $m=0$ then they follow automatically for each $m\in\N$. 
\\
 
We are going to impose a smallness  condition  on $G$. 
\begin{Lemma}\label{Lemma:epsilon}
Assume that
\begin{eqnarray}\label{GEV-cond}
\|G\|_{L}\le \frac{\varepsilon_0}{2c_0L^d}\, .
\end{eqnarray}
Then
\begin{equation}
\begin{array}{lcrr}
\left\{ \begin{array}{l}|G_0|_{h_0}\leq \varepsilon_0\\[0.3cm]
|G_m-G_{m-1}|_{h_{m}}\leq \varepsilon_{m+1}^{4(1+K_*/\sigma)}, \quad m\geq 1, \\[0.3cm]
\sup_{\theta\in \T^d}|G(\theta)-G_{m-1}(\theta)|_{h_m}\leq
\varepsilon_{m+1}^{4(1+K_*/\sigma)}.
\end{array} \right.
\end{array}
\label{eq:estimates-G}
\end{equation}
\end{Lemma}

\begin{proof}
The claim follows directly from \eqref{G},  \eqref{eps-cond1} and \eqref{GEV-cond} since
\begin{equation}\label{tech-cond}
\begin{array}{lcrr}
\displaystyle e^{-(cL h_{m-1})^{-1/(\rho-1)}}=e^{-(1+\frac{\sigma}{2})^{-2}(cLh_{m+1})^{-\frac{1}{\rho-1}}}\\[0.3cm] 
\displaystyle < e^{-\frac{1}{4}(cLh_{m+1})^{-\frac{1}{\rho-1}}} = 
\varepsilon_{m+1} ^{4(1+K_*/\sigma)}
\end{array}
\end{equation}
for $m\ge 1$. \end{proof}

From now on we denote by $\left(N_m\right)_{m\in\N}$   the increasing  sequence
\begin{equation}
\begin{array}{lcrr}
\displaystyle N_m:=\left(\frac{2n+1}{\kappa}\right)^n\left(\frac{2}{ h_m} \log \frac{1}{\varepsilon_m}+1\right)\\[0.3cm]
\displaystyle =\left(\frac{2n+1}{1-(1+\frac{\sigma}{2})^{1-\rho}}\right)^n\left(\frac{2}{ h_m} \log \frac{1}{\varepsilon_m}+1\right) .
\end{array}
\label{eq:N-m}
\end{equation}
In view of  $(\ref{KAM-cond})$,  one can apply Proposition \ref{KAM}
for  $\varepsilon=\varepsilon_m$, $h=h_m$ and
$N=N_m$, $m\in\N$.  Set
\begin{equation}
L_m=N_0+\cdots+N_{m-1}.
\label{eq:L-m}
\end{equation}
We can state now the iterative Lemma. 

\begin{Lemma}\label{anal-appro-gev}
For each $m\geq 0$ there is $A_m\in U(n)$, $F_m\in
C^\omega_{h_m}(\T^d, u(n))$, and $R^{(m)}\in C^\omega_{h_m}(\T^d,
U(n))$, such that
\begin{eqnarray}\label{anal-appro-gev-sequence}
\left\{ \begin{array}{l} {\rm Ad\, }(R^{(m)}).(\alpha, Ae^{G_{m-1}})=(\alpha, A_{m}e^{F_{m}}) \\[0.3cm]
|F_{m}|_{h_{m}}\leq \varepsilon_{m}\\[0.3cm]
|R^{(m)}|_{h_m}\leq \varepsilon_m^{-2K_*/\sigma}\\[0.3cm]
R^{(m)}\in \Pi(L_{m}, \frac{1}{n^{m}},
\frac{1}{n^{m-1}}\varepsilon_{0}^{1-4\sigma}+\cdots+\frac{1}{n}\varepsilon_{m-2}^{1-4\sigma}+\varepsilon_{m-1}^{1-4\sigma})
\end{array} \right.
\end{eqnarray}
(recall that the definition and property of $\Pi$ has been given in section 2.3).
\end{Lemma}

\begin{proof} Applying Proposition \ref{KAM} one can find $R^{(1)}$
such that
\begin{eqnarray*}
\left\{ \begin{array}{l} {\rm Ad\, }(R^{(1)}).(\alpha, Ae^{G_0})=(\alpha, A_{1}e^{F_{1}}) \\[0.3cm]
|F_{1}|_{h_{1}}\leq \varepsilon_0^{1+\sigma}< \frac{1}{8}\varepsilon_1\\[0.3cm]
|R^{(1)}|_{h_1}\leq \varepsilon_0^{-K_* } = \varepsilon_1^{-K_*/(1+\frac{\sigma}{2}) } < \varepsilon_1^{-2K_*/\sigma }\\[0.3cm]
R^{(1)}\in \Pi (N_0, 1/n, \varepsilon_0^{1-4\sigma}).
\end{array} \right.
\end{eqnarray*}
Arguing by induction assume that for given  $m\geq 2$ we have
\begin{eqnarray*}
\left\{ \begin{array}{l} {\rm Ad\, }(R^{(m)}).(\alpha, Ae^{G_{m-1}})=(\alpha, A_{m}e^{F_{m}}) \\[0.3cm]
|F_{m}|_{h_{m}}\leq \varepsilon_m\\[0.3cm]
|R^{(m)}|_{h_{m}}\leq \varepsilon_{m}^{-2K_*/\sigma}\\[0.3cm]
R^{(m)}\in \Pi(L_m, \frac{1}{n^{m}},
\frac{1}{n^{m-1}}\varepsilon_{0}^{1-4\sigma}+\cdots+\frac{1}{n}\varepsilon_{m-2}^{1-4\sigma}+\varepsilon_{m-1}^{1-4\sigma}).
\end{array} \right.
\end{eqnarray*}
Using   Proposition \ref{KAM} and \eqref{small-cond} one can find $R_m$ and
$F_{m+1}^{(0)}\in C^\omega_{h_{m+1}}(\T^d, u(n))$ such that
\begin{eqnarray*}
\left\{ \begin{array}{l} {\rm Ad\, }(R_m).(\alpha, A_me^{F_{m}})=(\alpha, A_{m+1}e^{F_{m+1}^{(0)}}) \\[0.3cm]
|F_{m+1}^{(0)}|_{h_{m+1}}\leq  \varepsilon_m^{1+\sigma}=\varepsilon_m^{\sigma/2}\varepsilon_m^{1+\sigma/2}
\le \frac{1}{8}\varepsilon_{m+1}\\[0.3cm]
|R_m|_{h_{m+1}}\leq \varepsilon_m^{-K_*}\\[0.3cm]
R_m\in \Pi ( N_m, \frac{1}{n}, \varepsilon_{m}^{1-4\sigma}).
\end{array} \right.
\end{eqnarray*}
Setting  $R^{(m+1)}=R^{(m)} R_m$ we obtain
\begin{eqnarray*}
&&R^{(m+1)}(\cdot+ \alpha)^{-1}Ae^{G_m}R^{(m+1)}\\
&=&R_m(\cdot+\alpha)^{-1} A_me^{F_m}R_m+R^{(m+1)}(\cdot+\alpha)^{-1}A(e^{G_m}-e^{G_{m-1}})R^{(m+1)}\\
&=&A_{m+1}e^{F_{m+1}^{(0)}}+R^{(m+1)}(\cdot+\alpha)^{-1}
A(e^{G_m}-e^{G_{m-1}})R^{(m+1)}.
\end{eqnarray*}
Moreover,
\begin{eqnarray*}
|R^{(m+1)}|_{h_{m+1}}\leq |R^{(m)}|_{h_m}|R_m|_{h_{m+1}}\leq
\varepsilon_m^{-K_*(1+2/\sigma)}=
\varepsilon_{m}^{-\frac{2(1+\sigma/2)K_*}{\sigma}}=\varepsilon_{m+1}^{-2K_*/\sigma} 
\end{eqnarray*}
while Lemma \ref{Lemma:product} implies
\begin{eqnarray*}
R^{(m+1)}&\in &\Pi(L_{m}+N_m, \frac{1}{n^{m+1}},
\frac{1}{n}(\frac{1}{n^{m-1}}\varepsilon_{0}^{1-4\sigma}+\cdots+\varepsilon_{m-1}^{1-4\sigma})+\varepsilon_m^{1-4\sigma})\\[0.3cm]
&=& \Pi(L_{m+1}, \frac{1}{n^{m+1}},
\frac{1}{n^m}\varepsilon_0^{1-4\sigma}+\cdots+\frac{1}{n}\varepsilon_{m-1}^{1-4\sigma}+\varepsilon_{m}^{1-4\sigma}).
\end{eqnarray*}
Taking into account  \eqref{eq:estimates-G} one  obtains
\begin{eqnarray*}
|R^{(m+1)}(\cdot+\alpha)^{-1}
A(e^{G_m}-e^{G_{m-1}})R^{(m+1)}|_{h_{m+1}}\\[0.3cm]
\leq
2\varepsilon_{m+1}^{-4K_*/\sigma} \varepsilon_{m+1} ^{4(1+K_*/\sigma)}=2\varepsilon_{m+1}
^4 \le \frac{1}{4}\varepsilon_{m+1} ,
\end{eqnarray*}
Moreover,
\begin{eqnarray*}
|A_{m+1}e^{F_{m+1}^{(0)}}-A_{m+1}|_{h_{m+1}}\leq 2\times \frac{1}{8}
\varepsilon_{m+1}= \frac{1}{4} \varepsilon_{m+1}
\end{eqnarray*}
which implies
\begin{eqnarray*}
| R^{(m+1)}(\cdot+
\alpha)^{-1}Ae^{G_m}R^{(m+1)}-A_{m+1}|_{h_{m+1}}\leq  \frac{1}{4}
\varepsilon_{m+1}+\frac{1}{4}\varepsilon_{m+1} \leq
\frac{1}{2}\varepsilon_{m+1}.
\end{eqnarray*}
Setting 
\[
F_{m+1}:= \log\left(A_{m+1}^{-1}R^{(m+1)}(\cdot+
\alpha)^{-1}Ae^{G_m}R^{(m+1)}-I\right)
\]
and using  $(\ref{log-cond})$ one  obtains 
 \begin{eqnarray*}
\left\{ \begin{array}{l} {\rm Ad\, }(R^{(m+1)}).(\alpha, Ae^{G_{m}})=(\alpha, A_{m+1}e^{F_{m+1}}) \\[0.3cm]
|F_{m+1}|_{h_{m+1}}\leq \varepsilon_{m+1}\\[0.3cm]
|R^{(m+1)}|_{h_{m+1}}\leq \varepsilon_{m+1}^{-2K_*/\sigma}\\[0.3cm]
 R^{(m+1)}\in
\Pi(L_{m+1}, \frac{1}{n^{m+1}},
\frac{1}{n^m}\varepsilon_0^{1-4\sigma}+\cdots+\frac{1}{n}\varepsilon_{m-1}^{1-4\sigma}+\varepsilon_{m}^{1-4\sigma}) .
\end{array} \right.
\end{eqnarray*}
This completes the induction argument and proves the iterative Lemma. 
\end{proof}

\subsection{Gevrey reducibility}\label{Gevrey reducibility}
In the previous section we have established almost reducibility of the Gevrey-${\mathcal G}^\rho$ cocycle $(\alpha, Ae^{G})$. More precisely, for each $m\in\N$ we have obtained in 
Lemma \ref{anal-appro-gev} a map  $R^{(m)}$ which 
conjugates $(\alpha, Ae^{G_{m-1}})$  to $(\alpha, A_me ^{F_m})$,
where $F_m$ is of the size of $\varepsilon_m$. In general, the sequence $R^{(m)}$ diverges. 
Our aim in this section is to prove that
the sequence $R^{(m)}$ is convergent in ${\mathcal G}^\rho$, provided   that there exits
a measurable function $B: \T^d\rightarrow U(n)$ with $\lceil B^{-1}\rfloor=\lceil B^* \rfloor\geq \epsilon$  and
$C\in \Sigma(\alpha)$ such that
\begin{eqnarray}\label{meas-conj}
{\rm Ad\, }(B). (\alpha, Ae^{G})=(\alpha, C)
\end{eqnarray} 
where $0<\epsilon \le 1$. To do this we impose condition \eqref{eq:small-condition-epsilon}. More precisely, choosing $0<h_0<1$ sufficiently small we assume as well that 
\begin{equation}
\varepsilon_0:=\exp\left(-\frac{1}{16(1+K_*/\sigma)}(cLh_0)^{-\frac{1}{\rho-1}}\right) =  \tilde \varepsilon_0 \epsilon^\ell
\label{eq:small-condition-epsilon1}
\end{equation}
where $\tilde \varepsilon_0:=\tilde \varepsilon_0(n,\sigma,\kappa,\rho,L)\ll 1$ is sufficiently small so that all previous assumptions on $\varepsilon_0$ hold when $\epsilon=1$. 
To do this we choose appropriately $h_0$ in a function of $\epsilon$ as well. Setting $\delta:= 2c_0L^d\tilde \varepsilon_0$ and using \eqref{GEV-cond} we obtain the small constant in \eqref{eq:small-condition-epsilon}.
  We assume that 
\begin{equation}
\tilde \varepsilon_0^{ \sigma(1-4\sigma)/2} \le \frac{1}{8n}.
\label{eq:eps-cond4}
\end{equation}
We have  $1-5\sigma\ge 1/\ell$ by \eqref{eq:sigma} which implies 
\begin{equation}\label{eps-cond2}
 \varepsilon_0^{1-4\sigma} <  \tilde\varepsilon_0^{\sigma} \varepsilon_0^{1/\ell} \le  \frac{1}{8  n}\epsilon^\ell
\end{equation}
and we obtain
\begin{eqnarray*}
\frac{1}{n^{m-1}}\varepsilon_{0}^{1-4\sigma}+\cdots+\frac{1}{n}\varepsilon_{m-2}^{1-4\sigma}+\varepsilon_{m-1}^{1-4\sigma} &<& \frac{1}{n^{m-1}}\varepsilon_{0}^{1-4\sigma}\left(1+\frac{1}{2} + \cdots +  \frac{1}{2^{m-1}}\right)\nonumber \\
&< & \frac{2}{n^{m-1}} \varepsilon_{0}^{1-4\sigma} \le  \frac{\epsilon}{4n^{m}} .
\end{eqnarray*}
This inequality combined with Lemma \ref{Lemma:Pi}, $5)$ and Lemma \ref{anal-appro-gev}, implies 
\begin{eqnarray}\label{eq:R}
R^{(m)}&\in &\Pi(L_m, \frac{1}{n^m}, \frac{1}{n^{m-1}}\varepsilon_{0}^{1-4\sigma}+\cdots+\frac{1}{n}\varepsilon_{m-2}^{1-4\sigma}+\varepsilon_{m-1}^{1-4\sigma})\nonumber\\
&\subseteq &\Pi(L_m, \frac{1}{n^m}, \frac{\epsilon}{4n^{m}} ).
\end{eqnarray}
Now we can get a good control on the size of $R_m$.  \\

\begin{Lemma}\label{conj-as-large}
For any $m$ sufficiently large  there is
$Y_m\in C^\omega_{h_{m}}(\T^d, u(n))$ such that
\begin{eqnarray*}
\left\{ \begin{array}{l}R_m=e^{Y_m} \\
|Y_m|_{h_{m}}\leq \varepsilon_m^{1-2\sigma}\\
|A_m-A_{m+1}|\leq \varepsilon_m^{1/2} .
\end{array} \right.
\end{eqnarray*}
\end{Lemma}

\begin{proof} The idea is to prove    
that 
$A_m\in {\rm NR \, }(N_m,\varepsilon_m^\sigma)$ for large  $m$ which allows us to use   Proposition \ref{KAM}, (ii). 
In the following we will need  only the  $L^\infty$ norm $\|\cdot\|_{\infty}$ on $\T^n$. 

 By Lemma $\ref{anal-appro-gev}$ and  (\ref{meas-conj}) ) we have the equalities 
\[
\left\{
\begin{array}{lcrr}
{\rm Ad\, }(R^{(m)}).(\alpha, Ae^{G_{m-1}})=(\alpha, A_m e^{F_m}), \\[0.3cm]
{\rm Ad\, }(B). (\alpha, Ae^{G})=(\alpha, C).
\end{array}
\right.
\]
Setting  $B_m=B^{-1}R^{(m)}$ and using  \eqref{eq:estimates-G} we obtain
\[
{\rm Ad\, }(B_m).(\alpha, C)=\left(\alpha, A_m
e^{F_m}+O(\varepsilon_{m+1}^{4(1+K_*/\sigma)})\right) .
\]
Hereafter, the symbol $O(\varepsilon_{m}^\alpha)$, $\alpha\in\R$, stands for a map $W:\T^d\to M_n(\R)$ such that $ \|W\|_{\infty} \le c \varepsilon_{m}^\alpha$, where $c>0$ does not depend on $m$ and $\theta$. 
Since the $L^\infty$-norm of any measurable unitary matrix function is $1$ this yields 
\begin{eqnarray*}
B_m(\cdot+\alpha)A_m(I+O(\varepsilon_{m}))=CB_m, 
\end{eqnarray*}
which implies 
\begin{eqnarray}\label{meas-conj-seq}
B_m(\cdot+\alpha)A_m=C B_m+O(\varepsilon_m).
\end{eqnarray}
By Lemma \ref{Lemma:Gamma}  there exists  $N_*$, such that $B^{-1}=B^\ast\in \Gamma(N_*, \epsilon/2)$, 
and using \eqref{eq:R} and  \eqref{eq:Pi} applied to $B^{*}$ we arrive at 
\begin{eqnarray}\label{meas-conj-est}
B_m=B^{*}R^{(m)}&\in &\Gamma(N_*+L_m,
\frac{\epsilon}{2n^m}- \frac{\epsilon}{4n^{m}} )\nonumber\\
&= &\Gamma(N_*+L_m,
\frac{\epsilon}{4 n^m}).
\end{eqnarray}
Recall that $C\in \Sigma(\alpha)$, which means that  there exist $\chi,\nu>0$,   
such that  
\begin{eqnarray}\label{nn7}
|\mu_{p}e^{2\pi  i\langle k,\alpha\rangle}-\mu_q|\geq
\frac{\chi}{(1+|k|)^\nu}, \quad \mbox{for any} \quad p\neq q , \quad k\in\Z^d,
\end{eqnarray}
where $
\{\mu_1,\cdots,\mu_n\}={\rm Spec\, }(C)$ .
Let $T\in U(n)$ be such that \[TC T^*= {\rm diag\, }(\mu_1,\cdots,\mu_n). \]
Choose  $S_m\in U(n)$ such that  
\[
S_m^*A_mS_m= {\rm diag\, }(\lambda^{(m)}_1,\cdots,\lambda^{(m)}_n)
\]
and set  $D_m=T^*B_mS_m$. By Lemma \ref{Lemma:Gamma} 
\begin{equation}
D_m\in
\Gamma\left(N_*+L_m,\frac{\epsilon}{4n^{m}}\right) 
\label{eq:D}
\end{equation}
and using 
(\ref{meas-conj-seq})   we obtain
\[
D_m(\cdot+\alpha){\rm diag\,
}(\lambda^{(m)}_1,\cdots,\lambda^{(m)}_n)= {\rm diag\,
}(\mu_1,\cdots,\mu_n)D_m+O(\varepsilon_m).
\]
which means that
\begin{eqnarray*}
D_m(\cdot+\alpha) {\rm diag\, }(\lambda^{(m)}_1,\cdots,\lambda^{(m)}_n)+W_m= {\rm diag\, }(\mu_1,\cdots,\mu_n)D_m
\end{eqnarray*}
with $\|W_m \|_{\infty}\leq \mbox{cst.}\, 
\varepsilon_m$.  Thus for any $1\le p,q\le n$  we have
\begin{eqnarray*}
e^{ 2\pi i\langle k,\alpha\rangle}\lambda^{(m)}_q\widehat{d}^{(m)}_{p,q}(k)+\widehat{w}^{(m)}_{p,q}(k)=\mu_p
\widehat{d}^{(m)}_{p,q}(k)
\end{eqnarray*}
where $\widehat{d}^{(m)}_{p,q}(k)$ and $\widehat{w}^{(m)}_{p,q}(k)$ stand fot the $(p,q)$ entries of the matrices $\widehat{D}_m(k)$ and of $\widehat{W}_m(k)$ respectively. We have  \[
|\widehat{w}^{(m)}_{p,q}(k)|\leq \|W_m \|_{\infty}\leq \mbox{cst.}\varepsilon_m
\] 
which implies 
\begin{eqnarray}\label{nn5}
|\widehat{d}^{(m)}_{p,q}(k)\|\lambda^{(m)}_qe^{2\pi
i\langle k,\alpha\rangle}-\mu_p|\leq \mbox{cst.} \varepsilon_m
\end{eqnarray}  
for all 
$p,q\in \{1,\cdots,n\}$ and $k\in \Z^d$. 
It follows from \eqref{eq:D} and the definition of $\Gamma$ in Section \ref{sec:measurable-functions}  that for any $p$ there exist $q\in \{1,\cdots,n\}$ and 
$k\in \Z^d$ with $|k|\leq N_*+L_m$   such
that 
\[
|\widehat{d}^{(m)}_{p,q}(k)|\geq \frac{\epsilon}{4n^{m}}
\]  
($q$ and $k$ depend on $m$ as well). 
Then using 
$(\ref{nn5})$ and choosing $m_1 = m_1(\epsilon) \gg 1$ we obtain 
\begin{eqnarray}\label{nn6}
|\lambda^{(m)}_{q}e^{2\pi i\langle k,\alpha\rangle}-\mu_p|&\leq &\mbox{cst.}\, \epsilon^{-1}
n^{m}\varepsilon_m\nonumber\\
&\leq &\mbox{cst.} \, \varepsilon_m^{1/2}
\end{eqnarray}
for any $m\ge m_1$. Consider the map
\[f_m: \{1,\cdots,n\} \to \{1,\cdots,n\}\] assigning to each $p\in \{1,\cdots,n\}$ an integer $q\in \{1,\cdots,n\}$ such that \eqref{nn6} holds with some $k\in \Z^d$ of length $|k|\leq N_*+L_m$. 
\begin{Lemma}\label{unique} 
There is $m_0\in\N$ such that the map $f_m$ is bijective for any $m\ge m_0$.  
\end{Lemma}
\begin{proof} It suffices to prove that $f_m$ is injective. Suppose that there are $p_1\neq p_2$ such that $f_m(p_1)=f_m(p_2)=q$. 
Then there exist  $k_1, k_2$ such that
\[
\left\{
\begin{array}{lcrr}
|\lambda^{(m)}_{q}e^{2\pi i\langle k_1,\alpha\rangle}-\mu_{p_1}|\leq
\mbox{cst.}\varepsilon_m^{1/2} ,\\[0.3cm]
|\lambda^{(m)}_{q}e^{
2\pi i\langle k_2,\alpha\rangle }-\mu_{p_2}|\leq \mbox{cst.}\varepsilon_m^{1/2}.
\end{array}
\right.
\]
This implies 
\begin{eqnarray*}
0&=&|\lambda^{(m)}_{q}-\lambda^{(m)}_{q}|\\ &\geq&
|\mu_{p_1}e^{ 2\pi i\langle-k_1+k_2,\alpha\rangle}-\mu_{p_2}| -
\mbox{cst.}\varepsilon_m^{1/2} \\
&\geq & \frac{\chi}{(2 L_m+2N_*+1)^\nu}-\mbox{cst.}\varepsilon_m^{1/2}
\end{eqnarray*}
which  can not be true for  $m\gg 1$ in view of  \eqref{eq:N-m},  \eqref{eq:L-m} and \eqref{rem-est3}. \end{proof}

\begin{Lemma}\label{A} 
There is $m_0\in\N$ such that $A_m\in {\rm NR \, }(N_m,\varepsilon_m^\sigma)$ for  $m\ge m_0$. 
\end{Lemma}
\begin{proof} Suppose that  $A_m$ is
$(N_m,\varepsilon_m^\sigma)-$resonant. Then there  exist $k\in \Z^d$ with
$0<|k|\leq N_m$ and $p,q\in \{1,\cdots\,n\}$ such that
\[|\lambda_{p}e^{2\pi  i\langle k,\alpha\rangle}-\lambda_q|\leq
\varepsilon_m^\sigma.\] If $p=q$, we get
\begin{eqnarray*}
\varepsilon_m^\sigma\geq|e^{2\pi  i\langle k,\alpha\rangle}-1|\geq
\frac{2\pi \gamma^{-1}}{N_m^\tau}
\end{eqnarray*}
which  can not be  true for $m\gg 1$. If $p\neq q$
and  $m$ is large enough, it follows from Lemma \ref{unique} that there exist 
 $k(p),k(q)\in \Z^d$ satisfying $|k(p)|,|k(q)|\leq L_m+N_*$, 
such that
\begin{eqnarray*}
\max\{|\lambda_pe^{2\pi
 i\langle k(p),\alpha\rangle }-\widetilde{\mu}_p|,\ |\lambda_qe^{2\pi
 i\langle k(q),\alpha\rangle }- \widetilde{\mu}_q|\}\leq \mbox{cst.}\,  \varepsilon_m^{1/2},
\end{eqnarray*}
where
\[\widetilde{\mu}_p=\mu_{f_m^{-1}(p)}\neq \mu_{f_m^{-1}(q)}= \widetilde{\mu}_q.\]
Then setting $l=k-k(p)+k(q)$  we get
\begin{eqnarray*}
|\widetilde{\mu}_p e^{2\pi i\langle l,\alpha\rangle }- \widetilde{\mu}_q | &=&
|\widetilde{\mu}_p e^{2\pi  i\langle k-k(p),\alpha\rangle }- \widetilde{\mu}_q e^{2\pi  i\langle -k(q),\alpha\rangle}|\\
&\leq & |\lambda_pe^{2\pi i\langle k,\alpha \rangle}-\lambda_q|+ \mbox{cst.}\, 
\varepsilon_m^{1/2}\\
&\leq & \varepsilon_m^\sigma+\mbox{cst.} \varepsilon_m^{1/2}.
\end{eqnarray*}
On the other hand, the assumption $(\ref{nn7})$ yields
\begin{eqnarray*}
|\widetilde{\mu}_p e^{2\pi  i\langle l,\alpha\rangle}- \widetilde{\mu}_q|
&\geq &
\frac{\chi}{(|l|+1)^\nu}\\
 &\geq  &\frac{\chi}{(N_m+2L_m+2N_*+1)^\nu},
\end{eqnarray*}
hence, 
\begin{eqnarray*}
\frac{\chi}{(N_m+2L_m+N_*+1)^\nu}\leq \varepsilon_m^\sigma+\mbox{cst.}\, 
\varepsilon_m^{1/2}
\end{eqnarray*}
which is not  true for $m\gg 1$. \end{proof}

 Lemma \ref{A} allows one to apply Proposition \ref{KAM}(ii),  which completes  the proof of 
 Lemma \ref{conj-as-large}. \end{proof}

Now we can prove reducibility in Gevrey classes. 
\begin{Lemma}\label{Gev-conj} (Gevrey$-\mathcal{G}^\rho$ Reducibility)
There exist  $R\in {\mathcal G}^{\rho}$ and $\widetilde{A}\in U(n)$ such that
\begin{eqnarray}\label{gev-conj-lim}
 {\rm Ad\, }(R).(\alpha, A e^{G})=(\alpha, \widetilde{A}).
\end{eqnarray}
\end{Lemma}

\begin{proof} Recall from Lemma \ref{GEV-cond} and Lemma \ref{anal-appro-gev} that
\[
\begin{array}{lcrr}
 {\rm Ad\, }(R^{(m)}).(\alpha, Ae^{G_{m-1}})=(\alpha,
 A_{m}e^{F_{m}}), \\[0.3cm]
|G_m-G_{m-1}|_{h_m}\leq \varepsilon_{m+1}^{4(1+K_*/\sigma)} \quad \mbox{and} \quad
|F_m-F_{m+1}|_{h_{m+1}}\leq 2\varepsilon_m .
\end{array}
\]
Moreover, by Lemma \ref{conj-as-large}, there exists $m_*$, such that for any
$m\geq m_*$ the following estimate holds true
\begin{eqnarray*}
\label{estimate-R} |R^{(m+1)}-R^{(m)}|_{h_{m+1}}&=&
|R^{(m)}R_m-R^{(m)}|_{h_{m+1}} \nonumber \\
&=&|R^{(m)}(R_m-I)|_{h_{m+1}} \nonumber \\
&=&|R^{(m)}(e^{Y_m}-I)|_{h_{m+1}} \nonumber \\
&\leq& 2 \varepsilon_m^{1-2\sigma}|R^{(m)}|_{h_{m+1}}.
\end{eqnarray*}
This implies 
\begin{eqnarray*}
|R^{(m)}|_{h_{m+1}} \leq 2 |R^{(m_*)}|_{h_{m_*+1}},
\end{eqnarray*}
and
\begin{eqnarray*}
 |R^{(m+1)}-R^{(m)}|_{h_{m+1}}
\leq 4 \varepsilon_m^{1-2\sigma} |R^{(m_*)}|_{h_{m_*+1}}.
\end{eqnarray*}
Taking  $m_{*}\gg 1$ we get for any  $m\geq m_{*}$ the estimate
\begin{eqnarray*}
 |R^{(m+1)}-R^{(m)}|_{h_{m+1}}
\leq \varepsilon_m^{1/2} .
\end{eqnarray*}
hence, for any $m\geq m_{*}$ we have
\[
\zeta_m :=\max\{|G_m-G_{m-1}|_{h_{m+1}},\ |F_m-F_{m+1}|_{h_{m+1}},\
|R^{(m+1)}-R^{(m)}|_{h_{m+1}}\}\leq\varepsilon_m^{1/2}.
\]
Recall that 
\[
h_m=h_0\left(1+\frac{\sigma}{2}\right)^{-m(\rho-1)} \quad \mbox{and} \quad 
\quad \varepsilon_m=\exp\left\{-\frac{1}{16(K_*/\sigma +1)}(cLh_m)^{-\frac{1}{\rho-1}}\right\}. 
\]
Then for any $m$  sufficiently large   we get
\begin{eqnarray}
\label{estimate-zeta}
&&\exp\left\{-\left((32(K_*/\sigma +1))^{\rho-1}cLh_m\right)^{-\frac{1}{\rho-1}}\right\} \\
&&= \exp\left\{-\frac{1}{32(K_*/\sigma +1)}(cLh_m)^{-\frac{1}{\rho-1}}\right\} \nonumber \\
&&= \varepsilon_m^{1/2}
\geq 
 \zeta_m.
\end{eqnarray}
By the inverse approximation lemma (Proposition \ref{inv-app-lem}) and \eqref{estimate-zeta}  choosing  
\[
\widetilde{L}\geq c_0(\rho,d)(32(K_*/\sigma +1))^{\rho-1}cL
\]
we obtain that
\begin{itemize}
\item $R^{(m)}$  converges in ${\mathcal G}_{\widetilde{L}}^\rho$ to some  $R\in {\mathcal G}_{\widetilde{L}}^\rho (\T^d, U(n))$,
\item  $G_m$ converges in ${\mathcal G}_{\widetilde{L}}$ and it converges  to $G$ in $C^0$, so $G_m$ converges to $G$ in  ${\mathcal G}_{\widetilde{L}}^\rho$,
\item  $F_m$  converges in ${\mathcal G}_{\widetilde{L}}^\rho$, and $F_m$ converges to $0$ in  $C^0$, hence,  $F_m$ converges to $0$ in  ${\mathcal G}_{\widetilde{L}}^\rho$.
\end{itemize}
To finish the proof, one  has just to let  $m \to \infty$ in
 \[{\rm Ad\, }(R^{(m)}).(\alpha, Ae^{G_{m-1}})=(\alpha, A_{m}e^{F_{m}}) \]
with $A_m$ converging to some $\widetilde{A}\in U(n)$ by Lemma \ref{conj-as-large}.  \end{proof}

\subsection{Proof of Theorem \ref{Local}}

We are going to complete the proof of  Theorem \ref{Local}. By Lemma \ref{Gev-conj}, there exist  $R\in {\mathcal G}_{\widetilde{L}}^{\rho}$ and $\widetilde{A}\in U(n)$ such that
\begin{eqnarray*}
 {\rm Ad\, }(R).(\alpha, A e^{G})=(\alpha, \widetilde{A}).
\end{eqnarray*}
Moreover, there exist by assumption a
measurable $B: \T^d\rightarrow U(n)$, $A\in U(n)$ and
$C\in \Sigma(\alpha)$ such that
\begin{eqnarray*}
{\rm Ad\, }(B). (\alpha, Ae^{G})=(\alpha, C).
\end{eqnarray*} 
Setting $V=RB^{-1}$, then $B=V^{-1}R$. $R$ is analytic, to prove that $B$ is almost surely Gevrey$-{\mathcal G}^{\rho}$, we just need to prove that $V$ is so. In fact, we can prove that there exist  $U_1, U_2, U_3\in U(n)$ and $k^{(1)},\cdots,k^{(n)}\in \Z^d$ such
that 
\begin{eqnarray}\label{const-conj}
V(\theta)=U_1 U_3^* \exp\{2\pi i  {\rm diag\,
}(\langle k^{(1)},\theta\rangle,\cdots,\langle k^{(n)},\theta\rangle)\}  U_2
\end{eqnarray}
for a.e. $\theta\in \T^d$. 
Then for a.e. $\theta\in \T^d$ we have 
\begin{eqnarray*}
\widetilde{B}(\theta):&=& V(\theta)^{-1} R(\theta)\\
&=& U_1U_3^*\exp\{2\pi i  {\rm diag\,
}(\langle k^{(1)},\theta\rangle,\cdots,\langle k^{(n)},\theta\rangle)\} U_2^*R(\theta)
= B(\theta).
\end{eqnarray*} 
Obviously  $\widetilde{B}$ is in $ {\mathcal G}_{\widetilde{L}}^{\rho}$ and we obtain the desired result.

To prove (\ref{const-conj}), let us consider the conjugation
\begin{eqnarray*}
{\rm Ad\, }(V). (\alpha, C)=(\alpha, \widetilde{A}).
\end{eqnarray*} 
We write \[C=U_1 {\rm diag\, }(\mu_1,\cdots,\mu_n)U_1^* \quad \mbox{and}\quad 
\widetilde A=U_2 {\rm diag\, }(\widetilde{\mu}_1\cdots,  \widetilde{\mu}_n)U_2^*\]
where 
$U_1,U_2\in U(n)$ and we get
\begin{align*}
\widetilde{V}(\cdot +\alpha) {\rm diag\, }(\mu_1,\cdots,\mu_n)
={\rm diag\, }(\widetilde{\mu}_1,\cdots,\widetilde{\mu}_n)\widetilde{V},
\end{align*}
where $\widetilde{V}:=U_2^*V^{-1}U_1:\T^d\to U(n)$ is measurable. 
Then  for any $p,q\in\{1,\cdots,n\}$ and $k\in
\mathbb{Z}^d$ we obtain
\begin{align*}
(e^{2 \pi i \langle k,\alpha\rangle }\mu_q-\widetilde{\mu}_p)
\widehat{\widetilde{v}}_{p,q}(k)=0.
\end{align*}
where $\widehat{\widetilde{v}}_{p,q}(k)$ denotes the $(p,q)$ entry of $\widehat{\widetilde{V}}(k)$. 
Since $\lceil \widetilde{V} \rfloor >0$  there exist $k^{(1)},\cdots,k^{(n)}
\in \mathbb{Z}^d$, such that
\[
\widetilde{\mu}_j=\mu_j e^{2 \pi i \langle k^{(j)},\alpha\rangle},\quad j=1,\ldots, n.
\]
Setting $W(\theta)={\rm diag\, } (
e^{2\pi i \langle k^{(1)},\theta\rangle},\cdots,e^{2\pi i\langle k^{(n)},\theta\rangle})\widetilde{V}(\theta)$
we get
\begin{align*}
W(\cdot +\alpha) {\rm diag\, }(\mu_1,\cdots,\mu_n)= {\rm diag\, }(\mu_1
,\cdots,\mu_n)W.
\end{align*}
Then for all $p,q\in \{1,\cdots,n\}$ and $k\in \mathbb{Z}^d$ we have 
\begin{align*}
(e^{2 \pi i \langle k,\alpha \rangle}\mu_q-\mu_p)\widehat{w}_{p,q}(k)=0.
\end{align*}
where $\widehat{w}_{p,q}(k)$ denotes the $(p,q)$ entry of $\widehat{W}(k)$. The relation $C\in\Sigma(\alpha)$ implies that 
\begin{align*}
\forall\, k\neq0, \quad e^{2 \pi i \langle k,\alpha\rangle }\mu_q-\mu_p\neq 0
\end{align*}
and we obtain   that $\widehat{W}(k)=0$ for all $0\neq
k\in\mathbb{Z}^d$. Thus there exists $U_3\in U(n)$, such that
\[W(\theta)=U_3\]
 for a.e. $\theta\in \mathbb{T}^d$, which implies (\ref{const-conj}). 

\section{Global Setting}\label{Global}

In this section we will prove the global rigidity
result (Theorem \ref{Gloabl}) using Theorem \ref{Local} and adapting the renormalization scheme of a
$\Z^2-$action developed by  Krikorian \cite{AK,Kr1,HY2} to the case of Gevrey classes.

\subsection{$\mathbb{Z}^2-$Action in Gevrey classes}

Given $\rho>1$ and $K>0$ we denote by
$\mathcal{G}_L^\rho([-K,K],U(n))$ the Banach space of all $C^\infty$ functions $A: [-K,K] \rightarrow U(n)$ with finite norm
\[
\|A\|^{\mathcal{G}}_{L,K} := \sup_{r\in \N}\sup_{|\theta|\leq K}(|\partial^{r}A(\theta)|L^{-|r|}r!^{-\rho})<\infty.
\]
Denote by $\mathcal{G}_L^\rho(\R,U(n))$  the functional space
$\bigcap_{K>0}\mathcal{G}_L^\rho([-K,K],U(n))$ equipped by the  projective limit topology (a sequence $A_m \in \mathcal{G}_L^\rho(\R,U(n))$ converges to  $A
\in \mathcal{G}_L^\rho(\R,U(n))$ if it converges in  $\mathcal{G}_L^\rho([-K,K],U(n))$ for any $K>0$)
and set $\mathcal{G}^\rho(\R,U(n)):=\cup_{L>0}\mathcal{G}_L^\rho(\R,U(n))$  equipped by the corresponding inductive limit topology (a sequence $A_m \in \mathcal{G}^\rho(\R,U(n))$ converges to  $A
\in \mathcal{G}^\rho(\R,U(n))$ if  there is $L>0$ such that the the sequence $A_m $ converges to in  $\mathcal{G}_L^\rho(\R,U(n))$).

Denote by  $SW^{\mathcal{G}}_{L,\rho}(\R, U(n))$ the composition group  of all $(\gamma,
A)\in \R\times \mathcal{G}_L^\rho(\R, U(n))$ acting on $\R\times U(n)$ by
\begin{eqnarray*}
(\gamma, A): &&\R\times U(n)\rightarrow \R\times U(n)\\
&& (\theta,v)\mapsto (\theta+\gamma, A(\theta)v).
\end{eqnarray*}
Denote as well  by $\Lambda^{L,\rho}$  the set of all Gevrey-$\mathcal{G}^\rho$ fibered
$\Z^2$-actions. By definition  $\Phi\in \Lambda^{L,\rho}$ if it is  a homomorphism from the additive group $\Z^2$ to the
composition group $SW^{\mathcal{G}}_{L,\rho}(\R, U(n))$,
\begin{eqnarray*}
\Phi: &&\Z^2\rightarrow SW^{\mathcal{G}}_{L,\rho}(\R, U(n))\\
&& (k_1,k_2)\mapsto \Phi(k_1,k_2)=(\gamma_{k_1,k_2}^{\Phi}, A_{k_1,k_2}^{\Phi}).
\end{eqnarray*}
Any $\Z^2$-action $\Phi$ is  completely determined by its values on $(1,0)$ and $(0,1)$ and we write it as follows
\[\Phi=\{\Phi(1,0),\Phi(0,1)\}
=\{(\gamma^{\Phi}_{1,0}, A^{\Phi}_{1,0}),(\gamma^{\Phi}_{0,1},
A^{\Phi}_{0,1})\}.\] We denote by  $\Lambda^\rho$ the union
$\cup_{L>0}\Lambda^{L,\rho}$. To shorten the notations we sometimes skip the index $\rho$ which is fixed.

A   $\Z^2$-action $\Phi$ is  said to be {\em normalized} if
$\Phi(1,0)=(1, I)$. If $\Phi$ is
normalized then $\Phi(0,1)=(\alpha,A)$ can be  viewed as a cocycle
in $SW^{\mathcal{G}}(\T^1, U(n))$, since $A$ is automatically
$\Z-$periodic.  Conversely, to any  $(\alpha,A)\in SW^{\mathcal{G}}(\T^1, U(n))$ one can associate a normalized $\Z^2$-action
$\Phi=\{(1,I),(\alpha,A)\}$.\\

A $\Z^2$-action $\Phi \in \Lambda^{L,\rho}$  is said to be
conjugated to a $\Z^2$-action  $\widetilde{\Phi}$ by
 $B\in \mathcal{G}_{L}^{\rho}(\R, U(n))$ if for any $(k_1,k_2)\in\Z^2$
\begin{eqnarray*}
 {\rm Ad\, }(B).\Phi(k_1,k_2)&:= &(0,B)^{-1} \circ \Phi(k_1,k_2) \circ (0,B)\\
&=&(0,B^{-1}) \circ \Phi(k_1,k_2) \circ (0,B) \\
&=&\widetilde{\Phi}(k_1,k_2).
\end{eqnarray*}
 The following lemma states that any Gevrey $\Z^2$-actions can  be conjugated to a normalized
one in the same Gevrey class.

\begin{Lemma}\label{norm} (Normalization).
Any  $\Z^2$-action of the form $\Phi=\{(1,C), (\gamma, D)\}\in
\Lambda^{L,\rho}$  can be conjugated to a normalized one by some
$P\in \mathcal{G}_{cL}^\rho(\R, U(n))$, where $c\ge 1$. Moreover,  if 
\begin{equation} \label{eq:C}
\sup_{|\theta|\le 3}\|C(\theta)-I\| <1/3
\end{equation}
then one can choose $P$
so that
\begin{eqnarray}\label{gev-norm-est}
\|P-I\|^{\mathcal{G}}_{cL,2}\leq \mbox{cst.} \|C-I\|^{\mathcal{G}}_{L,3}.
\end{eqnarray}
\end{Lemma}

\begin{proof}
Choose  $X_0\in u(n)$ such that $C(0) = e^{X_0}$, fix $0<\delta<3/5$ so that 
\begin{equation} \label{eq:C0}
\|e^{-X_0}C(\theta-1) - I\|<2/3 \quad  \mbox{for} \quad \theta\in [1-\delta,1+\delta]
\end{equation}
and define
 $Y\in \mathcal{G}_L^\rho([1-\delta,1+\delta], u(n))$ by
\begin{equation}
\label{eq:Y}
Y(\theta):=\log\left\{e^{-X_0}C(\theta-1)\right\} \quad  \mbox{for}\quad  \theta\in
[1-\delta,1+\delta].
\end{equation}
We have 
\begin{equation}
\label{eq:Y-definition}
e^{X_0}e^{Y(\theta)}=C(\theta-1) \quad  \mbox{for}\quad \theta\in
[1-\delta,1+\delta].
\end{equation}
Since $\log t= (1-t)h(t)$, where $h$ is analytic in the unit ball $\{|t|<1\}$, we get the following inequality by estimating  the composition of Gevrey functions (see \cite{P}, Proposition A.3) 
\begin{eqnarray}\label{y-gev-est}
\sup_{r\in \N} \, \sup_{|\theta-1|\leq \delta}\,   (|\partial^r Y(\theta)|L^{-|r|}r!^{-\rho})&\leq &
\mbox{cst.}\,  \|C(0)^{-1}C(\cdot)-I\|^{\mathcal{G}}_{cL,\delta} \nonumber\\
&\leq &
\mbox{cst.}\,  \|C-I\|^{\mathcal{G}}_{cL,\delta}  
\end{eqnarray}
where $c\ge 1$ is a constant independent of the constant $L\ge 1$. 

Choosing $c\gg 1$ there is   $b$  in 
$\mathcal{G}_{c}^\rho \subset \mathcal{G}_{cL}^\rho$ such that $0\le b \le 1$, $b(\theta)=0$ for $\theta\in (-\infty,1/3]$, $b(\theta)=1$ for $\theta\in [2/3,+\infty]$ and 
$b$ is strictly increasing  on $[1/3,2/3]$. Set 
\[
b_\delta(\theta)= b\left(\frac{\theta-(1-\delta)}{\delta}\right) \quad \mbox{and}\quad   f_\delta(\theta)=(1-\delta)+ (\theta-(1-\delta)) b_\delta(\theta)  
\]
and define
\begin{equation}
\label{eq:P}
P(\theta)=e^{b_\delta(\theta)X_0}e^{b_\delta(\theta)Y(f_\delta(\theta ))},  \quad \theta\in [-\delta,1+\delta].
\end{equation}
The functions $b_\delta$ and $f_\delta$ have the 
 following properties
\begin{itemize}
\item  $b_\delta(\theta)=1$ and $f_\delta(\theta)=\theta$ for $\theta\in [1-\delta/3,1+\delta]$,
\item $1-\delta\le f_\delta(\theta)\le\theta$ and $f_\delta$ is increasing on  $[1-2\delta/3,1-\delta/3]$,
\item $b_\delta(\theta)=0$ and $f_\delta(\theta)=1-\delta$ for $\theta\in
[-\delta,1-2\delta/3]$.
\end{itemize}
The theorem about the composition of Gevrey functions implies that $P\in \mathcal{G}_{cL}^\rho([1-\delta,1+\delta], U(n))$, where $c\ge 1$. Moreover, 
$P(\theta)= e^{X_0}e^{Y(\theta)}=C(\theta-1)$ for $\theta\in [1-\delta/3, 1+\delta]$ and $P(\theta)= I$ for $\theta\in [-\delta,\delta]\subset [-\delta, 1-2\delta/3]$ since $0<\delta<3/5$, and  we obtain 
\begin{equation}\label{eq:B-C}
P(\theta+1)^{-1}C(\theta)P(\theta)=I \quad \mbox{for} \quad |\theta|< \delta . 
\end{equation}
Now we
extend $P$ in $\R$ by
\begin{equation}\label{eq:P-extension}
P(\theta)= 
\left\{
\begin{array}{lcrr}
C(\theta-1)P(\theta-1), \quad   &\theta&\in
[1,+\infty),\\
C(\theta)^{-1}P(\theta+1),  \quad   &\theta&\in
(-\infty,1].
\end{array}
\right.
\end{equation}
By \eqref{eq:B-C} the function $P$ is well defined. Moreover,  
$P\in \mathcal{G}_{cL}^\rho(\R, U(n))$ with some $c\ge 1$ independent of $L\ge 1$ and  it satisfies the relation
\[
\forall\,  \theta \in \R, \quad P(\theta+1)^{-1}C(\theta)P(\theta)=I.
\]
It remains to prove \eqref{gev-norm-est} if \eqref{eq:C} holds. We can choose now $X_0=\log C(0)$.  Moreover, the inequality \eqref{eq:C0} holds for $\theta\in [-2,2]$, $Y$ is well defined by \eqref{eq:Y-definition} in $[-2,2]$ and as in \eqref{y-gev-est} we get
\[
\|X_0\| + \|Y\|^{\mathcal{G}}_{dL,2}\leq \mbox{cst.} \|C-I\|^{\mathcal{G}}_{L,3}
\]
where $d\ge 1$. Choose $\delta=1/2$ in \eqref{eq:P}.
 Writing   $e^X= I + Xg(X)$, where $g$ is an entire function and using \eqref{eq:P} and the theorems about the multiplication and the  composition of Gevrey functions (\cite{P}, Proposition A.3)  we obtain the following estimate  in $\mathcal{G}_{cL}^\rho([-1/2,3/2], M_n)$
\[
\begin{array}{lcrr}
\|P-I\|^{\mathcal{G}}_{cL,2} = \|e^{-b_\delta X_0}- e^{b_\delta Y\circ f_\delta}\|^{\mathcal{G}}_{cL,2} \\[0.3cm]
\le \|b_\delta X_0 g(b_\delta X_0)\|^{\mathcal{G}}_{cL,2} + \|(b_\delta Y\circ f_\delta) g(b_\delta Y\circ f_\delta)\|^{\mathcal{G}}_{cL,2} \\[0.3cm]
\le \mbox{cst.}(\|X_0\| + \|Y\|^{\mathcal{G}}_{dL,2})\leq \mbox{cst.} \|C-I\|^{\mathcal{G}}_{L,3}
\end{array}
\]
with $\delta=1/2$ and a suitable $c>d$. Using \eqref{eq:P-extension} we obtain the estimate \eqref{gev-norm-est} in $\mathcal{G}_{cL}^\rho([-2,2], M_n)$. 
 \end{proof}

When $A^\Phi_{k_1,k_2}$($(k_1,k_2)\in \Z^2$) are all constants, we say that
$\Phi$ is  {\em constant}. The following simple lemma provides
a normalization of  constants.

\begin{Lemma}\label{const}
Any  constant $\{(1,C),(\alpha,D)\}$ can be conjugated to a
normalized constant.
\end{Lemma}

\begin{proof}
As $C$ and $D$ commute, they generate an abelian
Lie subgroup $\mathcal{T}$ of $U(n)$, hence, one can choose $X_0$ in the
Lie algebra of $\mathcal{T}$ satisfying $C=e^{X_0}$ and $X_0D=DX_0$.
Now $\{(1,e^{X_0}),(\alpha,D)\}$ can be conjugated to $\{(1,I),(\alpha,
De^{-\alpha X_0})\}$ by $B(\theta)=e^{\theta X_0}$. \end{proof}

A $\Z^2$-action $\Phi\in \Lambda$ is said to be {\em reducible}
 if it  can be conjugated to a constant by some
$B\in \mathcal{G}^{\rho}(\R,U(n))$. From Lemma \ref{const} one obtains the following

\begin{Lemma}\label{red-red}
Let $(\alpha, A)$ be a cocycle. Then the    $\Z^2$-action  $\{(1,I),(\alpha,A)\}$  is  reducible if and only if $(\alpha, A)$
 is  reducible as a cocycle.  
\end{Lemma}

\subsection{Renormalization}

We recall from \cite{AK,Kr1}  the following operations on  $\Lambda$. 
\begin{itemize}
\item[a)] For any  $\theta\in \R$  a translation $T_{\theta}$ is defined by
\[
T_{\theta} (\Phi) (k_1,k_2)= (\gamma_{k_1,k_2}^\Phi,
 A_{k_1,k_2}^\Phi(\cdot+\theta)).
 \]
\item[b)] For any $\lambda \neq 0$ denote by
  $M_\lambda$ the rescaling  
 \[
 M_\lambda (\Phi) (k_1,k_2) =
(\lambda^{-1}\gamma_{k_1,k_2}^\Phi,
 A_{k_1,k_2}^\Phi(\lambda\cdot)).
 \]
\item[c)] For any $U\in GL(2,\Z)$
 we denote by   $N_U$ the base change  
 \[
 N_U (\Phi)(k_1,k_2) =
 \Phi((k_1,k_2)(U^T)^{-1}).
 \]
\end{itemize}
Reducibility is
invariant under conjugation, translation, rescaling and  base change. \\

Given an   irrational $\alpha \in (0,1)$ we consider the continued fractional
expansion
\begin{eqnarray*}
\alpha= \frac{1}{\displaystyle a_1+ \frac{\displaystyle
1}{\displaystyle a_2+ \cdots}} ,
\end{eqnarray*}
We set $\alpha_0=\alpha$, and
\begin{eqnarray*}
\alpha_m= \frac{1}{\displaystyle a_{m+1}+ \frac{\displaystyle
1}{\displaystyle a_{m+2}+\cdots }} .
\end{eqnarray*}
In fact, $\alpha_m=G^m(\alpha)$ where $G$ is the Gauss map $x\mapsto \{1/x\}$ assigning to each $x\neq 0$ its  fractional part $\{1/x\}$. The
integers $a_m$ are given by $a_m=[\alpha_{m-1}^{-1}]$, where $[\cdot]$
denotes the integer part. We also set $a_0=0$ for convenience.

Let $\beta_m=\prod_{j=0}^m \alpha_j$. Define
\begin{eqnarray*}
Q_0&=& \left [\begin{matrix}
 q_0 & p_0 \cr
 q_{-1} & p_{-1}
 \end{matrix} \right ]
 =\left [\begin{matrix}
 1 & 0 \cr
 0 & 1
 \end{matrix}\right ], \\
Q_m&=& \left [\begin{matrix}
 q_m & p_m \cr
 q_{m-1} & p_{m-1}
 \end{matrix}\right]
 =\left [\begin{matrix}
 a_m & 1 \cr
 1 & 0
 \end{matrix}\right]
 \left [\begin{matrix}
 q_{m-1} & p_{m-1} \cr
 q_{m-2} & p_{m-2}
 \end{matrix}\right].
\end{eqnarray*}
It is easy to  see that  $Q_m=U(\alpha_m)\cdots U(\alpha_1)$ where
\begin{eqnarray*}
U(x)= \left [\begin{matrix}
 [x^{-1}] & 1 \cr
 1 & 0
 \end{matrix}\right ].
\end{eqnarray*}
Thus  $\textrm{det} (Q_m) = q_m p_{m-1}-p_m q_{m-1}=(-1)^{m}$. Note
that
\[
\begin{array}{lcrr}
\displaystyle \beta_m= (-1)^m(q_m \alpha - p_m)=
\frac{1}{q_{m+1}+\alpha_{m+1}q_m} , \\[0.5cm]
\displaystyle \frac{1}{q_m+q_{m+1}}< \beta_m < \frac {1}{q_{m+1}} .
\end{array}
\]
The renormalization operator $R$ is defined by
$R(\Phi)=M_{\alpha^\Phi}N_{U(\alpha)}(\Phi)$. Notice that if
$\Phi\in \Lambda$ satisfies $\gamma^{\Phi}_{1,0}=1$, then
$\alpha^{R(\Phi)}=G(\alpha^{\Phi})$ (recall that $G$ is Gauss
map) and 
\[
R^m(\Phi)=M_{\alpha_{m-1}}\circ N_{U(\alpha_{m-1})}\circ \cdots \circ
M_{\alpha_{0}}\circ N_{U(\alpha_{0})}(\Phi)= M_{\beta_{m-1}} (N_{
Q_m}(\Phi)).
\] 
It is obvious that the reducibility is invariant under
renormalization.

\subsection{Proof of Theorem \ref{Gloabl}}

Let $\Phi=\{(1,I),(\alpha,A)\}$
be a normalized $\Z^2-$action such that $\alpha\in {\rm RDC\, }(\gamma,\tau)$
for some $\gamma,\tau>0$. Suppose that there is  a measurable function $B:\T\rightarrow U(n)$
 satisfying
\begin{equation} \label{eq:B-conjugation}
{\rm Ad\, }(B).(\alpha,A)=(\alpha, C)
\end{equation}
where $C$ is  constant. Denote the spectrum of $C$ by
 \[{\rm spec\,}(C)=\{e^{2\pi i \phi_1},\cdots, e^{2\pi i \phi_n}\}.\]
 By assumption
$C\in \Sigma(\alpha)$ which means that
$\phi:=(\phi_1,\cdots,\phi_n)\in \Upsilon(\alpha)$
 (for the definition $\Upsilon(\alpha)$ we refer to (\ref{upsilon-def})).\\

The following fact of  $B$ will be needed in the proof of Theorem \ref{Gloabl}.

\begin{Lemma}\label{integral estimates}
For a.e. $\theta_0\in\T$, we have
\begin{eqnarray*}
\lim_{t\rightarrow 0+} \int_0^1|B(\theta_0+t\theta)B(\theta_0)^*-I|d\theta =0.
\end{eqnarray*}
\end{Lemma}
\begin{proof}
Denote by $X$ the set of  measurable
continuity points   of $B$ and $B^*$. By the Lebesgue density theorem,
$X$ has full Lebesgue measure.

 Fix $\theta_0\in X$. For any $\epsilon>0$, let
\[I_t(\epsilon)=\{\theta\in [\theta_0,\theta_0+t]:
|B(\theta)B(\theta_0)^*-I|<\epsilon\}\] and
$J_t(\epsilon)=[0,t]-I_t(\epsilon)$. Note that
\begin{eqnarray*}
\lim_{t\rightarrow 0+}\frac{{\rm Leb.}(J_t(\epsilon))}{t}=0,
\end{eqnarray*}
where $Leb.$ denotes Lebesgue measure. Now we have
\begin{eqnarray*}
&&\int_{[0, 1]}|B(\theta_0+t \theta)B(\theta_0)^*-I|d\theta\\
&&=\frac{1}{t}\int_{[\theta_0,\theta_0+t]}|B(\theta)B(\theta_0)^*-I|d\theta \\
&&=\frac{1}{t}\left(\int_{I_t(\epsilon)}+\int_{J_t(\epsilon)}\right)|B(\theta)B(\theta_0)^*-I|d\theta\\
&&\leq \epsilon+\frac{{\rm Leb.}(J_t(\epsilon))}{t}\rightarrow \epsilon
\end{eqnarray*}
as $t\rightarrow 0+$, where ${\rm Leb.}$ stands for the Lebesgue measure. Since the inequality holds for any
$\epsilon>0$,
 the lemma is proved. \end{proof}

Without loss of generality, we
assume that the conclusion of Lemma \ref{integral estimates} holds for $\theta_0=0$ (if  not  we  make a translation). Consider
\[
R^m(\Phi)
=\left\{(1, A_{s_m}(\beta_{m-1}\cdot)),(\alpha_m,
A_{\tilde s_m}(\beta_{m-1}\cdot))\right\}, \quad m=1,2,\cdots 
\]
where 
\[
s_m:= (-1)^{m-1}q_{m-1} \quad \mbox{and} \quad \tilde s_m:= (-1)^{m}q_{m}.
\] 
We need the following Lemma.

\begin{Lemma}\label{Priori Bounds} There exists a constant $c>0$ such that for any 
 $\alpha$ irrational  and  that $A\in
\mathcal{G}_L^\rho (\T, U(n))$ the sequences 
$A_{s_m}(\beta_{m-1}\cdot)$ and $A_{\tilde s_m}(\beta_{m-1}\cdot)$, $m\ge 1$,   are uniformly
bounded in 
$\mathcal{G}_{\widetilde{L}} ^\rho (\R, U(n))$, where $\widetilde L=L/c$.
\end{Lemma}

\begin{proof} Recall from \cite{AK, HY2} the following estimate
\[
\forall\, m,r\in\{1,2,\cdots\},\ \forall\, \theta\in \T : \quad
\|\partial^r A_m(\theta)\| \leq m^r c^r ||\partial^rA\|_0,
\]
where $\|\cdot\|_{0}$ is the $C^0$ norm and $c>0$ is a
constant.
Applying it  to both  $(\alpha, A)$ and $(\alpha, A)^{-1}$ we
obtain
\[||\partial^rA_k||_0\leq |k|^r c^r ||\partial^rA||_0\]
 for all $0\neq k\in
\Z$. It follows that for all $\theta\in \R$ and $m\ge 1$
\[||\partial^rA_{(-1)^{m-1}q_{m-1}}(\beta_{m-1}\theta)|| \leq  |\beta_{m-1}q_{m-1}|^r  c^r
||\partial^rA||_0\leq c^r ||\partial^rA||_0;\]
\[||\partial^rA_{(-1)^{m}q_{m}}(\beta_{m-1}\theta)|| \leq  |\beta_{m-1}q_{m}|^r  c^r
||\partial^rA||_0\leq c^r ||\partial^rA||_0\]
By assumption  $A\in \mathcal{G}_L^\rho(\T, U(n))$, thus there
exists $M>0$ such that
\[\sup_{r\in  \N}\sup_{\theta\in \R}
(\|\partial^{r}A(\theta)\|L^r r!^{-\rho})\leq M.\]
This implies
\[
\sup_{r\in  \N}\sup_{\theta\in \R}
(\|\partial^{r}A_{(-1)^{m-1}q_{m-1}}(\beta_{m-1}\theta)\|\left(\frac{L}{c}\right)^r r!^{-\rho})\leq M, 
\]
\[ 
\sup_{r\in  \N}\sup_{\theta\in \R}
(\|\partial^{r}A_{(-1)^{m}q_{m}}(\beta_{m-1}\theta)\|\left(\frac{L}{c}\right)^r r!^{-\rho})\leq M
\]
which completes the proof. \end{proof}

Set $H(\cdot):=e^{h(\cdot)}$, where
$h(\varphi_1,\cdots,\varphi_n):=2\pi i\, {\rm diag\, }(\varphi_1,\cdots,\varphi_n).$
 Choose 
 $S\in U(n)$ so that
\begin{equation} \label{eq:C-conjugation}
C=S^* H(\phi) S.
\end{equation}
The following Lemma give us  information about the limit points of
the set of  $\Z^2-$actions $\{R^m(\Phi): m=1,2,\cdots\}$.

\begin{Lemma}\label{convertoconst}
There exist   sequences $(m_j)\subset \N$ and $(l_j),\, (\widetilde{l}_j)\subset \Z^n$ and  $\psi,\widetilde{\psi}\in [0,1]^n$ such that
$\alpha_{m_j}\in {\rm DC\, }(\gamma,\tau)$ and
 \[\alpha_{\infty}:=\lim_{j\rightarrow \infty}\alpha_{m_j}\in {\rm DC\, }(\gamma,\tau);\]
\[ \lim_{j\rightarrow \infty}(s_j\phi+l_j)= \psi;\]
\[ \lim_{j\rightarrow \infty}(\widetilde{s}_j\phi+\widetilde{l}_j)=\widetilde{\psi},\]
where $s_j=(-1)^{m_j-1}q_{m_j-1}$ and
$\widetilde{s}_j=(-1)^{m_j}q_{m_j}$, and the sequences of functions $A_{s_j}(\beta_{m_j-1}\cdot)$ and
$A_{\widetilde{s}_j}(\beta_{m_j}\cdot)$
 converge  in $\mathcal{G}_N^\rho(\R, U(n))$ to  $B(0)^*S^*H(\psi)SB(0)$ and  $B(0)^*S^*H(\widetilde{\psi})SB(0)$
 respectively, for some fixed $N>\widetilde
L$.
\end{Lemma}

\begin{proof} By Lemma $\ref{Priori Bounds}$, there exist
$\widetilde{L}>0$, such that $U_m$ and $V_m$  are uniformly bounded
in $\mathcal{G}_{\widetilde{L}}^\rho(\R, U(n))$. Fix $N>\widetilde
L$. Then for any $K>0$ the inclusion
\[
\mathcal{G}_{\widetilde{L}}^\rho([-K,K], U(n))\ \hookrightarrow \
\mathcal{G}_N^\rho([-K,K], U(n))
\] 
is compact (see \cite{L-M}, Ch. 7) and by the diagonal procedure one can find a subsequence of $m_j$ such that the sequences 
\[
U_j:=A_{s_j}(\beta_{m_j-1}\cdot)
\]
and 
\[
\widetilde{U}_j:= A_{\widetilde{s}_j}(\beta_{m_j}\cdot)
\]
 converge in $\mathcal{G}_{N}^\rho(\R, U(n))$ to some $U_\infty$ and $\widetilde{U}_\infty$, respectively. Without loss of generality we
assume that 
\[
H(s_j \phi)\to H(\psi),\quad
H(\widetilde{s}_j \phi)\to H(\widetilde{\psi})
\] 
for some $\psi,\widetilde{\psi}\in [0,1]^n$
and $\alpha_{m_j}\to \alpha_{\infty}$ for some $\alpha_{\infty}\in {\rm DC\, }(\gamma,\tau)$ 
(otherwise we choose  subsequences).

   Recall that 
$0$ is a measurable continuity point of both $B$ and $B^*$. Thus for any
$\varepsilon>0$ and  $d>1$ fixed 
\begin{eqnarray*}
\lim_{j\rightarrow \infty}\, \frac{1}{2d\beta_{m_j} }\, \mbox{Leb.}\left(I(0,\varepsilon)\cap [-\beta_{m_j}
d, \beta_{m_j} d]\right)=1,
\end{eqnarray*}
where Leb. is the Lebesgue measure and 
\[
I(0,\varepsilon):=\{\vartheta\in \R: \|B(\vartheta)-B(0)\|<\varepsilon,
\|B(\vartheta)^*-B(0)^*\|< \varepsilon\}.
\]
By Lemma $\ref{Priori Bounds}$ the functions  $U_{j}$ are $C^1$-uniformly bounded on $\R$, hence, they are  
equicontinuous and  there exists a positive $\widehat{\delta}$ such that 
\[
\|U_{j}(\theta)-U_{j}(\vartheta)\|<\varepsilon
\] 
for all $j$ as long as $|\theta-\vartheta|<\widehat{\delta}$. Then for any 
$j\gg 1$  and 
$\theta\in [-d+1, d-1]$ there exists $\vartheta\in
[-d+1,d-1]\cap (\theta-\widehat{\delta}, \theta+\widehat{\delta})$  such that
\[
\beta_{m_j} \vartheta, \quad \beta_{m_j}( \vartheta+1)\in I(0,\varepsilon)
\]
and we obtain
\begin{eqnarray}\label{3-small-cond}
\max\{\|U_{j}(\theta)-U_{j}(\vartheta)\|,\,
\|B(\beta_{m_j} \vartheta)-B(0)\|,\,
\|B(\beta_{m_j}(\vartheta+1))-B(0)\|\}<\varepsilon.
\end{eqnarray} 
Using the conjugations \eqref{eq:B-conjugation} and \eqref{eq:C-conjugation} we get  
\begin{eqnarray}\label{ren-conj}
B(\beta_{m_j} (\vartheta+1))^{-1}U_{j}(\vartheta)B(\beta_{m_j}
\vartheta)=S^*H(s_j \phi)S, 
\end{eqnarray}
then by  \eqref{3-small-cond} and \eqref{ren-conj} we obtain that
\begin{eqnarray*}
\|B(0)^*\widetilde{U}_{j}(\theta)B(0)-S^*H(s_j
\phi)S\|<3\varepsilon.
\end{eqnarray*}
Recall that $U_{j}$ converge in $\mathcal{G}_{N}^\rho(\R, U(n))$ to  $U_\infty$ and $H(s_j \phi) \rightarrow H(\psi)$. Then  for any $\varepsilon>0$ and  $d>0$ fixed and $\theta\in [-d+1,d-1]$ we have
\begin{eqnarray*}
\|B(0)^*U_\infty(\theta)B(0)-S^*H(\psi)S\|<3\varepsilon, 
\end{eqnarray*}
hence,  
\[
 B(0)^*U_\infty
B(0)\equiv S^*H(\psi)S. 
\]
In the same way we get 
\[
 B(0)^*\widetilde{U}_\infty B(0)\equiv S^*H(\widetilde{\psi})S
\] 
which completes the proof. \end{proof}

We have proved that $R^m(\Phi)$ has a subsequence
\begin{eqnarray*}
\Phi_j:=R^{m_j}(\Phi)&=&\{(1, A_{s_j}(\beta_{m_j-1}\cdot)),(\alpha_{m_j},
A_{\tilde s_j}(\beta_{m_j-1}\cdot))\}\\
&=:&\{(1, U_{j}),(\alpha_{m_j},
\widetilde{U}_{j})\}
\end{eqnarray*}
with $s_j=(-1)^{m_j-1}q_{m_j-1}$ and $\tilde s_j=(-1)^{m_j}q_{m_j}$, 
converging to the constant $\Z^2-$action
\[\Phi_\infty:=\{(1,B(0)S^*H(\psi)SB(0)^*),(\alpha_\infty,
B(0)S^*H(\widetilde{\psi})SB(0)^*)\}.\] We want to show that after a conjugation
this subsequence will become a sequence of normalized $\Z^2-$actions
converging in $\mathcal{G}_{\widetilde{L}}^\rho(\R, U(n))$ to a
constant normalized $\Z^2-$action. Set $Q(\theta ):=S^*H(  \theta \psi )S$.

\begin{Lemma}\label{normform}
There is  a sequence  $P_j$ in
$\mathcal{G}_{N}^\rho(\R, U(n))$ converging  to  $I$ in
$\mathcal{G}_N^\rho([-2,2], U(n))$,  such that
\begin{eqnarray}
\Psi_j:={\rm Ad\, }\left(B(0)QP_j\right).\Phi_j= \left\{(1,I),(\alpha_{m_j},W_j)\right\} .
\end{eqnarray}
Moreover,  $W_j\in \mathcal{G}_{N}^\rho(\T, U(n))$ converges to
$S^*H(\widetilde{\psi}- \alpha_{\infty}\psi)S$.
\end{Lemma}

\begin{proof}
Set $\widetilde{\Phi}_j:={\rm Ad\,}(B(0)Q).\Phi_j=Ad(B(0)Q).R^{m_j}(\Phi).$ 
It follows from the definition of  $\Phi_\infty$ and $Q$  that
\[{\rm Ad \,}(B(0)Q).\Phi_\infty=\{(1,I),(\alpha_\infty,S^*H(\widetilde{\psi}-\alpha_\infty \psi)S)\}.\]
Thus $\widetilde{\Phi}_j(1,0)$ and
$\widetilde{\Phi}_j(0,1)$ converge uniformly in $\mathcal{G}_{N} ^\rho (\R, U(n))$ to $(1,I)$ and $(\alpha_\infty, S^*H(\psi-\alpha_\infty
\psi)S)$ respectively. 
By Lemma $\ref{norm}$, there exist a sequence of $\,P_j\in
\mathcal{G}_{N}^\rho(\R, U(n))$, satisfying
\[\lim_{j\rightarrow \infty}\|P_j-I\|_{N,2}^{\mathcal{G}}=0,\]
 such
that ${\rm Ad \,}(P_j).\widetilde{\Phi}_j$ is a sequence of normalized
$\Z^2$-action $\{(1,I),(\alpha_{m_j},W_j)\}$. Thus $W_j\in
\mathcal{G}_{N}^\rho(\T, U(n))$ ($W_j$ is automatically
$\Z-$periodic thanks to the commutation) and it converges to
$S^*H(\widetilde{\psi}- \alpha_{\infty}\psi)S$, i.e.,
\[\lim_{j\rightarrow \infty}\|W_j-S^*H(\widetilde{\psi}- \alpha_{\infty}\psi)S\|^{\mathcal{G}}_{N}=0. \]
\end{proof}

On the other hand,  $(\alpha_{m_j},W_j)$ can be
conjugated to the constant
\[
(\alpha_{m_j}, S^*H((-1)^{m_j}\beta_{m_j-1}^{-1}\phi)S) \
\] by a
 measurable conjugation. To put it more precisely, let us consider the following sequences
\[E_j(\theta)=S^*H((-1)^{m_j-1}q_{m_j-1} \phi \theta)S, \]
\[T_j(\theta)=P_j(\theta)^{-1}Q(\theta)^{-1}B(0)^{*}, \] 
and \[ G_j(\theta)=T_j(\theta)B(\beta_{m_j-1}
\theta)E_j(\theta),\]
where $Q=S^*H(\psi\theta)S$ is introduced  in Lemma \ref{normform}). 
\begin{Lemma}\label{measconj}
The measurable functions $G_j$ are all
$\Z-$periodic and satisfy
\[\liminf_{j\rightarrow
+\infty}[G_j]>\frac{1}{2}n^{-3/2},\]
\[
 {\rm Ad\, }(G_j).(\alpha_{m_j},
W_j)=(\alpha_{m_j}, S^*H((-1)^{m_j-1}\beta_{m_j-1}^{-1}\phi)S).\]
\end{Lemma}

\begin{proof}
One can easily check  that
\begin{eqnarray*}
{\rm Ad \, }(G_j).\Psi_j&=&{\rm Ad \, }(B(\beta_{m_j-1}
\theta)E_j(\theta)).R^{m_j}(\Phi)\\
&=&\{(1,I), (\alpha_{m_j},
S^*H((-1)^{m_j}\beta_{m_j-1}^{-1}\phi)S)\}.    
\end{eqnarray*}
In particular, any  $G_j$ is $\Z-$periodic.
 
Let 
\begin{eqnarray*}
\widetilde{G}_j(\theta)&=&G_j(\theta)S^*H(l_j \theta)S\\
&=&P_j(\theta)^{-1}S^*H(-\psi\theta)SB(0)^*B(\beta_{m_j-1}\theta)S^*H((s_j \phi+l_j )\theta)S 
\end{eqnarray*}
 ($l_j$ was given in  Lemma \ref{convertoconst}). Then   for any $\theta\in \T$ we
have the estimates
\begin{eqnarray*}
||\widetilde{G}_j(\theta)-I|| &\leq & 
||P_j(\theta)^{-1}-I||\\
&+&||B(0)^*B(\beta_{m_j-1}\theta)-I||\\
&+&||H((-\psi+s_j\phi+l_j)\theta)-I||
\end{eqnarray*}
where  $s_j=(-1)^{m_j}q_{m_j-1}$. 
By Lemma \ref{convertoconst} and  \ref{normform}, we have  $s_j\phi+l_j\rightarrow \psi$ and
$P_j(\theta)\rightarrow I$ uniformly on $[-2,2]$. Moreover, by
Lemma $\ref{integral estimates}$
\begin{eqnarray*}
\lim_{j\rightarrow \infty}\int_0^1||B(0)^*B(\beta_{m_j}\theta)-I||d\theta
=0.
\end{eqnarray*}
Thus we have
\begin{eqnarray*}
\lim_{j\rightarrow \infty}\int_0^1||G_j(\theta)^*-S^*H(-l_j\theta)S||d\theta
&=&\lim_{j\rightarrow \infty}\int_0^1||\widetilde{G}_j(\theta)^*-I||d\theta\\
&=&\lim_{j\rightarrow \infty}\int_0^1||\widetilde{G}_j(\theta)-I||d\theta=0.
\end{eqnarray*}
By Lemma $\ref{Lemma:main-norm}$,  $\lceil H(-l_j\theta)\rfloor = \lceil S^*H(-l_j\theta)S\rfloor\geq n^{-3/2}$. Then by Lemma $\ref{Uniform Bound}$, $\lceil G_j^*\rfloor>\frac{1}{2} n^{-3/2}$
as $j$ is large enough. 
\end{proof}

We are ready to complete the  proof of  Theorem $\ref{Gloabl}$.  We would like to apply  the local result Theorem $\ref{Local}$. To
this end, we need the following fact. 

\begin{Lemma}\label{Renormalized Condition}
Let $\phi =(\phi_1,\cdots, \phi_n)\in \Upsilon(\alpha)$, where
 $\Upsilon(\alpha)$ is defined in
(\ref{upsilon-def}).  Then for all $m\in \N$ one has
$(-1)^m\beta_{m-1}^{-1}\phi\in \Upsilon(\alpha_m)$.
\end{Lemma}

\begin{proof}
Without loss of generality, we assume
\[|\phi_t-\phi_{\widetilde{t}}|\leq 2,\quad t\neq \widetilde{t}\in \{1,\cdots, n\}.\]
We consider two cases. If  $|l|> 2|k|+3$ and  $t\neq \widetilde{t}\in\{1,\cdots,n\}$ we have
\begin{eqnarray*}
|k\alpha-(-1)^m\beta_{m-1}^{-1}(\phi_t-\phi_{\widetilde{t}})-l|\geq |l|-|k|-2 \geq
2|k|+3-|k|-2=|k|+1.
\end{eqnarray*}
Let  $|l|\le 2|k|+3$ and $t\neq \widetilde{t}\in\{1,\cdots,n\}$. 
There exist $\sigma,\nu>0$, such that for any $k,l\in \Z$ and
$t\neq \widetilde{t}\in \{1,\cdots, n\}$
\begin{eqnarray*}
|k\alpha-(\phi_t-\phi_{\widetilde{t}})-l|\geq \frac{\sigma}{(1+|k|)^\nu}
\end{eqnarray*}
and then
\begin{eqnarray*}
&&|k\alpha_m-(-1)^m\beta_{m-1}^{-1}(\phi_t-\phi_{\widetilde{t}})-l|\\
&&=\beta_{m-1}^{-1}|(-1)^m k\beta_{m}-(\phi_t-\phi_{\widetilde{t}})-(-1)^ml\beta_{m-1}|\\
&&=\beta_{m-1}^{-1}|(q_{m} \alpha-p_{m})k-(\phi_t-\phi_{\widetilde{t}})+(q_{m-1} \alpha-p_{m-1})l|\\
&&=\beta_{m-1}^{-1}|(q_m k+q_{m-1}
l)\alpha-(\phi_t-\phi_{\widetilde{t}})-(p_m+p_{m-1})|\\
 &&\geq\frac{\beta_{m-1}^{-1}\sigma}{(1+q_m |k|+q_{m-1}|l|)^\nu}.
\end{eqnarray*}
Since $|l|\le 2|k|+3$ and  $t\neq \widetilde{t}\in\{1,\cdots,n\}$ we have
\begin{eqnarray*}
&&|k\alpha_m-(-1)^m\beta_{m-1}^{-1}(\phi_t-\phi_{\widetilde{t}})-l|\\
 &&\geq\frac{\beta_{m-1}^{-1}\sigma}{(1+q_m |k|+q_{m-1}|l|)^\nu}\\
 &&\geq\frac{\beta_{m-1}^{-1}\sigma}{(1+q_m |k|+2q_{m}|k|+3q_{m})^\nu}\\
 &&\geq\frac{\beta_{m-1}^{-1}\sigma}{(4q_m)^\nu (|k|+ 1)^\nu}.
\end{eqnarray*}

So there exists $\sigma_m>0$, such that for any $k,l\in \Z$ and
$t\neq \widetilde{t}\in\{1,\cdots,n\}$
\begin{eqnarray*}
|k\alpha-(-1)^m\beta_{m-1}^{-1}(\phi_t-\phi_{\widetilde{t}})-l|\geq \frac{\sigma_m}{(1+|k|)^\nu}.
\end{eqnarray*} 
\end{proof}

Now we go back to the proof of Theorem $\ref{Gloabl}$. Lemma \ref{measconj} , \ref{Renormalized Condition} allows us to
apply Theorem $\ref{Local}$.  For $j$ is sufficiently large, there
exists $B_j$ in   $\mathcal{G}^\rho(\T,U(n))$ satisfying \[B_j(\theta
)=G_j(\theta)\] for a.e. $\theta \in
\T$ . Since $E_j$ and $T_j$ are analytic,   
 \[\widetilde{B}_j(\theta
):=T_j(\beta_{m_j-1}^{-1}\theta)^{-1}B_j(\beta_{m_j-1}^{-1}\theta)E_j(\beta_{m_j-1}^{-1}\theta)^{-1}\]
is in  $\mathcal{G}^\rho(\R,U(n))$. Moreover, 
\begin{eqnarray*}
\widetilde{B}_j(\theta
)=T_j(\beta_{m_j-1}^{-1}\theta)^{-1}G_j(\beta_{m_j-1}^{-1}\theta)E_j(\beta_{m_j-1}^{-1}\theta)^{-1}=B(\theta )
\end{eqnarray*}
for a.e. $\theta
\in \R$. Recall that  $B$ is $\Z-$periodic for  a.e. $\theta
\in \R$, so $\widetilde{B}_j$ is also  $\Z-$periodic  for  a.e. $\theta
\in \R$ and  it is   then $\Z-$periodic  for  all $\theta
\in \R$ (thanks to the continuity of $\widetilde{B}_j$). 

By assumption, for a.e.
$\theta \in \T$,
\[B(\theta+\alpha)^{-1}A(\theta)
B(\theta)= C,\]  hence, for a.e.
$\theta \in \T$,
\[\widetilde{B}_j(\theta +\alpha)^{-1}A(\theta ) \widetilde{B}_j(\theta
)=C,\] which
implies that for all $\theta \in \T$ 
\[\widetilde{B}_j(\theta +\alpha)^{-1}A (\theta )
\widetilde{B}_j(\theta )=C\]  
(thanks to the continuity of $\widetilde{B}_j$). This completes  the proof of Theorem $\ref{Gloabl}$. $\Box$\\

Xuanji Hou\\
 School of Mathematics and
Statistics,
 Central China Normal University, \\
Wuhan 430079, P. R.
China. \\
Email: hxj@mail.ccnu.edu.cn\\

 Georgi Popov\\
 University de Nantes, Laboratoire de mathematiques Jean Leray,\\
CNRS: UMR 6629, France, \\
Email: georgi.popov@univ-nantes.fr


\begin{thebibliography}{99}
\bibitem{AK}
\newblock A. Avila and R. Krikorian,
\newblock \emph{Reducibility or non-uniform hyperbolicity for quasiperiodic
Schr\"{o}dinger cocycles},
\newblock  Annals of Mathematics 164 (2006).
911-940.


\bibitem{C1}
\newblock C. Chavaudret,
\newblock \emph{Strong almost reducibility for analytic and
Gevrey quasi-periodic cocycles},
\newblock  Bull. Soc. Math. France 141 (2013), no. 1, 47-106


\bibitem{E}
\newblock H. Eliasson,
 \newblock \emph{Floquet solutions for the 1-dimensional
quasi-periodic Schr\"{o}dinger equation},
\newblock  Comm. Math. Phys. 146
(1992). NO.3. 447-482.

\bibitem{El01}
     \newblock H. Eliasson, 
     \newblock \emph{Almost reducibility of linear quasi-periodic systems},
      \newblock  in ``Smooth ergodic theory and its applications" (Seattle, WA.
      1999),  679--705, Proc. Sympos. Pure Math., 69, Amer. Math. Soc., Providence, RI, 2001.
             MR1858550 (2003a:34064)

\bibitem{HeY}
\newblock H. Her and J. You,
 \newblock \emph{Full Measure Reducibility for Generic One-Parameter Family of
Quasi-Periodic Linear Systems},
\newblock Journal of Dynamics and Differential Equations,
20 (2008), no. 4 (2008), 831-866.

\bibitem{HY1}
\newblock X. Hou and J. You,
     \newblock \emph{The  Local Rigidity of Reducibility of Analytic
Quasi-periodic Cocycles on $U(N)$},
      \newblock Discrete Contin. Dyn. Syst. 24 (2009), no. 2, 441--454.

\bibitem{HY2}
\newblock X. Hou and J. You,
     \newblock \emph{The rigidity of reducibility of cocycles on $SO(N,\mathbb{R})$},
      \newblock Nonlinearity 21 (2008), no. 10, 2317--2330.

\bibitem{HY3}
\newblock X. Hou and J. You,
     \newblock \emph{Almost reducibility and non-perturbative reducibility of quasi-periodic linear systems},
      \newblock Inventiones Mathematicae 190 (2012), no. 1, 209--260.

\bibitem{Kr99a}
\newblock R. Krikorian,
\newblock \emph{R\'{e}ductibilit\'{e} presque partout des
flots fibr\'{e}s quasi-p\'{e}riodiques \`{a} valeurs dans les
groupes compacts},
\newblock Ann. Sci. \'{e}c. Norm. Super. 32, 187-240 (1999) 22.
269-326.

\bibitem{Kr99b}
\newblock R. Krikorian,
\newblock \emph{R\'{e}ductibilit\'{e} des syst\`{e}mes produits-crois\'{e}s \`{a}
valeurs das des groupes compacts},
\newblock Ast\'{e}risque, vol. 259 (1999).


\bibitem{Kr1}
\newblock R. Krikorian,
\newblock \emph{Global density of reducible quasi-periodic cocycles
on $\T^1\times SU(2)$},
\newblock Annals of Mathematics (2)154 (2001). no 2.
269-326.

\bibitem{Kr2}
\newblock R. Krikorian,
\newblock  \emph{Reducibility, differentiable rigidity and Lyapunov
exponents for quasi-periodic cocycles on $\T\times SL(2,\R)$},
\newblock Preprint (www.arXiv.org).


\bibitem{L-M}
\newblock J.-L. Lions and E. Magenes,
\newblock \emph{ Problemes aux Limites Non
homog\`enes et Applications}, (French),  Vol. 3. Travaux et recherches
math\'ematiques  20,
\newblock Dunod, Paris,   1970. 

\bibitem{M-S}  J.-P. Marco and D. Sauzin,   \textit{Stability and instability for
Gevrey quasi-convex near-integrable Hamiltonian systems},  
{Publ. Math. IH\'ES},  {\bf 96} (2002),  199-275.  

\bibitem{M-P} 
\newblock T. Mitev and G. Popov,  
\newblock \emph{Gevrey Normal Form and Effective Stability of Lagrangian Tori}, 
\newblock Discrete and Continuous Dynamical Systems - Series S 3 (2010) 643-666

\bibitem {P1} G. Popov,  
\textit{Invariant tori, effective stability, and quasimodes
with exponentially small error terms I - Birkhoff Normal Forms,} 
Ann.   Henri Poincar\'e, \textbf{1} (2000), 223-248. 

\bibitem{P}
\newblock G. Popov,
 \newblock \emph{KAM theorem for Gevrey Hamiltonians},
\newblock  Ergod. Th. and  Dynam. Sys.  24
(2004). 1753-1786.
\end{thebibliography}
\end{document}